\documentclass[11pt]{amsart}

\usepackage{geometry}
\usepackage{amsmath,amssymb,amsthm,mathtools,mathrsfs}
\usepackage[colorlinks,citecolor=red,linkcolor=blue]{hyperref}
\usepackage{bookmark}
\usepackage{tikz}
\usepackage{enumitem}
\usetikzlibrary{calc}

\newtheorem{theorem}{Theorem}[section]
\newtheorem{lemma}[theorem]{Lemma}
\newtheorem{proposition}[theorem]{Proposition}
\newtheorem{corollary}[theorem]{Corollary}
\theoremstyle{definition}

\theoremstyle{remark}
\newtheorem{remark}[theorem]{Remark}
\numberwithin{equation}{section}

\newcommand{\D}{\mathbb{D}}
\newcommand{\T}{\mathbb{T}}
\newcommand{\C}{\mathbb{C}}
\newcommand{\R}{\mathbb{R}}
\newcommand{\Z}{\mathbb{Z}}
\newcommand{\A}{A^2_\omega}
\newcommand{\supp}{\operatorname{supp}}
\newcommand{\norm}[1]{\left\lVert #1\right\rVert}
\newcommand{\abs}[1]{\left|#1\right|}
\newcommand{\set}[1]{\left\{#1\right\}}
\newcommand{\ip}[2]{\left\langle #1,#2\right\rangle_{\A}}
\newcommand{\dd}{\,\mathrm d}

\setlist[itemize,enumerate]{
  topsep=4pt,
  itemsep=1.5pt,
  parsep=0pt,
  partopsep=0pt
}

\AtBeginDocument{%
  \setlength{\abovedisplayskip}{6pt plus 2pt minus 2pt}%
  \setlength{\belowdisplayskip}{6pt plus 2pt minus 2pt}%
  \setlength{\abovedisplayshortskip}{3pt plus 2pt minus 1pt}%
  \setlength{\belowdisplayshortskip}{4pt plus 2pt minus 1pt}%
  \setlength{\jot}{2pt}%
  \setlength{\textfloatsep}{10pt plus 2pt minus 2pt}%
  \setlength{\floatsep}{8pt plus 2pt minus 2pt}%
  \setlength{\intextsep}{8pt plus 2pt minus 2pt}%
  \setlength{\abovecaptionskip}{5pt}%
  \setlength{\belowcaptionskip}{0pt}%
}

\begin{document}

\title[Compact Toeplitz operators]
{Compact Toeplitz operators via the Berezin transform\\ on radial weighted Bergman spaces}

\author{Yuerang Li}
\address{College of Mathematics and Statistics, Chongqing University,
Chongqing 401331, China}
\email{yuerangli@outlook.com}

\author{Zipeng Wang}
\address{College of Mathematics and Statistics, Chongqing University,
Chongqing 401331, China}
\email{zipengwang2012@gmail.com,zipengwang@cqu.edu.cn}


\begin{abstract}
Let $\omega$ be a radial $\widehat{\mathcal D}$-weight and $u$ be a bounded function on the unit disk $\mathbb D$. We prove that the Toeplitz operator \(T_{\omega,u}\) is compact on \(A_\omega^2\) if and only if its Berezin transform vanishes at the boundary.
Our approach is based on a polynomial frame for $A_\omega^2$ and a detailed localization analysis of the resulting infinite matrix representation of $T_{\omega,u}$. Even in the unweighted Bergman space \(A^2\), our argument is new and does not rely on the classical translation operators.

We further show that this Axler--Zheng compactness characterization does not extend, in general, to products of Toeplitz operators on \(\widehat{\mathcal D}\)-weighted Bergman spaces, and hence to the corresponding Toeplitz algebra generated by bounded symbols. More precisely, we construct a radial log-subharmonic \(\widehat{\mathcal D}\)-weight \(\omega\) and bounded symbols \(u,v\) such that the product \(T_{\omega,v}T_{\omega,u}\) is noncompact, whereas its Berezin transform vanishes at the boundary. 
\end{abstract}

\maketitle

\section{Introduction}\label{sec:introduction}
\subsection{Motivations}

The point of departure for the present work is the classical theorem of Axler and Zheng~\cite{AZ98}, which characterizes compact Toeplitz operators with bounded symbols on the Bergman space $A^2$ via the Berezin condition near the boundary. Beyond the characterization itself, their proof introduced an influential mechanism based on a family of unitary operators induced by automorphisms of the disk, also referred to as ``translation'' operators.

The Axler-Zheng theorem may be viewed as a reproducing kernel thesis for compactness: it characterizes compactness by testing an operator on normalized reproducing kernels. Mitkovski and Wick~\cite{MW14} subsequently established a general framework for this principle that applies to a broad range of Bergman-type spaces. One of their assumptions, condition A.5, facilitates the construction and control of the associated translation operators on the Bergman-type spaces. Isralowitz, Mitkovski, and Wick~\cite{IMW15} further developed new methods for operators on Bergman and Fock spaces. A notable feature of their approach is that translation operators are not used in the case of $A^2$.  This leads naturally to the question of whether translation operators can be avoided, particularly on weighted Bergman spaces for which condition A.5 may be unavailable (see \cite[Section 4]{IMW15} for further discussion). The main result of this paper gives an affirmative answer for the weighted Bergman space $A_\omega^2$ induced by radial weights in \(\widehat{\mathcal D}\) satisfying a one-sided doubling condition. In particular, all classical standard weights are $\widehat{\mathcal{D}}$ weights.

A second source of inspiration for this paper is Xia and Zheng's kernel-localization method~\cite{XZ13} for the Fock space. Their method is a frame-theoretic approach. The normalized reproducing kernels are used to obtain a localized frame representation. Moreover, the decay of the inner products reflects the underlying geometry of the complex plane. It is natural to expect an analogue on the unit disk in which Euclidean separation is replaced by the hyperbolic one. The construction used here may be regarded as a realization of this general idea, although its technical form is substantially different. Since a general radial weighted Bergman space does not have the translation symmetry available in the Fock setting, we use a discrete frame consisting of polynomials which is a careful truncation of the normalized reproducing kernels.

Our treatment of Toeplitz operators on weighted Bergman spaces is built upon the systematic study of $\widehat{\mathcal D}$-weights; see \cite{PR14,PR-embed,Pel16,PR21,PRWW26}. Their work provides a substantial foundation of related weighted Bergman spaces, which includes sharp estimates for reproducing kernels, characterizations of Carleson measures, descriptions of dual spaces, and Littlewood--Paley formulas. These tools have led to a clear understanding of several operator-theoretic questions, such as boundedness, compactness, and membership in Schatten classes, for Toeplitz operators induced by positive measures; see, for example, \cite{PRS18}. Then, it is natural to consider Toeplitz operators with complex-valued symbols on the weighted Bergman space. In this work, we take bounded symbols as a first step toward developing this theory.

\subsection{Main results}
Any nonnegative integrable function $\omega$ on $[0,1)$ induces a weight on $\D$ by setting $\omega(z)=\omega(|z|)$. Let $L^2_\omega$ be the weighted Lebesgue space on the unit disk consisting of all measurable functions
\[
\|f\|_{L^2_\omega}=\bigg(\int_\mathbb{D}|f(z)|^2\omega(z)dA(z)\bigg)^{\frac{1}{2}}<\infty,
\]
where $dA(z)=dxdy/\pi$ is normalized area measure on the unit disk $\D$. The weighted Bergman space $A^2_\omega$ is the subspace of $L^2_\omega$:
\[
 A^2_\omega
 =\set{f\in L^2_\omega:f\text{ is holomorphic on }\D}.
\]

When $\omega\equiv 1$, $A_\omega^2$ reduces to the classical Bergman space $A^2$. Moreover, we observe that $A^2_\omega$ is a reproducing kernel Hilbert space if and only if $\int_r^1 \omega(s)ds>0$ for any $r\in(0,1)$.

Hence, throughout the paper, by a weighted Bergman space we mean a reproducing kernel Hilbert space $A^2_\omega$ induced by a radial weight. The reproducing kernel $B_z^\omega$ of $A^2_\omega$, called the Bergman kernel, then has the series representation
\begin{equation*}\label{eq:kernel-series}
 B_z^\omega(\zeta)
 =\sum_{n=0}^\infty\frac{(\overline z\zeta)^n}{2\omega_{2n+1}},
 \qquad z,\zeta\in\D
\end{equation*}
where $\omega_x$ is the $x$-th moment of a weight $\omega$ defined by
$
 \omega_x=\int_0^1r^x\omega(r)\dd r.
$
Then, the orthogonal Bergman projection $P_\omega:L^2_\omega\to A^2_\omega$ can be written as
\begin{equation*}\label{eq:Bergman-projection}
 P_\omega f(z) =\int_\D f(\zeta)\overline{B_z^\omega(\zeta)}\omega(\zeta)\dd A(\zeta), \qquad f\in L^2_\omega.
\end{equation*}

Let $\omega$ be a nonnegative integrable function on $[0,1)$. For $r\in(0,1)$, define $\widehat\omega(r)=\int_r^1\omega(s)\dd s$.
A radial weight $\omega$ belongs to $\widehat{\mathcal D}$ if there exists a constant $C_\omega\geq1$ such that
\begin{equation}\label{eq:Dhat-intro}
 \widehat\omega(r)
 \leq C_\omega\widehat\omega\left(\frac{1+r}{2}\right),
 \qquad 0\leq r<1.
\end{equation}
By \cite[Lemma~2.1]{Pel16}, there exists a constant $C=C(\omega)\geq1$ such that
\begin{equation}\label{eq:moment-doubling}
 \omega_t\leq C\omega_{2t},
 \qquad t\geq0.
\end{equation}
There also exist constants $C_1=C_1(\omega),C_2=C_2(\omega)>0$ such that
\begin{equation}\label{eq:moment-tail-doubling}
 C_1\omega_x
 \leq\widehat\omega\left(1-\frac1x\right)
 \leq C_2\omega_x,
 \qquad x\geq1.
\end{equation}
Weighted Bergman spaces induced by \(\widehat{\mathcal D}\)-weights provide a natural generalization of \(A^2\). For example, it is the radial weight such that the Littlewood-Paley formula holds and the Bergman projection $P_\omega$ is $L^p$ bounded for $1<p<\infty$. At the endpoint spaces, $P_\omega$ satisfies the $L^\infty$-Bloch estimate and is of weak type $(1,1)$; see, for example, \cite{PR21} and \cite{LW26b} for more details.

For $u\in L^\infty(\D)$, the Toeplitz operator with symbol $u$ on the $\widehat{\mathcal{D}}$-weighted Bergman space is
\begin{align*}\label{eq:def-toeplitz}
 T_{\omega,u}f=P_\omega(uf), \qquad f\in A_\omega^2.
\end{align*}
This operator is bounded
$
 \norm{T_{\omega,u}}_{A^2_\omega\to A^2_\omega}
 \leq\norm{u}_{L^\infty(\D)},
$
and its Berezin transform is
\begin{equation}\label{eq:Berezin-transform}
 \widetilde u_\omega(a)
 :=\ip{T_{\omega,u}k_a^\omega}{k_a^\omega}
 =\int_\D u(z)|k_a(z)|^2\omega(z)\dd A(z),
\end{equation}
where $k_z^\omega  =\frac{B_z^\omega}{\norm{B_z^\omega}_{A^2_\omega}}$ is the normalized Bergman kernel.

\begin{theorem}\label{thm:scalar-AZ-Dhat}
Let $\omega\in\widehat{\mathcal D}$ and
$u\in L^\infty(\D)$. The Toeplitz operator $T_{\omega,u}$ is compact on $A^2_\omega$ if and only if
its Berezin transform satisfies the vanishing condition
 $$
  \lim_{|a|\to 1^-}\widetilde u_\omega(a)=0.
 $$
\end{theorem}

Suppose that $S$ is a finite sum of finite products of Toeplitz operators $T_{u_1}\cdots T_{u_n}$ on $A^2$ where each $u_j\in L^\infty(\D)$. The full Axler‑Zheng compactness theorem \cite[Theorem 2.2]{AZ98} states that $S$ is compact if and only if its Berezin transform $\widetilde{S}(a)=\langle Sk_a,k_a\rangle\to 0$ as $a\to \partial\mathbb{D}$. After we obtained Theorem \ref{thm:scalar-AZ-Dhat}, in a private communication, Professor Dechao Zheng asked one of us whether Theorem \ref{thm:scalar-AZ-Dhat} remains valid for the operator $S$ on $A_\omega^2$. 

Our next construction illustrates that Theorem \ref{thm:scalar-AZ-Dhat} is not true for the general $S$.

\begin{proposition}\label{prop:conuterexample}
Let $\omega$ be a radial weight given by
\begin{equation*}\label{eq:weight}
 \omega(r)=
 \begin{cases}
 \displaystyle \frac{e}{4}, &0\le r\le 1-e^{-1},\\[6pt]
 \displaystyle \frac{1}{(1-r)\left(\log\frac{e}{1-r}\right)^2}, &1-e^{-1}<r<1.
 \end{cases}
\end{equation*}
Then $\omega$ is a log-subharmonic weight, and is in $\widehat{\mathcal{D}}$. 
Moreover, there exist $L^\infty(\mathbb{D})$ functions $u$ and $v$ such that $S=T_{\omega,v}T_{\omega,u}$ is not compact on $A_\omega^2$ but its Berezin transform 
$$\lim_{|a|\to 1^-}\widetilde{S}_\omega(a)=0,$$
where $\widetilde{S}_\omega(a)=\ip{Sk_a^\omega}{k_a^\omega}$.
\end{proposition} 
\begin{remark}
A weight $\omega$ on the unit disk is said to be log-subharmonic if $\log \omega$ is a subharmonic function in the sense of distribution. Such a weight induces a natural hyperbolic Riemannian metric on the unit disk, is a starting point for proving the positivity of the Green function for the weighted biharmonic operator, and has many important applications in the operator theory of function spaces (see \cite{Hed00}, \cite{HJS02}, \cite{HS02}, \cite{AHR05} and \cite[Chapter 9]{HKZ} for more details). In his deep work, Shimorin \cite{Shi02} obtained a complete Bergman-type representation of the kernel of a radial log-subharmonic weighted Bergman space. Peláez, Rättyä, and Wick \cite{PRW19} established a surprising result by showing that all radial log-subharmonic weights are $\widehat{\mathcal{D}}$ weights. One can also find some results in \cite{LW26a}.
\end{remark}

\begin{remark}
The recent work by Looi \cite{Loo26} shows that, on each standard weighted Bergman space $A^2_{\omega_\alpha}$, where $\omega_\alpha(z)=(1+\alpha)(1-|z|^2)^\alpha$ with $\alpha>-1$, there is a bounded, non-compact $L^1$-symbol Toeplitz operator whose Berezin transform vanishes at the boundary. Hence, the constructed Toeplitz operator is not in the Toeplitz algebra generated by bounded symbols.
\end{remark}

The main part of Theorem~\ref{thm:scalar-AZ-Dhat} is the proof of sufficiency, and the ideas are illustrated as follows:
\begin{itemize}
\item Using the standard dyadic system $\mathcal Q$ on the unit circle, we associate with each $J\in\mathcal Q_j$ the polynomial $\phi_J$ defined by \eqref{eq:frame-element-definition}. The cutoff function \eqref{eq:chi-partition} restricts its polynomial expansion to degrees $n\asymp 2^j$, so the length of $J$ determines its frequency scale, while the midpoint $\theta_J$ determines its angular localization. Thus, $\phi_J$ is a dyadic polynomial which is a truncation of the normalized reproducing kernel at the point $a_J=r_je^{i\theta_J}$. The geometric meaning of the point $a_J$ is illustrated in Figure~\ref{fig:dyadic-generations-aJ}.
\begin{figure}[htbp]
\centering
\begin{tikzpicture}[scale=0.8,>=stealth]
    \def\R{3}
    \def\rj{1.90}
    \def\rjp{2.43}

    \coordinate (O) at (0,0);

    \coordinate (aJone) at (90:\rj);
    \coordinate (aJtwo) at (130:\rj);

    \coordinate (aJthree) at (80:\rjp);
    \coordinate (aJfour) at (100:\rjp);

    \path[fill=blue!6]
        (70:\rj)--(70:\R)
        arc[start angle=70,end angle=110,radius=\R]
        --(110:\rj)
        arc[start angle=110,end angle=70,radius=\rj]
        --cycle;

    \path[fill=blue!6]
        (110:\rj)--(110:\R)
        arc[start angle=110,end angle=150,radius=\R]
        --(150:\rj)
        arc[start angle=150,end angle=110,radius=\rj]
        --cycle;

    \path[fill=red!10]
        (70:\rjp)--(70:\R)
        arc[start angle=70,end angle=90,radius=\R]
        --(90:\rjp)
        arc[start angle=90,end angle=70,radius=\rjp]
        --cycle;

    \path[fill=red!10]
        (90:\rjp)--(90:\R)
        arc[start angle=90,end angle=110,radius=\R]
        --(110:\rjp)
        arc[start angle=110,end angle=90,radius=\rjp]
        --cycle;

    \draw[thick] (O) circle (\R);
    \draw[densely dashed,blue!65] (O) circle (\rj);
    \draw[densely dashed,red!65] (O) circle (\rjp);

    \draw[line width=3.5pt,blue!70,line cap=round]
        (70:\R)
        arc[start angle=70,end angle=110,radius=\R];

    \draw[line width=3.5pt,blue!70,line cap=round]
        (110:\R)
        arc[start angle=110,end angle=150,radius=\R];

    \draw[line width=1.5pt,red!80!black,line cap=round]
        (70:\R)
        arc[start angle=70,end angle=90,radius=\R];

    \draw[line width=1.5pt,red!80!black,line cap=round]
        (90:\R)
        arc[start angle=90,end angle=110,radius=\R];

    \foreach \angle in {70,90,110,150}
        \fill[black] (\angle:\R) circle (2.2pt);

    \draw[->,thick,gray!75]
        (aJone) to[bend right=12] (aJthree);

    \draw[->,thick,gray!75]
        (aJone) to[bend left=12] (aJfour);

    \fill[blue!75!black] (aJone) circle (2.4pt);
    \fill[blue!75!black] (aJtwo) circle (2.4pt);
    \fill[red!75!black] (aJthree) circle (2.4pt);
    \fill[red!75!black] (aJfour) circle (2.4pt);

    \node[anchor=north]
        at ($(aJone)+(0,-0.08)$) {$a_{J_1}$};

    \node[anchor=east]
        at ($(aJtwo)+(-0.10,0)$) {$a_{J_2}$};

    \node[anchor=south west]
        at ($(aJthree)+(0.05,0.05)$) {$a_{J_3}$};

    \node[anchor=south east]
        at ($(aJfour)+(-0.05,0.05)$) {$a_{J_4}$};

    \node[blue!80!black] at (90:3.43) {$J_1$};
    \node[blue!80!black] at (130:3.35) {$J_2$};
    \node[red!75!black] at (80:3.18) {$J_3$};
    \node[red!75!black] at (100:3.18) {$J_4$};

    \node[blue!75!black,fill=white,inner sep=1pt]
        at (194:\rj) {$r_j$};

    \node[red!75!black,fill=white,inner sep=1pt]
        at (165:\rjp) {$r_{j+1}$};

    \fill (O) circle (1.5pt);
    \node[below left] at (O) {$0$};
    \node[anchor=south west] at (12:\R) {$\partial\mathbb D$};
\end{tikzpicture}
\caption{Dyadic systems and the location of $a_J$}
\label{fig:dyadic-generations-aJ}
\end{figure}
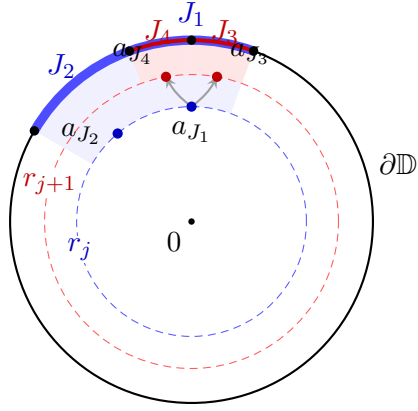
\item We then show that the family $\{\phi_J\}_{J\in\mathcal{Q}}$ forms a frame for $A^2_\omega$, and an isometry $W$ from $A^2_\omega$ into $\ell^2(\mathcal Q)$ is naturally introduced (Proposition~\ref{prop:Parseval}). Therefore, the compactness problem for $T_{\omega,u}$ on the weighted Bergman space $A_\omega^2$ 
    reduces to the compactness of
\[
A_u=WT_{\omega,u}W^*
\quad\text{on}\quad
\ell^2(\mathcal Q).
\]
Moreover, the operator $WT_{\omega,u}W^*$ has a natural matrix representation on $\ell^2(\mathcal Q)$, and its analysis is naturally divided into the diagonal, lower-triangular, and upper-triangular parts.
\item The first key step is to decompose the diagonal part further, according to the natural angular distance \eqref{eq:coarse-angular-distance} between the indices of $\phi_J$, into the local near-diagonal region \eqref{eq:lnd} and the separated near-diagonal region \eqref{eq:asnd}. After this decomposition, the strategy for estimating the diagonal part becomes clear. 
\item First, the separated near-diagonal part is controlled directly by the operator norm and the bound on the uniform angular distance; see Lemma~\ref{lem:near-diagonal-scalar-tail-2}. On the other hand, the vanishing Berezin transform condition plays a role in estimating the local near-diagonal region. More precisely, Lemma~\ref{lem:affine-polarization-Dhat}, combined with the kernel representation in Proposition~\ref{prop:frame-element-kernel-representation}, converts the boundary decay of the Berezin transform into the uniform decay of the matrix entries in every fixed angular band, as formulated in Lemma~\ref{lem:local-band-entries-Dhat}. Using the row- and column-degree bounds for the matrix from Lemma~\ref{lem:local-band-degree-Dhat}, by Schur's test in Lemma~\ref{lem:Schur}, we obtain the desired estimate in Proposition~\ref{lem:near-diagonal-scalar-tail}.
\item For the triangular parts, adjoint symmetry (Lemma~\ref{lem:adjoint-triangular-symmetry}) reduces the problem to the upper-triangular part \eqref{eq:ut}. This requires a finer band decomposition of the upper-triangular matrix \eqref{eq:fixed-band-matrix-projection}, and the norm of each band has a suitable decay related to a Carleson embedding estimate (Proposition~\ref{lem:upper-triangular-Carleson-scalar}). We use a stopping-time decomposition and the vanishing Berezin-transform condition to obtain the desired estimate; see Proposition~\ref{prop:upper-Carleson-one-sided} and Proposition~\ref{cor:scalar-Carleson-tail}.
\end{itemize}
Combining all the above estimates, we obtain a proof of Proposition~\ref{prop:scalar-tail-compression}. It is then used to show that, under the vanishing Berezin-transform condition, $A_u$ can be approximated in operator norm by finite-rank operators.

\subsection{Further discussions}
Following the breakthrough of Axler and Zheng~\cite{AZ98}, Engli\v{s}~\cite{Eng99} extended the Axler--Zheng theorem to Bergman spaces on irreducible bounded symmetric domains. On the unit ball, Su\'arez~\cite{Su07} enlarged the class of operators from finite sums of finite products to the Toeplitz algebra on $A^p(\mathbb B_n),1<p<\infty$. Mitkovski, Su\'arez, and Wick~\cite{MSW13} subsequently established the corresponding characterization for the standard weighted Bergman spaces $A^p_\alpha(\mathbb B_n)$ for $(1<p<\infty)$. Mitkovski and Wick~\cite{MW14} later developed a unified framework for reproducing kernel theses for boundedness and compactness on Bergman-type function spaces, recovering many of the preceding results within a common setting. More recently, Wang and Xia~\cite{WX21} established the analogous compactness characterization for operators in the Toeplitz algebra on Bergman spaces over smoothly bounded strongly pseudoconvex domains.

To keep the exposition and notation sufficiently simple, we restrict our attention in this work to one Toeplitz operator $T_{\omega,u}$, and formulate our main result in the Hilbert-space setting $A^2_\omega$. The argument used here is expected to yield the corresponding compactness characterization on $A^p_\omega$ for $1<p<\infty$ and $\omega\in\widehat{\mathcal D}$, although we do not pursue this extension in this paper. At present, however, our method makes essential use of the matrix representation of $T_{\omega,u}$ and does not extend directly to arbitrary elements of the Toeplitz algebra. It would therefore be interesting to identify natural subalgebras of the Toeplitz algebra on $A^p_\omega$ for which compactness is still characterized by the boundary vanishing of the Berezin transform. In particular, what are the precise relationships between compact operators in the Toeplitz algebra on the weighted radial Bergman space $A_\omega^2$ and the corresponding reproducing kernel theses.

Our frame is also related to the continuous-frame approaches of Batayneh--Mitkovski~\cite{BM16} and Mitkovski--Stockdale--Wagner--Wick~\cite{MSWW23} which are based on normalized reproducing kernels. Their general compactness criteria indeed do not require translation operators. In their applications to Axler--Zheng-type theorems, on the other hand, ideas associated with translation operators are used explicitly in~\cite[Theorem~5.1]{BM16} and indirectly through \cite[Proposition~2.7]{IMW15}. In contrast, our discrete frame of dyadic analytic polynomials yields a translation-free argument.

\medskip
\noindent\textbf{Organization of the paper.}
We construct the dyadic polynomial frame and the associated matrix representation of Toeplitz operators in Section~\ref{sec:preliminaries}. Moreover, Theorem~\ref{thm:scalar-AZ-Dhat} is proved if we assume that Proposition \ref{prop:scalar-tail-compression} holds. Section~\ref{sec:frame-estimate} establishes the localization and almost-orthogonality estimates for the frame. By using the boundary behavior of the Berezin transform, we obtain the diagonal estimate in Section \ref{sec:diag-estimate}. In Section~\ref{sec:low-uppwe-estimate}, the upper and lower triangular blocks are controlled and we complete the proof of Proposition \ref{prop:scalar-tail-compression}. The proof of Proposition \ref{prop:conuterexample} is given in Section \ref{sec:counter}. The appendix contains the Carleson embedding estimate used in previous sections.

\medskip

\noindent \textbf{Acknowledgements.} This work is supported by the National Natural Science Foundation of China (No.12471116) and 2025CDJ-IAIS YB-004 (Chongqing University).

\medskip
\noindent\textbf{AI Use Disclosure.}
The authors developed the mathematical ideas and verified all arguments. AI-based tools assisted with drafting, language editing, and consistency checks; the authors take full responsibility for the manuscript.

\section{The frame and matrix representation}\label{sec:preliminaries}
Throughout the paper, we fix a nonnegative function $\chi\in C_c^\infty((1/2,2))$ such that
\begin{equation}\label{eq:chi-partition}
 \sum_{j\in\mathbb Z}\chi(2^{-j}t)^2=1, \qquad t>0.
\end{equation}
In particular, we have $\text{supp}(\chi)\subset(1/2,2)$. One such function can be constructed as follows. Let
\[
 \alpha(t)=
 \begin{cases}
 \displaystyle
 \exp\left(-\frac{1}
 {(t-\frac35)(\frac53-t)}\right),
 &\frac35<t<\frac53,\\[6pt]
 0,&\text{otherwise},
 \end{cases}
\]
and define
\[
 \chi(t)
 =
 \frac{\alpha(t)}
 {\left(\displaystyle\sum_{k\in\mathbb Z}
 \alpha(2^{-k}t)^2\right)^{1/2}},
 \qquad t>0.
\]
We shall use the following estimate (see \cite[Lemma~5.2]{Pavlovic2014}). For completeness, we include a short proof here.
For a set $E\subset\R$, write
\[
 \operatorname{diam}(E)
 :=
 \sup\{|x-y|:x,y\in E\}.
\]

\begin{lemma}\label{lem:periodic-multiplier}
Let $\ell\ge2$ be an integer, and let $\{m_k\}_{k\ge1}\subset C_c^\ell(\R)$. Suppose that
\[
 L_0:=\sup_{k\ge1}\frac{\operatorname{diam}(\operatorname{supp}m_k)}{k},
 \qquad
 C_0:=\sup_{k\ge1}\norm{m_k}_{L^\infty(\R)},
 \qquad
 C_\ell:=\sup_{k\ge1}k^\ell
 \norm{m_k^{(\ell)}}_{L^\infty(\R)}
\]
are finite. Then there is a constant $C=C(\ell,L_0,C_0,C_\ell)>0$ such that, for every $k\ge1$,
\begin{equation}\label{eq:periodic-pointwise}
 \abs{\sum_{n\in\mathbb Z}m_k(n)e^{int}}
 \le Ck(1+k|t|)^{-\ell},
 \qquad |t|\le\pi.
\end{equation}
\end{lemma}

\begin{proof}
Since at most $(L_0+1)k$ terms on the left-hand side of \eqref{eq:periodic-pointwise} are nonzero, we have an upper bound  
\begin{equation}\label{eq:first-est}
 \abs{\sum_{n\in\mathbb Z}m_k(n)e^{int}} \le (L_0+1)C_0k.
\end{equation}
By summation by parts and the fundamental theorem of calculus,
\begin{align*}\label{eq:identity}
 (e^{-it}-1)^\ell\sum_{n\in\mathbb Z}m_k(n)e^{int} = \sum_{n\in\mathbb Z}\nabla^\ell m_k(n)e^{int},
\end{align*}
where
\[
 \nabla^\ell m_k(n) = \int_{[0,1]^\ell}
 m_k^{(\ell)}(n+u_1+\cdots+u_\ell)
 \,\dd u_1\cdots\dd u_\ell.
\]
Since that at most $(L_0+\ell+1)k$ terms $\nabla^\ell m_k(n)$ are nonzero, we have
\[
 \sum_{n\in\mathbb Z}\abs{\nabla^\ell m_k(n)}
 \le
 (L_0+\ell+1)C_\ell k^{1-\ell}.
\]
Since $|e^{-it}-1|\ge 2|t|/\pi$ for $|t|\le\pi$, it follows that
\[
 \abs{\sum_{n\in\mathbb Z}m_k(n)e^{int}}
 \le
 C_{\ell,L_0}C_\ell k^{1-\ell}|t|^{-\ell}
 =
 C_{\ell,L_0}C_\ell k(k|t|)^{-\ell},
 \qquad 0<|t|\le\pi.
\]
Combining this estimate with \eqref{eq:first-est}, we obtain
\[
 \abs{\sum_{n\in\mathbb Z}m_k(n)e^{int}}
 \le
 Ck\min\{1,(k|t|)^{-\ell}\},
 \qquad |t|\le\pi,
\]
where $C$ depends only on $\ell,L_0,C_0$, and $C_\ell$. Finally, note that
$ \min\{1,x^{-\ell}\}\le 2^\ell(1+x)^{-\ell}$ for
$ x>0,$
and therefore
\[
 \abs{\sum_{n\in\mathbb Z}m_k(n)e^{int}}
 \le Ck(1+k|t|)^{-\ell},
 \qquad |t|\le\pi.
\]
The proof is complete.
\end{proof}

\subsection{The Parseval frame}\label{sec:frame}
For an arc $I\subset\T$, let $|I|$ denote its arc-length. For $j\ge1$, set
\[
 \mathcal Q_j=\{J_{j,\ell}:0\leq\ell<2^{j+3}\},
\]
where
\[
 J_{j,\ell}
 =\set{e^{i\theta}:\frac{2\pi\ell}{2^{j+3}}
 \leq\theta<\frac{2\pi(\ell+1)}{2^{j+3}}}.
\]
For convenience, set $\mathcal Q_0=\{J_{0,0},J_{0,1}\}$ and define a dyadic system:
$$\mathcal Q=\bigcup_{j\geq0}\mathcal Q_j.$$
For $J=J_{j,\ell}\in\mathcal Q_j$, denote its midpoint by
\[
 \theta_J=\frac{2\pi(\ell+1/2)}{2^{j+3}}.
\]
Let $\mathbb{N}=\{0,1,2\cdots\}$ be the set of nonnegative integers. Then the family of functions
\begin{equation*}\label{eq:orthonormal-monomials}
 e_n(z)=\frac{z^n}{\sqrt{2\omega_{2n+1}}}, \quad n\in\mathbb{N},
\end{equation*}
is the normalized orthonormal basis of the $\widehat{\mathcal{D}}$-weighted Bergman space $A^2_\omega$. For each $J\in\mathcal{Q}_j,j\geq 1$, we define
\begin{equation}\label{eq:frame-element-definition}
 \phi_J =\frac1{\sqrt{2^{j+3}}}\sum_{n\ge1}\chi(2^{-j}n)e^{-in\theta_J}e_n.
\end{equation}
For $J\in\mathcal Q_j$, since $\operatorname{supp}\chi\subset(1/2,2)$ we have
\[\chi(2^{-j}n)\neq0 \quad \text{ only if}\quad 2^{j-1}<n<2^{j+1}.\]
Hence,
\[
\phi_J\in \operatorname{span}\bigl\{e_n:2^{j-1}<n<2^{j+1}\bigr\}.
\]
Thus, $\phi_J$ is a polynomial whose degrees are localized at the dyadic scale $2^j$. If $j=0$, set
$\phi_{J_{0,0}}=e_0$ and $\phi_{J_{0,1}}=e_1.$
Write 
\begin{align*}\label{eq:polyframe}
\mathcal F=\{\phi_J:J\in\mathcal Q\}.
\end{align*}

\begin{proposition}\label{prop:Parseval}
Let $\omega\in\widehat{\mathcal{D}}$. Then the family $\mathcal F$ forms a Parseval frame for $A^2_\omega$. More precisely,
\begin{equation*}\label{eq:Parseval}
 \norm{f}_{A^2_\omega}^2
 =\sum_{j\geq0}\sum_{J\in\mathcal Q_j}|\ip f{\phi_J}|^2,
 \quad f\in A^2_\omega.
\end{equation*}
\end{proposition}

\begin{proof}
For any $f\in A_\omega^2$, we have $f=\sum_{n\ge0}c_ne_n$. For $j\ge1$,
\[
 \ip f{\phi_J}
 =2^{-(j+3)/2}\sum_{n\ge1}
 c_n\chi(2^{-j}n)e^{in\theta_J}.
\]
Observe the following identity
\begin{align*}\label{eq:fourierind}
 \sum_{J\in\mathcal Q_j}e^{i(n-m)\theta_J}=0
 \quad\text{unless}\quad 2^{j+3}\mid(n-m).
\end{align*}
Since $\operatorname{supp}\chi\subset(1/2,2)$, we have
\begin{equation}\label{eq:single-level-Parseval}
 \sum_{J\in\mathcal Q_j}|\ip f{\phi_J}|^2
 =\sum_{n\ge1}\chi(2^{-j}n)^2|c_n|^2.
\end{equation}
For $n\ge2$, if $j\leq 0$, the terms in \eqref{eq:chi-partition} are equal to 0. Hence, we get
$
 \sum_{j\ge1}\chi(2^{-j}n)^2=1.
$
Observe that $\chi(2^{-j})=0$ for $j\ge1$. Summing \eqref{eq:single-level-Parseval} over $j\geq 0$, we obtain
\[
 \sum_{j\geq0}\sum_{J\in\mathcal Q_j}|\ip f{\phi_J}|^2
 =|c_0|^2+|c_1|^2+\sum_{n\ge2}|c_n|^2
 =\norm{f}_{A^2_\omega}^2.
\]
This completes the proof.
\end{proof}

\subsection{The matrix representation}
Let $\ell^2(\mathcal Q)$ denote the Hilbert space of all complex families $x=(x_J)_{J\in\mathcal Q}$ such that
\[
 \norm{x}_{\ell^2(\mathcal Q)}=\left(\sum_{J\in\mathcal Q}|x_J|^2\right)^{1/2}<\infty.
\]
Define a linear operator
$
 W:A^2_\omega\longrightarrow\ell^2(\mathcal Q)
$
by
\[
 Wf=\bigl(\ip f{\phi_J}\bigr)_{J\in\mathcal Q}.
\]
By Proposition~\ref{prop:Parseval}, $W$ is an isometry. Moreover, 
\begin{equation*}\label{eq:analysis-isometry}
 W^*W=I_{A^2_\omega},
 \quad
 \norm{W}=\norm{W^*}=1.
\end{equation*}
For $u\in L^\infty(\D)$, let
\begin{equation}\label{eq:Au}
 A_u=WT_{\omega,u}W^*.
\end{equation}
Then, we have
\begin{proposition}\label{prop:equcom}
Let $u\in L^\infty(\D)$. The Toeplitz operator $T_{\omega,u}$ is compact on $A^2_\omega$ if and only if $A_u$ is compact on
$\ell^2(\mathcal Q)$.
\end{proposition}
Observe that $A_u$ is an infinite matrix indexed by $\mathcal{Q}\times\mathcal{Q}$. Moreover, for $I,J\in\mathcal Q$,
\begin{equation}\label{eq:Toeplitz-frame-entries}
 A_u(I,J)=\left\langle T_{\omega,u} \phi_J,\phi_I\right\rangle_{A^2_\omega}
 =\int_\D u(z)\phi_J(z)\overline{\phi_I(z)}
 \omega(z)\dd A(z).
\end{equation}
For each $j\geq0$, let $\Pi_j$ denote the orthogonal projection of $\ell^2(\mathcal Q)$ onto $\ell^2(\mathcal Q_j)$. Then we have
\begingroup
\small
\setlength{\arraycolsep}{5pt}
\renewcommand{\arraystretch}{1.4}
\[
 A_u=
 \begin{array}{c|ccccc}
  &\ell^2(\mathcal Q_0)&\ell^2(\mathcal Q_1)
  &\ell^2(\mathcal Q_2)&\ell^2(\mathcal Q_3)&\cdots
  \\ \hline
 \ell^2(\mathcal Q_0)
  &\Pi_0A_u\Pi_0&\Pi_0A_u\Pi_1
  &\Pi_0A_u\Pi_2&\Pi_0A_u\Pi_3&\cdots
  \\
 \ell^2(\mathcal Q_1)
  &\Pi_1A_u\Pi_0&\Pi_1A_u\Pi_1
  &\Pi_1A_u\Pi_2&\Pi_1A_u\Pi_3&\cdots
  \\
 \ell^2(\mathcal Q_2)
  &\Pi_2A_u\Pi_0&\Pi_2A_u\Pi_1
  &\Pi_2A_u\Pi_2&\Pi_2A_u\Pi_3&\cdots
  \\
 \ell^2(\mathcal Q_3)
  &\Pi_3A_u\Pi_0&\Pi_3A_u\Pi_1
  &\Pi_3A_u\Pi_2&\Pi_3A_u\Pi_3&\cdots
  \\
 \vdots&\vdots&\vdots&\vdots&\vdots&\ddots
 \end{array},
\]
\endgroup
where
\begin{equation*}\label{eq:level-block-matrix}
 \Pi_kA_u\Pi_j
 =\left(
 A_u(I,J)
 \right)_{\substack{I\in\mathcal Q_k,J\in\mathcal Q_j}}
 :\ell^2(\mathcal Q_j)\longrightarrow\ell^2(\mathcal Q_k).
\end{equation*}

For $N\ge 2$, let
\begin{equation*}\label{eq:matrix-tail-projection}
 \Pi_{\ge N}=\sum_{j\geq N}\Pi_j
\end{equation*}
be the tail projection on $\ell^2(\mathcal Q)$. 

\begin{proposition}\label{prop:scalar-tail-compression}
Let $\omega\in\widehat{\mathcal D}$ and $u\in L^\infty(\D)$. 
If the Berezin transform
 $
  \lim_{|a|\to 1^-}\widetilde u_\omega(a)=0,
 $
then
\begin{equation*}\label{eq:scalar-tail-compression}
 \lim_{N\to\infty}\norm{\Pi_{\ge N}A_u\Pi_{\ge N}}=0.
\end{equation*}
\end{proposition}

We shall use the following Bergman kernel estimate (see \cite[Theorem~1]{PR16}). We also use the notation $B_\omega(z,\zeta)=B_\zeta^\omega(z)$.
\begin{lemma}\label{lem:norm-estimates-for-reproducing-kernel}
Let $\omega\in\widehat{\mathcal D}$. There exist constants $C_1=C_1(\omega)>0$ and $C_2=C_2(\omega)>0$ such that
\begin{equation}\label{eq:norm-estimates-for-reproducing-kernel}
 \frac{C_1}{(1-|z|)\widehat\omega(|z|)} \leq B_\omega(z,z) \leq\frac{C_2}{(1-|z|)\widehat\omega(|z|)},
 \qquad z\in\D.
\end{equation}
\end{lemma}

\begin{proof}[Proof of Theorem~\ref{thm:scalar-AZ-Dhat} assuming Proposition~\ref{prop:scalar-tail-compression}]
Suppose first that $T_{\omega,u}$ is compact. The normalized reproducing kernels $k_a^\omega$ converge weakly to zero in
$A_\omega^2$ as $|a|\to1^-$. Indeed, polynomials are dense in $A_\omega^2$, and, for every polynomial $f$, by Lemma \ref{lem:norm-estimates-for-reproducing-kernel}
\[
 \lim_{|a|\to 1^-}\left\langle f,k_a^\omega\right\rangle_{A_\omega^2} = \lim_{|a|\to 1^-} \frac{f(a)}{\sqrt{B_a^\omega(a)}}=0.
\]
Therefore,
\[
 \left|\widetilde u_\omega(a)\right|
 =
 \left|\left\langle
 T_{\omega,u}k_a^\omega,k_a^\omega
 \right\rangle_{A_\omega^2}\right|
 \leq
 \left\|T_{\omega,u}k_a^\omega\right\|_{A_\omega^2}
 \longrightarrow0.
\]
Conversely, assume $\widetilde u_\omega(a)\to0$ as $|a|\to 1^-$. By Proposition~\ref{prop:scalar-tail-compression} and observe that
\[
A_u - \Pi_{\ge N} A_u \Pi_{\ge N}=(I_{\ell^2(\mathcal Q)}-\Pi_{\ge N})A_u+\Pi_{\ge N}A_u(I_{\ell^2(\mathcal Q)}-\Pi_{\ge N}),
\]
we have $A_u$ is compact on $\ell^2(\mathcal Q)$. By Proposition~\ref{prop:equcom}, the Toeplitz operator $T_{\omega,u}$ is compact on $A^2_\omega$. This completes the proof.
\end{proof}

For the proof of Proposition~\ref{prop:scalar-tail-compression}, we use the following natural distance between
$I\in\mathcal Q_k$ and $J\in\mathcal Q_j$ for $j,k\geq1$:
\begin{equation}\label{eq:coarse-angular-distance}
 d_{j\wedge k}(I,J)=2^{\min\{j,k\}}d_\T(\theta_I,\theta_J),
\end{equation}
where $d_\T$ denotes the shortest angular distance on $\T$.

Fix integers $M\ge4$ and $L\ge1$. We split the matrix $\Pi_{\ge N}A_u\Pi_{\ge N}$ into four regions (see Figure \eqref{fig:four-matrix-regions}):
\begin{enumerate}
 \item \emph{near-diagonal}:
 \begin{align}\label{eq:lnd}
 \{(I,J):\ I\in\mathcal Q_k,\ J\in\mathcal Q_j,\ j,k\ge N,\
 |k-j|<M,\ d_{j\wedge k}(I,J)\le L\};
 \end{align}
 \item \emph{separated near-diagonal}:
 \begin{align}\label{eq:asnd}
 \{(I,J):\ I\in\mathcal Q_k,\ J\in\mathcal Q_j,\ j,k\ge N,\
 |k-j|<M,\ d_{j\wedge k}(I,J)>L\};
 \end{align}
 \item \emph{lower-triangular}:
 \begin{align}\label{eq:lt}
 \{(I,J):\ I\in\mathcal Q_k,\ J\in\mathcal Q_j,\ j,k\ge N,\
 k-j\ge M\};
 \end{align}
 \item \emph{upper-triangular}:
 \begin{align}\label{eq:ut}
 \{(I,J):\ I\in\mathcal Q_k,\ J\in\mathcal Q_j,\ j,k\ge N,\
 j-k\ge M\}.
 \end{align}
\end{enumerate}
Let the corresponding blocks be $A_{N,1}^{M,L}(u)$, $A_{N,2}^{M,L}(u)$, $A_{N,3}^{M}(u)$, and $A_{N,4}^{M}(u)$. Thus
\begin{equation}\label{eq:four-region-decomposition-scalar}
 \Pi_{\ge N}A_u\Pi_{\ge N}=A_{N,1}^{M,L}(u)+A_{N,2}^{M,L}(u)+A_{N,3}^{M}(u)+A_{N,4}^{M}(u)
\end{equation}
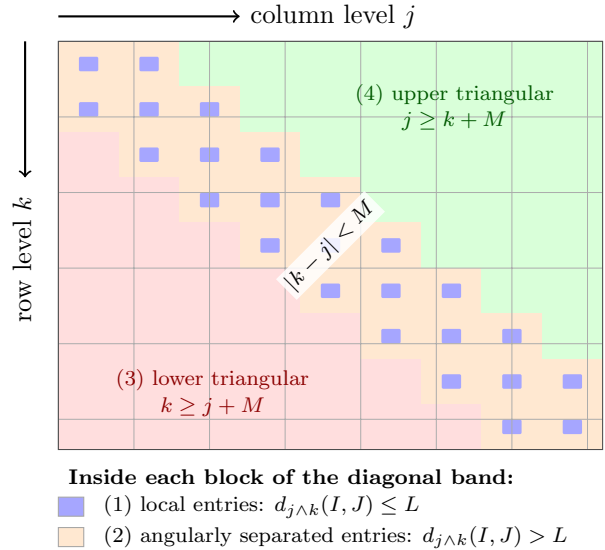
\begin{figure}[htbp]
\centering
\begin{tikzpicture}[x=0.8cm,y=-0.6cm,
  every node/.style={font=\small}]
  \foreach \j in {0,...,8}{
    \foreach \k in {0,...,8}{
      \pgfmathtruncatemacro{\d}{\k-\j}
      \ifnum\d>1
        \fill[red!13] (\j,\k) rectangle ++(1,1);
      \else
        \ifnum\d<-1
          \fill[green!15] (\j,\k) rectangle ++(1,1);
        \else
          \fill[orange!18] (\j,\k) rectangle ++(1,1);
          \fill[blue!35,rounded corners=0.6pt]
            (\j+0.34,\k+0.34) rectangle (\j+0.66,\k+0.66);
        \fi
      \fi
    }
  }
  \draw[gray!65,very thin] (0,0) grid (9,9);
  \draw[black!70] (0,0) rectangle (9,9);
  \draw[->,thick] (0,-0.55) -- (3.0,-0.55)
    node[right] {column level $j$};
  \draw[->,thick] (-0.55,0) -- (-0.55,3.0);
  \node[rotate=90] at (-0.6,4.8) {row level $k$};
  \node[align=center,font=\scriptsize,text=green!35!black] at (6.55,1.45)
    {(4) upper triangular\\$j\ge k+M$};
  \node[align=center,font=\scriptsize,text=red!55!black] at (2.55,7.75)
    {(3) lower triangular\\$k\ge j+M$};
  \node[rotate=45,fill=white,fill opacity=.84,text opacity=1,
    font=\scriptsize,inner sep=1.5pt] at (4.5,4.5) {$|k-j|<M$};
  \node[anchor=west,font=\scriptsize\bfseries] at (0,9.65)
    {Inside each block of the diagonal band:};
  \fill[blue!35,rounded corners=0.6pt] (0,10.05) rectangle (0.42,10.47);
  \draw[gray!70] (0,10.05) rectangle (0.42,10.47);
  \node[anchor=west,font=\scriptsize] at (0.58,10.26)
    {(1) local entries: $d_{j\wedge k}(I,J)\le L$};
  \fill[orange!18] (0,10.72) rectangle (0.42,11.14);
  \draw[gray!70] (0,10.72) rectangle (0.42,11.14);
  \node[anchor=west,font=\scriptsize] at (0.58,10.93)
    {(2) angularly separated entries: $d_{j\wedge k}(I,J)>L$};
\end{tikzpicture}
\label{fig:four-matrix-regions}
\caption{The decomposition of the matrix}
\end{figure}

The proof of Proposition \ref{prop:scalar-tail-compression} is based on Lemma~\ref{lem:adjoint-triangular-symmetry}, Proposition~\ref{lem:near-diagonal-scalar-tail}, Lemma \ref{lem:near-diagonal-scalar-tail-2}, Lemma~\ref{lem:local-band-entries-Dhat} and Proposition~\ref{prop:upper-compact}.

\begin{proof}[Proof of Proposition~\ref{prop:scalar-tail-compression}]
By \eqref{eq:four-region-decomposition-scalar}, 
\begin{align*}
 \norm{\Pi_{\geq N}A_u\Pi_{\geq N}}
 \leq
 \norm{A_{N,1}^{M,L}(u)}
 +\norm{A_{N,2}^{M,L}(u)}
 +\norm{A_{N,3}^{M}(u)}
 +\norm{A_{N,4}^{M}(u)}.
\end{align*}
By Proposition~\ref{lem:near-diagonal-scalar-tail} and Lemma \ref{lem:near-diagonal-scalar-tail-2},
\begin{align}
 \norm{\Pi_{\geq N}A_u\Pi_{\geq N}}
 &\leq
 2^{M+6}LM\,\tau(N;M,L)
 +
 C2^{M/2}(1+L)^{-2}\norm{u}_{L^\infty(\D)}
 \notag\\
 &\quad+
 \norm{A_{N,3}^{M}(u)}
 +
 \norm{A_{N,4}^{M}(u)}.
 \label{eq:scalar-tail-master-quantitative}
\end{align}
By Proposition~\ref{prop:upper-compact} and Lemma~\ref{lem:adjoint-triangular-symmetry}, we have
\begin{equation*}\label{eq:triangular-double-limit-scalar}
 \lim_{M\to\infty}\limsup_{N\to\infty}
 \left(\norm{A_{N,3}^{M}(u)} + \norm{A_{N,4}^{M}(u)}\right)
 = 0.
\end{equation*}

Fix $\delta>0$. Choose $M\geq4$ such that
\[
 \limsup_{N\to\infty}
 \left(
 \norm{A_{N,3}^{M}(u)}
 +
 \norm{A_{N,4}^{M}(u)}
 \right)
 <
 \frac{\delta}{2},
\]
and then choose $L\geq1$ so large that
\[
 C2^{M/2}(1+L)^{-2}\norm{u}_{L^\infty(\D)}
 <
 \frac{\delta}{4}.
\]
For this fixed pair $(M,L)$, it follows from Lemma~\ref{lem:local-band-entries-Dhat} that
\[
 \lim_{N\to\infty}\tau(N;M,L)=0.
\]
Moreover, by \eqref{eq:scalar-tail-master-quantitative} 
\[
 \limsup_{N\to\infty}
 \norm{\Pi_{\geq N}A_u\Pi_{\geq N}}
 <
 \delta.
\]
Since $\delta>0$ is arbitrary,  we have the desired result
\begin{equation*}
 \lim_{N\to\infty}\norm{\Pi_{\ge N}A_u\Pi_{\ge N}}=0.
\end{equation*}
This completes the whole proof.
\end{proof}

\section{The frame estimate}\label{sec:frame-estimate}
\subsection{The frame and reproducing kernels}
For $j\geq 1$, set $r_j=1-2^{-j}$. Then the normalized reproducing kernel of $A_\omega^2$ at the point $r_je^{i\vartheta}$ 
can be written as
$$
 k_{r_je^{i\vartheta}}^\omega =\sum_{n\geq0}\kappa_{j,n}e^{-in\vartheta}e_n,
$$
where
$
 \kappa_{j,n}
 =r_j^n/\sqrt{2\omega_{2n+1}B_\omega(r_j,r_j)}.
$
For $x\in\R$, define
\begin{equation*}\label{eq:qj-definition}
 q_j(x)=
 \begin{cases}
 \displaystyle
 \frac{\chi(2^{-j}x)}
 {2^{(j+3)/2}} \frac{\sqrt{2\omega_{2x + 1} B_\omega(r_j,r_j)}}{r_j^x},
 &x>0,\\[8pt]
 0,&x\leq0.
 \end{cases}
\end{equation*}
Since $\supp\chi\subset(1/2,2)$, we have $\supp q_j\subset(2^{j-1},2^{j+1})$, and $q_j\in C_c^\infty(\R)$. Moreover,
\[
 q_j(n)\kappa_{j,n}
 =
 \frac{\chi(2^{-j}n)}{2^{(j+3)/2}},
 \quad n\geq0.
\]

By \eqref{eq:moment-tail-doubling}, the monotonicity of $\widehat\omega$, and \eqref{eq:Dhat-intro}, there exist positive constants $c_0=c_0(\omega)$ and $C_0=C_0(\omega)$ such that
\[
 c_0\widehat\omega(r_j)\leq \omega_{2x+1}\leq
 C_0\widehat\omega(r_j),
 \quad x\in\supp q_j.
\]
By \eqref{eq:norm-estimates-for-reproducing-kernel}, there exist constants $c_1=c_1(\omega), C_1=C_1(\omega)>0$ such that
\begin{equation}\label{eq:qj-moment-kernel}
 c_1 2^j\leq 2\omega_{2x+1}B_\omega(r_j,r_j)\leq C_1 2^j, \quad x\in\supp q_j.
\end{equation}
For $t\in[-\pi,\pi]$, define
\[
 L_j(t)=\sum_{n\in\Z}q_j(n)e^{int}.
\]
By Lemma~\ref{lem:periodic-multiplier}, for every integer $p\ge2$,
\begin{equation*}\label{eq:Lj-pointwise}
 |L_j(t)|
 \leq
 C_{\omega,p,\chi}2^j(1+2^j|t|)^{-p},
 \quad |t|\leq\pi.
\end{equation*}
Then,
\begin{equation}\label{eq:Lj-L1}
 \frac1{2\pi}\int_{-\pi}^{\pi}|L_j(t)|\dd t
 \leq C_{\omega,p,\chi}.
\end{equation}
Moreover, for every $\rho\ge1$,
\begin{equation}\label{eq:Lj-tail}
 \frac1{2\pi}
 \int_{\substack{-\pi\leq t\leq\pi\\|t|>\rho2^{-j}}}
 |L_j(t)|\dd t
 \leq
 C_{\omega,p,\chi}(1+\rho)^{1-p}.
\end{equation}

\begin{proposition}\label{prop:frame-element-kernel-representation}
For every $j\geq1$ and $J\in\mathcal Q_j$, we have
\begin{equation*}\label{eq:frame-element-kernel-representation}
 \phi_J =\frac1{2\pi}\int_{-\pi}^{\pi} L_j(t)k_{r_je^{i(\theta_J+t)}}^\omega\dd t,
\end{equation*}
where the integral is the Bochner integral.
\end{proposition}

\begin{proof}
The map $t\mapsto k_{r_je^{i(\theta_J+t)}}^\omega$ is continuous in $A^2_\omega$, and hence the mapping is strongly measurable. By
\eqref{eq:Lj-L1},
\[
 \int_{-\pi}^{\pi}
 \norm{L_j(t)k_{r_je^{i(\theta_J+t)}}^\omega}_{A^2_\omega}\dd t
 =\norm{L_j}_{L^1(-\pi,\pi)}<\infty,
\]
so the Bochner integral exists. For any integer $n\geq 0$,
\begin{align*}
 \frac1{2\pi}
 \int_{-\pi}^{\pi}L_j(t)
 \ip{k_{r_je^{i(\theta_J+t)}}^\omega}{e_n}\dd t 
 &=\kappa_{j,n}e^{-in\theta_J}
 \frac1{2\pi}\int_{-\pi}^{\pi}
 L_j(t)e^{-int}\dd t\\
 &=\kappa_{j,n}q_j(n)e^{-in\theta_J}=2^{-(j+3)/2}\chi(2^{-j}n)e^{-in\theta_J}.
\end{align*}
These are precisely the coefficients of $\phi_J$ defined by \eqref{eq:frame-element-definition}.
This completes the proof.
\end{proof}
For every integer $s\geq1$, by the monotonicity of the moments and the
elementary estimate
$
 \sup_{0<r<1} r^x|\log r|^s\leq C_sx^{-s},
$
we have, for every $x\in\supp q_j$,
\begin{align*}
 \left|\frac{d^s}{dx^s}\omega_{2x+1}\right|
 \leq 2^s\int_0^1 r^{2x+1}|\log r|^s\omega(r)\,dr\leq C_sx^{-s}\omega_{x+1}
 \leq C_{\omega,s}2^{-js}\omega_{2x+1}.
\end{align*}
For $\alpha=1/2$, the $s$-th derivative of $\bigl(\omega_{2x+1}\bigr)^\alpha$ is a finite linear combination of
terms of 
\[
 \bigl(\omega_{2x+1}\bigr)^{\alpha-m}
 \prod_{\ell=1}^m
 \frac{d^{r_\ell}}{dx^{r_\ell}}\omega_{2x+1},
 \qquad
 r_\ell\geq1,\qquad
 r_1+\cdots+r_m=s.
\]
Consequently, for every integer $s\geq0$,
\begin{equation}\label{eq:power-moment-derivatives}
 \left|
 \frac{d^s}{dx^s}
 \bigl(\omega_{2x+1}\bigr)^\alpha
 \right|
 \leq
 C_{\omega,s,\alpha}2^{-js}
 \bigl(\omega_{2x+1}\bigr)^\alpha,
 \qquad x\in\supp q_j.
\end{equation}
Since $\supp q_j\subset(2^{j-1},2^{j+1})$, it follows that, for every integer $s\geq0$,
\begin{equation*}\label{eq:qj-derivatives}
 \norm{q_j^{(s)}}_\infty
 \leq C_{\omega,s,\chi}2^{-js}.
\end{equation*}

\subsection{Local estimates for the frame}
For $j\geq1$ and $J\in\mathcal Q_j$, define the dyadic Carleson tent
\[
 T(J)=\set{re^{i\theta}:e^{i\theta}\in J,\ 1-2^{-j}\leq r<1}
\]
and $a_J=(1-2^{-j})e^{i\theta_J}$. Since
$|J|=2\pi/2^{j+3}$ and $j\geq 1$, we have
\[
 \omega(T(J))
 =\frac1{2^{j+2}}\int_{1-2^{-j}}^1r\omega(r)\dd r
\]
and
\begin{equation}\label{eq:tent-volume}
 2^{-j-3}\widehat\omega(1-2^{-j})\leq\omega(T(J))\leq 2^{-j-2}\widehat\omega(1-2^{-j}).
\end{equation}
If $k\ge j\ge1$, $I\in\mathcal Q_k$, and $J\in\mathcal Q_j$, then
\begin{equation}\label{eq:tent-volume-ratio}
 \omega(T(I))\le 2^{j-k}\omega(T(J)).
\end{equation}

\begin{lemma}\label{lem:frame-element-localization}
Let $\omega\in\widehat{\mathcal{D}}$. For every integer $p\geq2$, there is a constant $C=C(\omega,p,\chi)>0$ such that, for $J\in\mathcal Q_j,j\geq 1$, and $z=re^{i\theta}\in\D$,
\begin{equation}\label{eq:frame-element-size}
 \bigl(1+2^j d_\T(\theta,\theta_J)\bigr)^p
 |\phi_J(re^{i\theta})|
 \leq
 \frac{C}{\sqrt{\omega(T(J))}}
 e^{-2^{j-2}(1-r)}.
\end{equation}
Moreover,
\begin{equation}\label{eq:frame-element-smoothness}
 \bigl(1+2^jd_\T(\theta,\theta_J)\bigr)^p
 \left(
  |\partial_r\phi_J(re^{i\theta})|
  +|\partial_\theta\phi_J(re^{i\theta})|
 \right)
 \leq
 \frac{C2^j}{\sqrt{\omega(T(J))}}
 e^{-2^{j-3}(1-r)}.
\end{equation}
\end{lemma}

\begin{proof}
Fix $p\geq2$ and define
\[
 a_j(x)=
 \begin{cases}
 \displaystyle
 \frac{\chi(2^{-j}x)}
 {\sqrt{2^{j+3}}\sqrt{2\omega_{2x+1}}},
 &x>0,\\[7pt]
 0,&x\leq0.
 \end{cases}
\]
Then $a_j\in C_c^\infty(\R)$, and its support is contained in $(2^{j-1},2^{j+1})$. Moreover,
\begin{equation}\label{eq:frame-element-periodic-sum}
 \phi_J(re^{i\theta})
 =
 \sum_{n\in\Z}a_j(n)r^ne^{in(\theta-\theta_J)}.
\end{equation}
By \eqref{eq:qj-moment-kernel} and Lemma~\ref{lem:norm-estimates-for-reproducing-kernel}, there exist
positive constants $c_2=c_2(\omega)$ and $C_2=C_2(\omega)$ such that
\begin{equation}\label{eq:moment-tent-comparison}
 c_2 2^j\omega(T(J))
 \leq\omega_{2x+1}
 \leq C_2 2^j\omega(T(J)),
 \qquad x\in\supp a_j.
\end{equation}
By the same argument used in
\eqref{eq:power-moment-derivatives}, for every nonnegative integer
$s$ and every $x\in\supp a_j$,
\[
 \left|
 \frac{d^s}{dx^s}\omega_{2x+1}^{-1/2}
 \right|
 \leq
 C_{\omega,s}2^{-js}\omega_{2x+1}^{-1/2}.
\]
Thus,
\begin{equation}\label{eq:aj-derivatives}
 \norm{a_j^{(s)}}_\infty
 \leq
 \frac{C_{\omega,s,\chi}2^{-j(s+1)}}
 {\sqrt{\omega(T(J))}},
 \qquad s\geq0.
\end{equation}
Let $0<r<1$. For every $x\in\supp a_j$, we have
\begin{align}
 \left|\frac{d^s}{dx^s}r^x\right| \leq C_s2^{-js}e^{-2^{j-2}(1-r)}.
 \label{eq:radial-multiplier-derivatives}
\end{align}
It follows from \eqref{eq:aj-derivatives} and \eqref{eq:radial-multiplier-derivatives} that
\[
 \left\|
 \frac{d^s}{dx^s}\bigl(a_j(x)r^x\bigr)
 \right\|_\infty
 \leq
 \frac{C_{\omega,s,\chi}2^{-j(s+1)}}
 {\sqrt{\omega(T(J))}}
 e^{-2^{j-2}(1-r)}.
\]
Applying Lemma~\ref{lem:periodic-multiplier} with $N=2^j$ to $a_j(x)r^x$, we obtain
\[
 \bigl(1+2^jd_\T(\theta,\theta_J)\bigr)^p
 |\phi_J(re^{i\theta})|
 \leq
 \frac{C_{\omega,p,\chi}}
 {\sqrt{\omega(T(J))}}
 e^{-2^{j-2}(1-r)}.
\]
The case $r=0$ follows by continuity. This proves the desired inequality
\eqref{eq:frame-element-size}.

For the angular derivative, by the above arguments, for
$0\leq s\leq p$, we have
\[
 \left\|
 \frac{d^s}{dx^s}\bigl(xa_j(x)r^x\bigr)
 \right\|_\infty
 \leq
 \frac{C_{\omega,s,\chi}2^{-js}}
 {\sqrt{\omega(T(J))}}
 e^{-2^{j-2}(1-r)}.
\]
By Lemma~\ref{lem:periodic-multiplier} and \eqref{eq:frame-element-periodic-sum}, we obtain
\begin{equation}\label{eq:angular-derivative-localization}
 \bigl(1+2^jd_\T(\theta,\theta_J)\bigr)^p
 |\partial_\theta\phi_J(re^{i\theta})|
 \leq
 \frac{C_{\omega,p,\chi}2^j}
 {\sqrt{\omega(T(J))}}
 e^{-2^{j-2}(1-r)}.
\end{equation}
For the radial derivative, suppose first that $j\geq2$. By direct computations,
\[
 \left\|
 \frac{d^s}{dx^s}\bigl(xa_j(x)r^{x-1}\bigr)
 \right\|_\infty
 \leq
 \frac{C_{\omega,s,\chi}2^{-js}}
 {\sqrt{\omega(T(J))}}
 e^{-2^{j-3}(1-r)}.
\]
By Lemma~\ref{lem:periodic-multiplier}, we get
\begin{equation}\label{eq:radial-derivative-localization}
 \bigl(1+2^jd_\T(\theta,\theta_J)\bigr)^p
 |\partial_r\phi_J(re^{i\theta})|
 \leq
 \frac{C_{\omega,p,\chi}2^j}
 {\sqrt{\omega(T(J))}}
 e^{-2^{j-3}(1-r)}.
\end{equation}
The estimate for $j=1$ is clear and the estimate at $r=0$ follows again by continuity. It follows from
\eqref{eq:angular-derivative-localization} and \eqref{eq:radial-derivative-localization} that \eqref{eq:frame-element-smoothness} holds. This completes the whole proof.
\end{proof}

\begin{lemma}\label{lem:frame-element-moments}
Let $\omega\in\widehat{\mathcal{D}}$. For every integer $q\geq0$, there is a constant $C=C(\omega, q,\chi)>0$ such that, for every $J\in\mathcal Q_j,j\geq 1$,
\begin{equation*}\label{eq:frame-element-moments}
 \int_{-\pi}^{\pi}\int_0^1
 \bigl(1+2^j(1-r)+2^j d_{\T}(\theta,\theta_J)\bigr)^q
 |\phi_J(re^{i\theta})|\,r\omega(r)\,dr\,d\theta
 \le C \sqrt{\omega(T(J))}.
\end{equation*}
\end{lemma}

\begin{proof}
For $r\in(0,1)$ and $\theta\in[-\pi,\pi]$, observe that
\[
 1+2^j(1-r)+2^jd_\T(\theta,\theta_J)
 \leq (1+2^j(1-r))(1+2^jd_\T(\theta,\theta_J)).
\]
Using Lemma~\ref{lem:frame-element-localization} with $p=q+2$, we have
\begin{align*}
 \bigl(1+2^j(1-r)+2^jd_\T(\theta,\theta_J)\bigr)^q
 |\phi_J(re^{i\theta})| \notag\\\leq
 \frac{C}{\sqrt{\omega(T(J))}}
 (1+2^j(1-r))^qe^{-2^{j-2}(1-r)}
 (1+2^jd_\T(\theta,\theta_J))^{-2}.
\end{align*}
Note that
\begin{equation*}\label{eq:angular-integral-bound}
 \int_{-\pi}^{\pi}(1+2^jd_\T(\theta,\theta_J))^{-2}\,d\theta=2\int_0^\pi(1+2^jt)^{-2}\,dt\leq2^{1-j}.
\end{equation*}
By \eqref{eq:Dhat-intro}, we have
\begin{equation}\label{eq:iterated-tail-bound}
 \widehat\omega\left(1-2^{m-j}\right) \leq C_\omega^m\widehat\omega(r_j), \quad 0\leq m\leq j.
\end{equation}
Note also that
$
 [0,1)=[r_j,1)\cup_{m=0}^{j-1} \left[1-2^{m+1-j},1-2^{m-j}\right).
$
By \eqref{eq:iterated-tail-bound}, we obtain
\begin{align*}
 \int_0^1(1+2^j(1-r))^qe^{-2^{j-2}(1-r)}r\omega(r)\,dr
 &\leq \widehat\omega(r_j)\left(2^q+\sum_{m=0}^{j-1}C_\omega^{m+1}(1+2^{m+1})^qe^{-2^m/4}\right)\\
 &\leq C_{\omega,q}\widehat\omega(r_j).
\end{align*}
Therefore, we get
\[
 \int_{-\pi}^{\pi}\int_0^1
 \bigl(1+2^j(1-r)+2^jd_\T(\theta,\theta_J)\bigr)^q
 |\phi_J(re^{i\theta})|r\omega(r)\,dr\,d\theta
 \leq
 \frac{C_{\omega,q,\chi}\widehat\omega(r_j)}
 {2^j\sqrt{\omega(T(J))}}.
\]
Since $\widehat\omega(r_j)\leq 2^{j+3}\omega(T(J))$, the right-hand side is bounded by $2^3C_{\omega,q,\chi}\sqrt{\omega(T(J))}$. This completes the proof.
\end{proof}
\subsection{Almost-orthogonality estimates for the frame}
If $k\geq j\geq1$, then
\[
 d_{j\wedge k}(I,J)=d_j(I,J)
 =2^j d_{\T}(\theta_I,\theta_J).
\]
Moreover, for any $\theta$, 
\begin{align}
 1+d_j(I,J) \leq \bigl(1+2^j d_{\T}(\theta,\theta_I)\bigr)\bigl(1+2^j d_{\T}(\theta,\theta_J)\bigr).
 \label{eq:overlap-distance-factorization}
\end{align}

\begin{lemma}\label{lem:frame-element-overlap}
Let $\omega\in\widehat{\mathcal{D}}$. There exists a constant $C=C(\omega,\chi)>0$ such that, for $J\in\mathcal Q_j$, $I\in\mathcal Q_k, k\geq j\geq1$, we have
\begin{equation}\label{eq:frame-element-overlap}
 (1+d_j(I,J))^3
 \int_\D|\phi_I(z)\phi_J(z)|\omega(z)\dd A(z)
 \leq C\left(\frac{\omega(T(I))}{\omega(T(J))}\right)^{1/2},
\end{equation}
and 
\begin{equation}\label{eq:frame-element-first-order-overlap}
 (1+d_j(I,J))^3
 \int_\D|\phi_I(z)|\,|\phi_J(z)-\phi_J(a_I)|
 \omega(z)\dd A(z)
 \leq C2^{j-k}
 \left(\frac{\omega(T(I))}{\omega(T(J))}\right)^{1/2}.
\end{equation}
\end{lemma}

\begin{proof}
Fix $I\in\mathcal Q_k$ and $J\in\mathcal Q_j$, where $k\geq j\geq1$,
and write $z=re^{i\theta}$. By
\eqref{eq:overlap-distance-factorization},
\[
 (1+d_j(I,J))^3
 \leq
 \bigl(1+2^jd_\T(\theta,\theta_I)\bigr)^3
  \bigl(1+2^jd_\T(\theta,\theta_J)\bigr)^3.
\]
Combining this with \eqref{eq:frame-element-size} for $\phi_J$ with $p=3$ and $k\geq j$, we obtain
\begin{align*}
 (1+d_j(I,J))^3
 |\phi_I(re^{i\theta})\phi_J(re^{i\theta})|\leq
 \frac{C}{\sqrt{\omega(T(J))}}
 \bigl(1+2^kd_\T(\theta,\theta_I)\bigr)^3
 |\phi_I(re^{i\theta})|,
\end{align*}
Integrating over $\D$ and applying Lemma~\ref{lem:frame-element-moments} with $q=3$
give 
\begin{align*}
 &(1+d_j(I,J))^3
 \int_\D|\phi_I(z)\phi_J(z)|\omega(z)\dd A(z)\\
 &\qquad\leq
 \frac{C}{\sqrt{\omega(T(J))}}
 \int_\D
 \bigl(1+2^kd_\T(\theta,\theta_I)\bigr)^3
 |\phi_I(re^{i\theta})|\omega(re^{i\theta})\dd A(re^{i\theta})\\
 &\qquad\leq
 C\left(\frac{\omega(T(I))}{\omega(T(J))}\right)^{1/2}.
\end{align*}
This proves the desired inequality \eqref{eq:frame-element-overlap}.

Recall that $a_I=r_ke^{i\theta_I}$, $r_k=1-2^{-k}$. Connect $re^{i\theta}$ to $a_I$ first by the radial segment from $re^{i\theta}$ to $r_ke^{i\theta}$ and then by the shorter circular arc from $r_ke^{i\theta}$ to $r_ke^{i\theta_I}$. Then, we have
\begin{align*}
 \phi_J(re^{i\theta})-\phi_J(a_I)
 &=
 \int_{r_k}^{r}
 \partial_\rho\phi_J(\rho e^{i\theta})\,\dd\rho
 +
 \int_{\theta_I}^{\theta}
 \partial_\vartheta\phi_J(r_ke^{i\vartheta})\,\dd\vartheta.
\end{align*}
On the radial segment, by \eqref{eq:overlap-distance-factorization} and \eqref{eq:frame-element-smoothness} with $p=3$, we have
\begin{align*}
 (1+d_j(I,J))^3 |\partial_\rho\phi_J(\rho e^{i\theta})| \leq \frac{C2^j}{\sqrt{\omega(T(J))}}\bigl(1+2^jd_\T(\theta,\theta_I)\bigr)^3.
\end{align*}
For every $\vartheta$ between $\theta_I$ and $\theta$ along the
chosen shorter arc,
$ d_\T(\vartheta,\theta_I) \leq d_\T(\theta,\theta_I).
$
Applying \eqref{eq:overlap-distance-factorization} at the angular
point $\vartheta$ and then using
\eqref{eq:frame-element-smoothness} with $p=3$, we similarly obtain
\begin{align*}
 (1+d_j(I,J))^3|\partial_\vartheta\phi_J(r_ke^{i\vartheta})|\leq \frac{C2^j}{\sqrt{\omega(T(J))}}\bigl(1+2^jd_\T(\theta,\theta_I)\bigr)^3.
\end{align*}
Therefore,
\begin{align*}
 (1+d_j(I,J))^3 |\phi_J(re^{i\theta})-\phi_J(a_I)|\leq
 \frac{C2^j}{\sqrt{\omega(T(J))}}
 \bigl(|r-r_k|+d_\T(\theta,\theta_I)\bigr)
 \bigl(1+2^jd_\T(\theta,\theta_I)\bigr)^3.
\end{align*}
Since $r_k=1-2^{-k}$ and $j\leq k$, we have
\begin{align*}
 (1+d_j(I,J))^3
 |\phi_J(re^{i\theta})-\phi_J(a_I)|
 \leq
 \frac{C2^{j-k}}{\sqrt{\omega(T(J))}}
 \bigl(1+2^k(1-r)+2^kd_\T(\theta,\theta_I)\bigr)^4.
\end{align*}
Multiplying by $|\phi_I(re^{i\theta})|$, integrating over $\D$, and using Lemma~\ref{lem:frame-element-moments} with $q=4$, we conclude that
\begin{align*}
 &(1+d_j(I,J))^3
 \int_\D|\phi_I(z)|\,|\phi_J(z)-\phi_J(a_I)|
 \omega(z)\dd A(z)\\
 &\qquad\leq
 \frac{C2^{j-k}}{\sqrt{\omega(T(J))}}
 \int_\D
 \bigl(1+2^k(1-r)+2^kd_\T(\theta,\theta_I)\bigr)^4
 |\phi_I(re^{i\theta})|
 \omega(re^{i\theta})\dd A(re^{i\theta})\\
 &\qquad\leq
 C2^{j-k}
 \left(\frac{\omega(T(I))}{\omega(T(J))}\right)^{1/2}.
\end{align*}
This proves \eqref{eq:frame-element-first-order-overlap} and completes the proof.
\end{proof}

\begin{lemma}\label{lem:frame-element-core}
Let $\omega\in\widehat{\mathcal{D}}$. There exists a constant $c=c(\omega,\chi)>0$ such that, for all
$J\in\mathcal Q_j$, $I\in\mathcal Q_k$, $I\subseteq J$, and $k\geq j+4$, we have
\begin{equation*}\label{eq:frame-element-core}
 |\phi_J(a_I)|\geq\frac{c}{\sqrt{\omega(T(J))}}.
\end{equation*}
\end{lemma}

\begin{proof}
By the identity \eqref{eq:chi-partition} and using $\supp\chi\subset(1/2,2)$, we obtain
$
 \chi(1)^2=1.
$
Since $\chi$ is nonnegative, we have $\chi(1)=1$. By the continuity of $\chi$, there exists $\delta\in(0,1/4]$ such that if $|t-1|<\delta$ then $\chi(t)\geq 1/2.$ Hence,
\begin{equation}\label{eq:cutoff-mass}
 \sum_{n\geq1}\chi(2^{-j}n)
 \geq\frac{\delta2^j}{2}.
\end{equation}
Since $I\subseteq J$, we have
$
d_\T(\theta_I,\theta_J)\leq\frac{|J|}{2}=\frac{\pi}{2^{j+3}}.
$
If $\chi(2^{-j}n)\neq0$, then
$
 2^{j-1}<n<2^{j+1},
$
and 
$
 |n(\theta_I-\theta_J)|
 <\frac{\pi}{4}.
$
It follows that if $\chi(2^{-j}n)\neq0$ then
\begin{equation}\label{eq:frame-core-phase}
 \cos(n(\theta_I-\theta_J))\geq\cos\frac{\pi}{4}=\frac{1}{\sqrt2}.
\end{equation}
Since $r_k=1-2^{-k}$ and $k\geq j+4$, for every $n$ such that $\chi(2^{-j}n)\neq0$,
\begin{align}
 r_k^n
 &\geq r_k^{2^{j+1}}
 =\left(1-2^{-k}\right)^{2^{j+1}}
 \geq1-2^{j+1-k}
 \geq\frac78.
 \label{eq:frame-core-radial}
\end{align}
Furthermore, by \eqref{eq:moment-tent-comparison}, for the $n$ with $\chi(2^{-j}n)\neq0$,
\begin{equation}\label{eq:frame-core-moment}
 \omega_{2n+1} \leq C2^j\omega(T(J)).
\end{equation}
By \eqref{eq:frame-element-definition} and the definition of $a_I$, we have
\[
 \phi_J(a_I)
 =
 \frac{1}{4}\frac{1}{2^{j/2}}
 \sum_{n\geq1}
 \frac{\chi(2^{-j}n)r_k^ne^{in(\theta_I-\theta_J)}}
 {\sqrt{\omega_{2n+1}}}.
\]
Taking the real part and using \eqref{eq:frame-core-phase}, \eqref{eq:frame-core-radial}, and \eqref{eq:frame-core-moment}, we obtain
\begin{align*}
 |\phi_J(a_I)|
 \geq 
 \frac{1}{4}\frac{1}{2^{j/2}}
 \sum_{n\geq1}
 \frac{\chi(2^{-j}n)r_k^n\cos(n(\theta_I-\theta_J))}
 {\sqrt{\omega_{2n+1}}}
 \geq
 \frac{c}{2^{j/2}}
 \sum_{n\geq1}
 \frac{\chi(2^{-j}n)}
 {\sqrt{2^j\omega(T(J))}}.
\end{align*}
By \eqref{eq:cutoff-mass}, we get
\begin{align*}
 |\phi_J(a_I)|
 &\geq
 \frac{c}{2^{j/2}}
 \frac{\delta2^j}
 {2\sqrt{2^j\omega(T(J))}}
 \geq
 \frac{c}{\sqrt{\omega(T(J))}},
\end{align*}
which completes the whole proof.
\end{proof}

\begin{lemma}\label{lem:angular-counting}
Let $p>1$. There exists a constant $C=C(p)>0$ such that,
for all $k\geq j\geq1$ and $\vartheta\in\T$,
\begin{align}
 \sum_{I\in\mathcal Q_k}
 \bigl(1+2^j d_{\T}(\vartheta,\theta_I)\bigr)^{-p}
 &\leq C2^{k-j},
 \label{eq:count-fine}\\
 \sum_{J\in\mathcal Q_j}
 \bigl(1+2^j d_{\T}(\vartheta,\theta_J)\bigr)^{-p}
 &\leq C.
 \label{eq:count-coarse}
\end{align}
Moreover, for every $N\geq1$,
\begin{align}
 \sum_{\substack{I\in\mathcal Q_k\\
  2^j d_{\T}(\vartheta,\theta_I)>N}}
 \bigl(1+2^j d_{\T}(\vartheta,\theta_I)\bigr)^{-p}
 &\leq C2^{k-j}(1+N)^{1-p},
 \label{eq:count-fine-tail}\\
 \sum_{\substack{J\in\mathcal Q_j\\
  2^j d_{\T}(\vartheta,\theta_J)>N}}
 \bigl(1+2^j d_{\T}(\vartheta,\theta_J)\bigr)^{-p}
 &\leq C(1+N)^{1-p}.
 \label{eq:count-coarse-tail}
\end{align}
\end{lemma}

\begin{proof}
For $\ell\geq j\geq 1$ and $E\subset\T$, we have
\begin{equation}\label{eq:dyadic-arc-count}
 \#\{L\in\mathcal Q_\ell:\theta_L\in E\}
 \leq 2+\frac{2^{\ell+3}|E|}{2\pi}.
\end{equation}

Fix $\vartheta\in\T$. For every integer $m\geq0$, define
\[
 \mathcal A_m^{(\ell)}
 =
 \left\{
  L\in\mathcal Q_\ell:
  m\leq 2^j d_{\T}(\vartheta,\theta_L)<m+1
 \right\}.
\]
Observe that
\[
|\{\theta\in\mathbb{T}: m\leq 2^j d_{\T}(\vartheta,\theta)<m+1\}|\leq 2^{1-j}.
\]
By \eqref{eq:dyadic-arc-count}, there exists a constant $C_0$ such that
\begin{align*}
 \#\mathcal A_m^{(\ell)} \leq 2+\frac{2^{\ell+3} 2^{1-j}}{2\pi} \leq C_0\,2^{\ell-j}.
\end{align*}
If $L\in\mathcal A_m^{(\ell)}$, then
$
 \bigl(1+2^j d_{\T}(\vartheta,\theta_L)\bigr)^{-p}
 \leq (1+m)^{-p}.
$
Moreover,
\begin{align*}
 \sum_{L\in\mathcal Q_\ell}
 \bigl(1+2^j d_{\T}(\vartheta,\theta_L)\bigr)^{-p}
 \leq
 \sum_{m=0}^{\infty}
 \#\mathcal A_m^{(\ell)}(1+m)^{-p}\leq
 C_0 2^{\ell-j}
 \sum_{m=0}^{\infty}(1+m)^{-p} \leq C_p2^{\ell-j}.
\end{align*}
Let $\ell=k$, we have the inequality \eqref{eq:count-fine}. The estimate \eqref{eq:count-coarse} follows from $\ell=j$.

For the remaining estimate, let $M=\lfloor N\rfloor$ be the greatest integer not exceeding \(N\).
If $2^j d_{\T}(\vartheta,\theta_L)>N$, then $L\in \mathcal A_m^{(\ell)}$ for some $m\geq M$.
Therefore,
\begin{align*}
 \sum_{\substack{L\in\mathcal Q_\ell\\
  2^j d_{\T}(\vartheta,\theta_L)>N}}\bigl(1+2^j d_{\T}(\vartheta,\theta_L)\bigr)^{-p} 
  \leq C_0 2^{\ell-j} \sum_{m=M}^{\infty}(1+m)^{-p}\leq C_p2^{\ell-j}(1+N)^{1-p},
\end{align*}
where the last inequality is due to the facts
\begin{align*}
 \sum_{m=M}^{\infty}(1+m)^{-p}
 \leq
 \left(1+\frac1{p-1}\right)(1+M)^{1-p}
 \leq2^{p-1}\left(1+\frac1{p-1}\right)(1+N)^{1-p}.
\end{align*}
Taking $\ell=k$ and $\ell=j$, we have \eqref{eq:count-fine-tail} and \eqref{eq:count-coarse-tail}. This completes the proof.
\end{proof}
\section{The diagonal block estimate}\label{sec:diag-estimate}
\subsection{Boundary behavior of the Berezin transform}
For $a\in\D\setminus\{0\}$, write $t_a=1-|a|$ and $ \zeta_a=\frac{a}{|a|}.$
For $0<\alpha<\beta$, define
\begin{align}\label{eq:def-r}
 \mathcal R_{\alpha,\beta} = \set{\lambda\in\C:\alpha\leq\operatorname{Re}\lambda\leq\beta,\ |\operatorname{Im}\lambda|\leq\beta}
\end{align}
and
\begin{align}\label{eq:def-w}
 \mathcal W_{\alpha,\beta}(a) = \set{(1-t_a\lambda)\zeta_a: \lambda\in\mathcal R_{\alpha,\beta}}.
\end{align}
For each fixed $0<\alpha<\beta$, we have $\mathcal W_{\alpha,\beta}(a)\subset\D$ whenever $|a|$ is sufficiently close to $1$.
Then, we define
\[
 \Phi_{\alpha,\beta}(a)
 =
 \sup_{b,c\in\mathcal W_{\alpha,\beta}(a)}
 \left|
 \left\langle A_uWk_b^\omega,Wk_c^\omega
 \right\rangle_{\ell^2(\mathcal Q)}
 \right|.
\]

The following lemma is usually formulated in the literature in terms of Toeplitz operators on pseudohyperbolic disks, one can consult \cite[Proposition 2.7]{IMW15}. 
\begin{lemma}\label{lem:affine-polarization-Dhat}
Let $\omega\in\widehat{\mathcal D}$ and $u\in L^\infty(\D)$. Suppose that $\widetilde u_\omega(a)\to0$ as $|a|\to1^-$. For fixed
$0<\alpha<\beta$, we have
$$\lim_{|a|\to 1^-}\Phi_{\alpha,\beta}(a)=0.$$
\end{lemma}

\begin{proof}
Fix $0<\alpha<\beta$. For $a\in\D\setminus\{0\}$, define
$
 z_a(\lambda)=(1-t_a\lambda)\zeta_a.
$
By definitions of \eqref{eq:def-r} and \eqref{eq:def-w}, we have
$
 \mathcal W_{\alpha,\beta}(a)=z_a(\mathcal R_{\alpha,\beta}).
$
Consider the open bounded rectangle
\[
 \Omega = \set{\lambda\in\C: \frac{\alpha}{2}<\operatorname{Re}\lambda<2\beta,\ |\operatorname{Im}\lambda|<2\beta}.
\]
Then $\mathcal R_{\alpha,\beta}\Subset\Omega$ and
$\overline\Omega\subset\set{\lambda\in\C:\operatorname{Re}\lambda>0}$.

For $\lambda\in\Omega$, a direct computation gives
\begin{equation*}\label{eq:affine-boundary-identity-Dhat}
 1-|z_a(\lambda)|^2 =t_a\bigl(2\operatorname{Re}\lambda-t_a|\lambda|^2\bigr).
\end{equation*}
Since $|\lambda|^2<8\beta^2$ on $\Omega$, choose $a_0\in(0,1)$ so that
$8\beta^2t_a\leq\frac{\alpha}{2}$ for all $a_0<|a|<1$.
Then $z_a(\Omega)\subset\D$ for $|a|>a_0$, and
\begin{equation}\label{eq:affine-boundary-distance-Dhat}
 \frac{\alpha}{4}(1-|a|)
 \leq1-|z_a(\lambda)|
 \leq4\beta(1-|a|),
 \qquad \lambda\in\Omega.
\end{equation}

We next take an integer $N\geq1$ with $2^N\geq\max\set{1,\frac4\alpha,4\beta}$. If $0\leq r\leq s<1$ and $1-s\geq2^{-N}(1-r)$, then
\begin{equation*}\label{eq:Dhat-comparable-boundary-scales}
 \widehat\omega(s)
 \leq\widehat\omega(r)
 \leq C_\omega^N\widehat\omega(s).
\end{equation*}
In particular, there is a constant $C_1=C_1(\omega,\alpha,\beta)\geq1$ such that
\begin{equation}\label{eq:affine-tail-comparison-Dhat}
 C_1^{-1}\widehat\omega(|a|)
 \leq\widehat\omega(|z_a(\lambda)|)
 \leq C_1\widehat\omega(|a|),
 \qquad |a|>a_0,\quad\lambda\in\Omega.
\end{equation}
Combining Lemma~\ref{lem:norm-estimates-for-reproducing-kernel} with \eqref{eq:affine-boundary-distance-Dhat} and
\eqref{eq:affine-tail-comparison-Dhat}, we obtain a constant $C_2=C_2(\omega,\alpha,\beta)\geq1$ such that
\begin{equation}\label{eq:affine-kernel-comparison-Dhat}
 C_2^{-1}B_\omega(a,a)
 \leq B_\omega(z_a(\lambda),z_a(\lambda))
 \leq C_2B_\omega(a,a),
 \qquad |a|>a_0,\quad\lambda\in\Omega.
\end{equation}
For $|a|>a_0$ and $\lambda,\mu\in\Omega$, define
\begin{equation}\label{eq:affine-Fa-integral-Dhat}
 F_a(\lambda,\mu)
 =
 \frac{1}{B_\omega(a,a)}
 \int_\D
 u(z)B_{z_a(\lambda)}^\omega(z)
 \overline{B_{z_a(\mu)}^\omega(z)}
 \omega(z)\dd A(z).
\end{equation}
Then, $F_a$ is anti-holomorphic in $\lambda$ and holomorphic in $\mu$. Moreover, since $W$ is an isometry,
\eqref{eq:Au} and $W^*W=I_{A^2_\omega}$ give
\begin{align}
 F_a(\lambda,\mu)
 &=
 \frac{
 \sqrt{B_\omega(z_a(\lambda),z_a(\lambda))
 B_\omega(z_a(\mu),z_a(\mu))}
 }{B_\omega(a,a)}
 \left\langle
 A_uWk_{z_a(\lambda)}^\omega,
 Wk_{z_a(\mu)}^\omega
 \right\rangle_{\ell^2(\mathcal Q)}.
 \label{eq:affine-Fa-matrix-Dhat}
\end{align}
By Cauchy--Schwarz and \eqref{eq:affine-kernel-comparison-Dhat}, for $\lambda,\mu\in\Omega$,
\begin{align} \label{eq:affine-normal-family-bound-Dhat}
 |F_a(\lambda,\mu)|
 \leq
 \frac{\norm{u}_{L^\infty(\D)}}{B_\omega(a,a)}
 \norm{B_{z_a(\lambda)}^\omega}_{A^2_\omega}
 \norm{B_{z_a(\mu)}^\omega}_{A^2_\omega}
 \leq C_2\norm{u}_{L^\infty(\D)}.
\end{align}
On the diagonal, \eqref{eq:Berezin-transform} and \eqref{eq:affine-Fa-integral-Dhat} imply
\begin{equation}\label{eq:affine-diagonal-identity-Dhat}
 F_a(\lambda,\lambda)
 =
 \frac{B_\omega(z_a(\lambda),z_a(\lambda))}{B_\omega(a,a)}
 \widetilde u_\omega(a) (z_a(\lambda)).
\end{equation}
By \eqref{eq:affine-boundary-distance-Dhat}, $|z_a(\lambda)|\to1$ uniformly for $\lambda\in\Omega$ as
$|a|\to1^-$. Hence, \eqref{eq:affine-kernel-comparison-Dhat} and \eqref{eq:affine-diagonal-identity-Dhat} imply
\begin{equation}\label{eq:affine-diagonal-zero-Dhat}
 \lim_{|a|\to 1^-}\sup_{\lambda\in\Omega}|F_a(\lambda,\lambda)|=0.
\end{equation}
We claim that
\begin{equation}\label{eq:affine-off-diagonal-zero-Dhat}
\lim_{|a|\to 1^-} \sup_{\lambda,\mu\in\mathcal R_{\alpha,\beta}}|F_a(\lambda,\mu)|=0.
\end{equation}
Suppose otherwise. Then there are $\varepsilon_0>0$, a sequence $a_n\in\D$ with $|a_n|\to1^-$, and points
$\lambda_n,\mu_n\in\mathcal R_{\alpha,\beta}$ such that for $n\geq 1$ $|F_{a_n}(\lambda_n,\mu_n)|\geq\varepsilon_0.$
Since $\mathcal R_{\alpha,\beta}$ is compact, after passing to a subsequence we may assume that
$\lambda_n\to\lambda_0$ and $\mu_n\to\mu_0$ for some $\lambda_0,\mu_0\in\mathcal R_{\alpha,\beta}$ as $n\to\infty$. Since that $\Omega$ is invariant under complex conjugation, the functions
\[
 G_a(\xi,\mu)=F_a(\overline\xi,\mu),
 \qquad \xi,\mu\in\Omega,
\]
are holomorphic on $\Omega\times\Omega$. By \eqref{eq:affine-normal-family-bound-Dhat}, they form a normal family. By Montel's theorem, there is a subsequence and a holomorphic function $G$ on $\Omega\times\Omega$ such that
$G_{a_n}\to G$ locally uniformly as $n\to\infty$. Set
$$
 F(\lambda,\mu)=G(\overline\lambda,\mu).
$$
Then $F$ is anti-holomorphic in its first variable and holomorphic in its second variable, and  $F_{a_n}\to F$ locally uniformly on $\Omega\times\Omega$ as $n\to\infty$. Since $\mathcal R_{\alpha,\beta}\times\mathcal R_{\alpha,\beta}$ is compactly contained in $\Omega\times\Omega$, it follows that $|F(\lambda_0,\mu_0)|\geq\varepsilon_0.$ On the other hand, \eqref{eq:affine-diagonal-zero-Dhat} implies
\begin{equation}\label{eq:affine-limit-diagonal-zero-Dhat}
 F(\lambda,\lambda)=0,
 \qquad \lambda\in\Omega.
\end{equation}
Next, we show that $F\equiv0$. Fix $\lambda_*\in\Omega$ and choose $r_0>0$ so that
$
 \set{\lambda_*+z:|z|\leq r_0}\subset\Omega.
$
Near $(\lambda_*,\lambda_*)$, the function $F$ has an expansion
\[
 F(\lambda_*+z,\lambda_*+w)
 =\sum_{p,q\geq0}c_{p,q}\,\overline z^{\,p}w^q,
 \qquad |z|,|w|<r_0.
\]
Writing $w=z=re^{i\theta}$ and using \eqref{eq:affine-limit-diagonal-zero-Dhat}, we obtain $\sum_{p,q\geq0}c_{p,q}r^{p+q}e^{i(q-p)\theta}=0$ for all  $0<r<r_0$.
In particular, we have
\[
 \sum_{q-p=d}c_{p,q}r^{p+q}=0,
 \qquad 0<r<r_0.
\]
If $d\geq0$, this identity is
$
 r^d\sum_{p\geq0}c_{p,p+d}r^{2p}=0.
$
Therefore, we have $c_{p,p+d}=0$ for every $p\geq0$. If $d=-\ell<0$, then
$
 r^\ell\sum_{q\geq0}c_{q+\ell,q}r^{2q}=0,
$
and hence $c_{q+\ell,q}=0$ for every $q\geq0$. Thus all coefficients $c_{p,q}$ vanish, and $F$ vanishes on a neighborhood of
$(\lambda_*,\lambda_*)$. Hence, we have $F\equiv0$ on $\Omega\times\Omega$ which is contradicted to $|F(\lambda_0,\mu_0)|\geq\varepsilon_0$. This proves \eqref{eq:affine-off-diagonal-zero-Dhat}. Finally, if $b=z_a(\lambda),$ and $c=z_a(\mu)$ for some  $\lambda,\mu\in\mathcal R_{\alpha,\beta}$, then \eqref{eq:affine-Fa-matrix-Dhat} implies
\[
 \left\langle A_uWk_b^\omega,Wk_c^\omega\right\rangle_{\ell^2(\mathcal Q)}
 =
 \frac{B_\omega(a,a)}
 {\sqrt{B_\omega(b,b)B_\omega(c,c)}}F_a(\lambda,\mu).
\]
By \eqref{eq:affine-kernel-comparison-Dhat},
\[
 0\leq\Phi_{\alpha,\beta}(a)
 \leq C_2
 \sup_{\lambda,\mu\in\mathcal R_{\alpha,\beta}}|F_a(\lambda,\mu)|.
\]
By \eqref{eq:affine-off-diagonal-zero-Dhat}, we complete the whole proof.
\end{proof}

\subsection{The diagonal estimate}
For fixed integers $M,L\geq1$, define
\begin{equation*}\label{eq:eta-local-band}
 \tau(N;M,L):= \sup\left\{
 |A_u(I,J)|:
 \begin{array}{c}
 I\in\mathcal Q_k,\ J\in\mathcal Q_j,\ j,k\geq N,\\
 |j-k|\leq M,\ d_{j\wedge k}(I,J)\leq L
 \end{array}
 \right\}.
\end{equation*}

\begin{lemma}\label{lem:local-band-entries-Dhat}
Assume that
$
 \lim_{|a|\to1^-}\widetilde u_\omega(a)=0 .
$ For every fixed $M,L\geq1$, then
\begin{equation*}\label{eq:local-band-entries-Dhat}
 \lim_{N\to\infty}\tau(N;M,L)=0 .
\end{equation*}
\end{lemma}

\begin{proof}
Fix $M,L,N\geq1$, and let $I\in\mathcal Q_k$ and
$J\in\mathcal Q_j$ satisfy
\[
 j,k\geq N,\qquad |j-k|\leq M,\qquad
 d_{j\wedge k}(I,J)\leq L .
\]
By Proposition~\ref{prop:frame-element-kernel-representation},
\begin{align*}
 A_u(I,J)
 &=
 \frac1{(2\pi)^2}
 \iint_{[-\pi,\pi]^2}
 L_j(t)\overline{L_k(s)}
 \left\langle
 A_uWk_{r_je^{i(\theta_J+t)}}^\omega,
 Wk_{r_ke^{i(\theta_I+s)}}^\omega
 \right\rangle_{\ell^2(\mathcal Q)}
 \dd t\dd s .
\end{align*}

Fix $R\geq1$ and let
\[
 E=\{(t,s)\in[-\pi,\pi]^2:
 |t|\leq R2^{-j},\ |s|\leq R2^{-k}\}.
\]
By $E^c\subset\{|t|>R2^{-j}\}\cup\{|s|>R2^{-k}\}$, \eqref{eq:Lj-L1} and \eqref{eq:Lj-tail}, 
\[
 \iint_{E^c}|L_j(t)||L_k(s)|\,\dd t\dd s
 \leq C(1+R)^{-2}.
\]
Since $\|A_u\|\leq\|u\|_{L^\infty(\D)}$, we have
\[
 \frac1{(2\pi)^2}
 \bigg|\iint_{E^c}
 L_j(t)\overline{L_k(s)}
 \left\langle
 A_uWk_{r_je^{i(\theta_J+t)}}^\omega,
 Wk_{r_ke^{i(\theta_I+s)}}^\omega
 \right\rangle_{\ell^2(\mathcal Q)}
 \dd t\dd s\bigg| \leq C\|u\|_{L^\infty(\D)}(1+R)^{-2}.
\]

For $(t,s)\in E$, let $z_{J,t}=r_je^{i(\theta_J+t)}$ and $z_{I,s}=r_ke^{i(\theta_I+s)}$. Also set $\lambda_J(t)=2^j(1-r_je^{it})$ and $\lambda_I(s)=2^j(1-r_ke^{i(\theta_I-\theta_J+s)}).$ Since $1-r_j=2^{-j}$, we have
\begin{equation*}\label{eq:affine-kernel-point-representation-Dhat}
 z_{J,t}=(1-2^{-j}\lambda_J(t))e^{i\theta_J}
 \quad \text{and} \quad
 z_{I,s}=(1-2^{-j}\lambda_I(s))e^{i\theta_J}.
\end{equation*}
For $|t|\leq R2^{-j}$, observe
\[
 \operatorname{Re}\lambda_J(t)
 =2^j(1-r_j\cos t)\geq1,
 \qquad
 |\lambda_J(t)|\leq1+2^j|t|\leq1+R .
\]
Writing $\delta=\theta_I-\theta_J+s$, we obtain
\[
 \operatorname{Re}\lambda_I(s)
 =2^j(1-r_k\cos\delta)
 \geq2^j(1-r_k)=2^{j-k}\geq2^{-M},
\]
while
\[
 2^jd_\T(\theta_I+s,\theta_J)
 \leq2^jd_\T(\theta_I,\theta_J)+2^j|s|
 \leq2^{j-(j\wedge k)}L+2^{j-k}R
 \leq2^M(L+R).
\]
Hence,
\[
 |\lambda_I(s)|
 \leq2^j(1-r_k)+2^j|1-e^{i\delta}|
 \leq2^{j-k}+2^jd_\T(\theta_I+s,\theta_J)
 \leq2^M(1+L+R).
\]
Therefore,
\begin{equation*}\label{eq:affine-parameters-fixed-region-Dhat}
 \lambda_J(t),\lambda_I(s)
 \in\mathcal R_{2^{-M-1},\,2^M(1+L+R)} .
\end{equation*}
For fixed $M,L,R$, if $N$ is sufficiently large, then
\[
 \mathcal W_{2^{-M-1},\,2^M(1+L+R)}(a)\subset\D,
 \qquad
 a\in\D,\quad1-2^{-N}\leq |a|<1 .
\]
Since $t_{a_J}=2^{-j}$, $\zeta_{a_J}=e^{i\theta_J}$, and
$j\geq N$, we have $1-|a_J|\leq2^{-N}$. Thus,
\[
 z_{J,t},z_{I,s}
 \in
 \mathcal W_{2^{-M-1},\,2^M(1+L+R)}(a_J).
\]
Consequently,
\[
 \left|
 \left\langle
 A_uWk_{z_{J,t}}^\omega,Wk_{z_{I,s}}^\omega
 \right\rangle_{\ell^2(\mathcal Q)}
 \right|
 \leq
 \sup_{\substack{a\in\D\\1-2^{-N}\leq|a|<1}}
 \Phi_{2^{-M-1},\,2^M(1+L+R)}(a).
\]
Using \eqref{eq:Lj-L1}, we obtain
\begin{multline*}
  \frac1{(2\pi)^2}
 \bigg|\iint_{E}
 L_j(t)\overline{L_k(s)}
 \left\langle
 A_uWk_{r_je^{i(\theta_J+t)}}^\omega,
 Wk_{r_ke^{i(\theta_I+s)}}^\omega
 \right\rangle_{\ell^2(\mathcal Q)}
 \dd t\dd s\bigg|\leq \\
 C
 \sup_{\substack{a\in\D:1-2^{-N}\leq|a|<1}}
 \Phi_{2^{-M-1},\,2^M(1+L+R)}(a).
\end{multline*}
Taking the supremum over all possible $I,J$, we obtain
\[
 \tau(N;M,L)
 \leq
 C
 \sup_{\substack{a\in\D:1-2^{-N}\leq|a|<1}}
 \Phi_{2^{-M-1},\,2^M(1+L+R)}(a)
 +
 C\|u\|_{L^\infty(\D)}(1+R)^{-2}.
\]
By Lemma~\ref{lem:affine-polarization-Dhat} and
$
 \lim_{|a|\to1^-}\widetilde u_\omega(a)=0,
$
we have
\[
 \lim_{N\to\infty}
 \sup_{\substack{a\in\D:1-2^{-N}\leq|a|<1}}
 \Phi_{2^{-M-1},\,2^M(1+L+R)}(a)=0 .
\]
Therefore,
\[
 \limsup_{N\to\infty}\tau(N;M,L)
 \leq
 C\|u\|_{L^\infty(\D)}(1+R)^{-2}.
\]
Letting $R\to\infty$, we have
\[
 \lim_{N\to\infty}\tau(N;M,L)=0 .
\]
This completes the proof.
\end{proof}

\begin{proposition}\label{lem:near-diagonal-scalar-tail}
Let $\omega\in\widehat{\mathcal D}$ and $u\in L^\infty(\D)$. If the Berezin transform $\lim_{|a|\to 1^-}\widetilde u_\omega(a)=0 $,
then
\begin{align} \label{eq:local-matrix-scalar}
 \norm{A_{N,1}^{M,L}(u)} \le 2^{M+6}LM \tau(N;M,L).
\end{align}
Hence, for any fixed $M$ and $L$, we have
\begin{align*}\label{eq:normnear1-zero} 
\lim_{N\to\infty}  \norm{A_{N,1}^{M,L}(u)} =0.
\end{align*}
\end{proposition}

We begin with an estimate for the number of nonzero entries in each row and column of a band matrix.

\begin{lemma}\label{lem:local-band-degree-Dhat}
Fix integers $M,L\geq1$. Let $B=(B(I,J))_{I,J\in\mathcal Q}$ be a matrix such that $B(I,J)\neq0$ only when
$I\in\mathcal Q_k$, $J\in\mathcal Q_j$, $j,k\geq1$, $|k-j|\leq M$, and $d_{j\wedge k}(I,J)\leq L$.
Then each row and each column of $B$ contains at most
$
 2^{M+6}LM
$
nonzero entries.
\end{lemma}

\begin{proof}
Fix $J\in\mathcal Q_j$. For each $k$ satisfying $|k-j|\leq M$, the condition $d_{j\wedge k}(I,J)\leq L$ gives
\[
 d_{\mathbb T}(\theta_I,\theta_J)
 \leq L2^{-\min\{j,k\}} .
\]
Since the angular spacing of $\{\theta_I:I\in\mathcal Q_k\}$ is $\frac{2\pi}{2^{k+3}}$, the number of possible indices
$I\in\mathcal Q_k$ is at most
\[
 2+\frac{2L2^{-\min\{j,k\}}}{\pi/(4\cdot2^k)}
 \leq 2+\frac8\pi L2^{k-\min\{j,k\}}
 \leq 2^{M+4}L,
\]
where we have used $k-\min\{j,k\}\leq M$. Since there are at most $3M$ integers $k$ satisfying $|k-j|\leq M$, the column indexed by
$J$ contains at most
$
 3M\cdot2^{M+4}L\leq2^{M+6}LM
$
nonzero entries. The same argument, with the roles of $I$ and $J$
interchanged, gives the same bound for every row. This completes the
proof.
\end{proof}

Classical Schur’s test will be used repeatedly (see, for example, \cite[Theorem~5.2]{HS78}).

\begin{lemma}\label{lem:Schur}
Let $B=(B_{i,j})$ be a matrix satisfying
\[
 R_B=\sup_i\sum_j|B_{i,j}|<\infty,
 \qquad
 C_B=\sup_j\sum_i|B_{i,j}|<\infty.
\]
Then $B$ is bounded on $\ell^2$, and
\begin{equation*}\label{eq:Schur}
 \norm{B}_{\ell^2\to\ell^2} \leq \sqrt{R_BC_B}.
\end{equation*}
\end{lemma}

\begin{proof}[Proof of Proposition~\ref{lem:near-diagonal-scalar-tail}]
By the definition of $\tau(N;M,L)$, every nonzero entry of
$A_{N,1}^{M,L}(u)$ satisfies
\[
 \bigl|A_{N,1}^{M,L}(u)(I,J)\bigr|
 \leq \tau(N;M,L).
\]
By Lemma~\ref{lem:local-band-degree-Dhat}, each row and each column of
$A_{N,1}^{M,L}(u)$ contains at most
$
 2^{M+6}LM
$
nonzero entries. Then, for every fixed $I\in\mathcal Q$,
\[
 \sum_{J\in\mathcal Q}
 \bigl|A_{N,1}^{M,L}(u)(I,J)\bigr|
 \leq
 2^{M+6}LM\,\tau(N;M,L).
\]
Hence,
\[
 \sup_{I\in\mathcal Q}
 \sum_{J\in\mathcal Q}
 \bigl|A_{N,1}^{M,L}(u)(I,J)\bigr|
 \leq
 2^{M+6}LM\,\tau(N;M,L).
\]
Similarly,
\[
 \sup_{J\in\mathcal Q}
 \sum_{I\in\mathcal Q}
 \bigl|A_{N,1}^{M,L}(u)(I,J)\bigr|
 \leq
 2^{M+6}LM\,\tau(N;M,L).
\]
Using Schur's test from Lemma~\ref{lem:Schur}, we obtain
\[
 \bigl\|A_{N,1}^{M,L}(u)\bigr\|
 \leq
 2^{M+6}LM\,\tau(N;M,L).
\]
This proves \eqref{eq:local-matrix-scalar}. Finally, for fixed
$M$ and $L$, Lemma~\ref{lem:local-band-entries-Dhat} gives
\[
 \lim_{N\to\infty}\tau(N;M,L)=0.
\]
Combining this with \eqref{eq:local-matrix-scalar}, we conclude that
\[
 \lim_{N\to\infty}
 \bigl\|A_{N,1}^{M,L}(u)\bigr\|=0 .
\]
This completes the proof.
\end{proof}

\begin{lemma}\label{lem:near-diagonal-scalar-tail-2}
Let $\omega\in\widehat{\mathcal D}$ and $u\in L^\infty(\D)$. There
exists a constant $C=C(\omega,\chi)>0$, independent of $N$, $M$, and
$L$, such that
\begin{equation}\label{eq:angular-matrix-scalar}
 \norm{A_{N,2}^{M,L}}
 \leq
 C2^{M/2}(1+L)^{-2}\norm{u}_{L^\infty(\D)} .
\end{equation}
\end{lemma}

\begin{proof}
The nonzero entries of $A_{N,2}^{M,L}$ are indexed by $I\in\mathcal Q_k$, $J\in\mathcal Q_j$ satisfying
\[
 j,k\geq N,\qquad |k-j|<M,\qquad d_{j\wedge k}(I,J)>L .
\]
First assume that $k\geq j$ and write $k=j+d, 0\leq d<M$. For $I\in\mathcal Q_{j+d}, J\in\mathcal Q_j$, by Lemma~\ref{lem:frame-element-overlap},
\[
 |A_u(I,J)|
 \leq
 C\|u\|_{L^\infty(\D)}
 \left(\frac{\omega(T(I))}{\omega(T(J))}\right)^{1/2}
 (1+d_j(I,J))^{-3}.
\]
By \eqref{eq:tent-volume-ratio},
$
 \omega(T(I))\leq2^{1-d}\omega(T(J)),
$
and therefore
\begin{equation*}\label{eq:angular-entry-bound}
 |A_u(I,J)| \leq C\|u\|_{L^\infty(\D)}2^{-d/2}(1+d_j(I,J))^{-3}.
\end{equation*}
Using Lemma~\ref{lem:angular-counting}, we have
\[
 \sum_{\substack{I\in\mathcal Q_{j+d}:d_j(I,J)>L}}
 (1+d_j(I,J))^{-3}
 \leq C2^d(1+L)^{-2},
\]
and
\[
 \sum_{\substack{J\in\mathcal Q_j:d_j(I,J)>L}}
 (1+d_j(I,J))^{-3}
 \leq C(1+L)^{-2}.
\]
Therefore,
\begin{align}
 \sup_{J\in\mathcal Q_j}
 \sum_{\substack{I\in\mathcal Q_{j+d}:d_j(I,J)>L}}
 |A_u(I,J)|
 &\leq
 C\|u\|_{L^\infty(\D)}
 2^{d/2}(1+L)^{-2},
 \label{eq:angular-column-level-bound}\\
 \sup_{I\in\mathcal Q_{j+d}}
 \sum_{\substack{J\in\mathcal Q_j:d_j(I,J)>L}}
 |A_u(I,J)|
 &\leq
 C\|u\|_{L^\infty(\D)}
 2^{-d/2}(1+L)^{-2}.
 \label{eq:angular-row-level-bound}
\end{align}

Now assume that $k<j$ and write $j=k+m,1\leq m<M$. By $d_{j\wedge k}(I,J)=d_k(I,J)$ and Lemma~\ref{lem:frame-element-overlap},
\[
 |A_u(I,J)|
 \leq
 C\|u\|_{L^\infty(\D)}
 \left(\frac{\omega(T(J))}{\omega(T(I))}\right)^{1/2}
 (1+d_k(I,J))^{-3}.
\]
Since
$
 \omega(T(J))\leq2^{1-m}\omega(T(I)),
$
we obtain
\begin{equation*}\label{eq:angular-entry-bound-negative}
 |A_u(I,J)| \leq C\|u\|_{L^\infty(\D)}2^{-m/2}(1+d_k(I,J))^{-3}.
\end{equation*}
By Lemma~\ref{lem:angular-counting},
\[
 \sum_{\substack{I\in\mathcal Q_k:d_k(I,J)>L}}
 (1+d_k(I,J))^{-3}
 \leq C(1+L)^{-2},
\]
and
\[
 \sum_{\substack{J\in\mathcal Q_{k+m}:d_k(I,J)>L}}
 (1+d_k(I,J))^{-3}
 \leq C2^m(1+L)^{-2}.
\]
Thus,
\begin{align}
 \sup_{J\in\mathcal Q_{k+m}}
 \sum_{\substack{I\in\mathcal Q_k:d_k(I,J)>L}}
 |A_u(I,J)|
 &\leq
 C\|u\|_{L^\infty(\D)}
 2^{-m/2}(1+L)^{-2},
 \label{eq:angular-column-level-bound-negative}\\
 \sup_{I\in\mathcal Q_k}
 \sum_{\substack{J\in\mathcal Q_{k+m}:d_k(I,J)>L}}
 |A_u(I,J)|
 &\leq
 C\|u\|_{L^\infty(\D)}
 2^{m/2}(1+L)^{-2}.
 \label{eq:angular-row-level-bound-negative}
\end{align}

Combining \eqref{eq:angular-column-level-bound} and
\eqref{eq:angular-column-level-bound-negative}, we get
\[
 \sup_{J\in\mathcal Q}
 \sum_{I\in\mathcal Q}|A_{N,2}^{M,L}(I,J)|
 \leq
 C2^{M/2}(1+L)^{-2}\|u\|_{L^\infty(\D)}.
\]
Similarly,
\[
 \sup_{I\in\mathcal Q}
 \sum_{J\in\mathcal Q}|A_{N,2}^{M,L}(I,J)|
 \leq
 C2^{M/2}(1+L)^{-2}\|u\|_{L^\infty(\D)}.
\]
The desired estimate \eqref{eq:angular-matrix-scalar} now follows from
Lemma~\ref{lem:Schur}. This completes the proof.
\end{proof}

\section{The estimate of upper and lower blocks}\label{sec:low-uppwe-estimate}

\begin{lemma}\label{lem:adjoint-triangular-symmetry}
For every $N\ge2$ and $M\ge4$,
\begin{equation}\label{eq:triangular-adjoint-symmetry}
 A_{N,3}^{M}(u)=A_{N,4}^{M}(\overline u)^*.
\end{equation}
\end{lemma}

\begin{proof}
By \eqref{eq:Toeplitz-frame-entries},
$
 A_{\overline u}(I,J) =\overline{A_u(J,I)}.
$
The adjoint of the region $j\ge k+M$ is exactly the region $k\ge j+M$, which proves \eqref{eq:triangular-adjoint-symmetry}. This completes the whole proof.
\end{proof}

Since \(\widetilde{\bar u}_\omega=\overline{\widetilde u_\omega}\), it suffices to estimate the upper-triangular block $A_{N,4}^{M}(u)$.

\begin{proposition}\label{prop:upper-compact}
Let $\omega\in\widehat{\mathcal D}$ and $u\in L^\infty(\D)$. If $\lim_{|a|\to 1^-}\widetilde u_\omega(a)=0 $,
then
\begin{equation*}\label{eq:upper-double-limit-scalar}
 \lim_{M\to\infty}\limsup_{N\to\infty} \norm{A_{N,4}^{M}(u)}_{\mathcal L(\ell^2(\mathcal Q))}=0.
\end{equation*}
\end{proposition}

For $J\in\mathcal Q_j$, $j\ge1$, we define
\begin{equation*}\label{eq:upper-coefficient-cJ}
 c_J=\int_\D u(z)\phi_J(z)\omega(z)\dd A(z).
\end{equation*}
For integers $N\geq2, M\geq4$, set
\begin{equation*}\label{eq:upper-shifted-Carleson-tail}
 \mathfrak C_N^M(u) = \sup_{\substack{K\in\mathcal Q_k}:k\geq N}
 \frac1{\omega(T(K))}
 \sum_{j\geq k+M}
 \sum_{\substack{J\in\mathcal Q_j:J\subseteq K}}
 |c_J|^2.
\end{equation*}

The proof of Proposition~\ref{prop:upper-compact} follows from Propositions~\ref{lem:upper-triangular-Carleson-scalar} and \ref{cor:scalar-Carleson-tail} below.
\begin{proposition}\label{lem:upper-triangular-Carleson-scalar}
Let $\omega\in\widehat{\mathcal D}$ and $u\in L^\infty(\D)$.  There exists $C=C(\omega,\chi)>0$ such that, for every
$N\ge2$ and $M\ge4$,
\begin{equation*}\label{eq:upper-total-scalar}
 \norm{A_{N,4}^{M}(u)}_{\mathcal L(\ell^2(\mathcal Q))}
 \le C\mathfrak C_N^M(u)^{1/2}
 +C2^{-M}\norm{u}_{L^\infty(\D)}.
\end{equation*}
\end{proposition}

\begin{proposition}\label{cor:scalar-Carleson-tail}
Let $\omega\in\widehat{\mathcal D}$ and $u\in L^\infty(\D)$. If $\lim_{|a|\to 1^-}\widetilde u_\omega(a)=0 $,
then
\begin{equation}\label{eq:Carleson-tail-double-limit}
 \lim_{M\to\infty}\limsup_{N\to\infty}\mathfrak C_N^M(u)=0.
\end{equation}
\end{proposition}

\subsection{A Carleson embedding estimate}
For $x=(x_I)_{I\in\mathcal Q}\in\ell^2(\mathcal Q)$ and integers $m\ge N\ge1$, define a linear operator (band-limited operator) between
$\ell^2(\mathcal{Q})$ and $A_\omega^2$ by
\begin{equation*}\label{eq:truncated-frame-synthesis}
 \mathcal S_{N,m}x= \sum_{\ell=N}^{m}\sum_{I\in\mathcal Q_\ell}x_I\phi_I .
\end{equation*}
We set $\mathcal S_{N,m}x=0$ when $m<N$. For $j\geq 0$, recall that $\Pi_j$ denote the orthogonal projection of $\ell^2(\mathcal Q)$ onto $\ell^2(\mathcal Q_j)$. For integers $m\geq N\geq1$, set $\Pi_{[N,m]}=\sum_{\ell=N}^{m}\Pi_\ell$. Then $\mathcal S_{N,m}x= W^*\Pi_{[N,m]}x.$ Moreover, by the definition of $\phi_I$,
\[
 (\mathcal S_{N,m}x)(z)
 =
 \sum_{n\geq1}
 \left[
  \sum_{\ell=N}^{m}
  \frac{\chi(2^{-\ell}n)}{\sqrt{2^{\ell+3}}}
  \sum_{I\in\mathcal Q_\ell}x_Ie^{-in\theta_I}
 \right]e_n(z).
\]
Since $\operatorname{supp}\chi\subset(1/2,2)$, the factor
$\chi(2^{-\ell}n)$ can be nonzero only when
$
 2^{\ell-1}<n<2^{\ell+1}.
$
Therefore,
\[
 \operatorname{Ran}(\mathcal S_{N,m})
 \subseteq
 \operatorname{span}
 \left\{e_n:2^{N-1}<n<2^{m+1}\right\}.
\]
If $x=Wf$ for some $f\in A_\omega^2$, then 
\[
 \mathcal S_{N,m}Wf
 =
 \sum_{n\geq1}
 \left(
  \sum_{\ell=N}^{m}\chi(2^{-\ell}n)^2
 \right)
 \langle f,e_n\rangle_{A_\omega^2}e_n.
\]
When $m<N$, we understand both $\Pi_{[N,m]}$ and $\mathcal S_{N,m}$ to be zero.

\begin{proposition}\label{lem:shifted-Carleson}
Fix $N\geq1$ and $M\geq4$. Let $(b_J)_{J\in\mathcal Q_j,\ j\geq N+M}$ be a sequence of complex
numbers, and suppose that
\begin{equation}\label{eq:shifted-Carleson-condition}
 \norm{b}_{\mathrm C;N,M}^2
 =
 \sup_{\substack{I\in\mathcal Q_k}:k\geq N}
 \frac1{\omega(T(I))}
 \sum_{j\geq k+M}
 \sum_{\substack{J\in\mathcal Q_j: J\subseteq I}}
 |b_J|^2
 <\infty.
\end{equation}
Then there exists a constant $C=C(\omega,\chi)>0$, independent of
$N$ and $M$, such that
\begin{equation}\label{eq:shifted-Carleson-embedding}
 \sum_{j\geq N+M}
 \sum_{J\in\mathcal Q_j}
 |b_J|^2
 \left|
  (\mathcal S_{N,j-M}x)(a_J)
 \right|^2
 \leq
 C\norm{b}_{\mathrm C;N,M}^2
 \norm{x}_{\ell^2(\mathcal Q)}^2.
\end{equation}
\end{proposition}

For the proof of Proposition~\ref{lem:shifted-Carleson}, we need to define another operator. For an integer $m\ge1$, define
\begin{equation}\label{eq:smooth-low-pass}
 \mathcal C_mf
 =
 \ip f{e_0}e_0
 +
 \sum_{n=1}^{\infty}
 \left(
  \sum_{k=0}^{\infty}\chi(2^{k-m}n)^2
 \right)
 \ip f{e_n}e_n ,
 \qquad f\in A_\omega^2 .
\end{equation}
If we write
$
 f(z)=\sum_{n=0}^\infty c_nz^n,
$
then
\begin{equation}\label{eq:smooth-low-pass-coefficients}
 \mathcal C_mf(z)
 =
 c_0+
 \sum_{n=1}^{\infty}
 \left(
  \sum_{k=0}^{\infty}\chi(2^{k-m}n)^2
 \right)c_nz^n.
\end{equation}
By \eqref{eq:chi-partition} and $\supp\chi\subset(1/2,2)$, for any $n\geq 1$ we have $0\leq\sum_{k=0}^\infty\chi(2^{k-m}n)^2\leq1,$
and
\begin{equation}\label{eq:low-pass-basic-properties}
 \mathcal C_me_n=e_n
 \quad(0\leq n\leq2^m),
 \qquad
 \mathcal C_me_n=0
 \quad(n\geq2^{m+1}).
\end{equation}
Therefore, $\|\mathcal C_m\|\le1$ and $\|I-\mathcal C_m\|\le1$.

For $\ell,m\geq N$, set
\[
 g_\ell
 =
 W^*\Pi_\ell x
 =
 \mathcal S_{\ell,\ell}x,
\quad
\text{and}
\quad
 f
 =
 W^*\Pi_{\geq N}x
 =
 \sum_{\ell\geq N}g_\ell,
 \]
and define
\begin{equation*}\label{eq:differenceSC}
 h_m =\mathcal S_{N,m}x-\mathcal C_mf.
\end{equation*}
By \eqref{eq:low-pass-basic-properties}, for $\ell\leq m-1$, $\mathcal C_mg_\ell=g_\ell$ and $\mathcal C_mg_\ell=0$ for $\ell\geq m+2$. Hence
\begin{equation*}\label{eq:synthesis-low-pass-boundary-new}
 h_m=(I-\mathcal C_m)g_m-\mathcal C_mg_{m+1}.
\end{equation*}

Using the weighted maximal operator (Lemma \ref{thm:tent-maximal} and Lemma \ref{lem:low-frequency-localization}), we shall prove the following version of Carleson embedding estimates in the Appendix.
\begin{lemma}\label{lem:car-emb-baby}
Let $\omega\in\widehat{\mathcal D}$, and let $N\geq1$ and $M\geq4$ be integers. Let $(b_J)_{J\in\mathcal Q_j,\ j\geq N+M}$
be a family of complex numbers satisfying \eqref{eq:shifted-Carleson-condition}. Then there exists $C=C(\omega,\chi)>0$, independent of $N$ and $M$, such that
\begin{equation*}\label{eq:analytic-shifted-Carleson-embedding}
 \sum_{j\geq N+M}\sum_{J\in\mathcal Q_j}
 |b_J|^2|\mathcal C_{j-M}f(a_J)|^2
 \leq  C\norm{b}_{\mathrm C;N,M}^2 \norm{f}_{A_\omega^2}^2,
 \qquad f\in A_\omega^2.
\end{equation*}
\end{lemma}

Now, we are going to prove Proposition~\ref{lem:shifted-Carleson}.
\begin{proof}[Proof of Proposition~\ref{lem:shifted-Carleson}]
By a standard density argument, it suffices to consider finitely
supported \(x\in\ell^2(\mathcal Q)\). Using the notation \(g_\ell,f\), and \(h_m\) as the above, we have
\[
 h_m
 =
 \mathcal S_{N,m}x-\mathcal C_mf
 =
 (I-\mathcal C_m)g_m-\mathcal C_mg_{m+1}.
\]
By the definition of \(g_m\) and \eqref{eq:low-pass-basic-properties}, 
\begin{equation}\label{eq:synthesis-low-pass-fixed-new}
 \mathcal C_{m+1}h_m=h_m.
\end{equation}
Moreover, since both \(\mathcal C_m\) and \(I-\mathcal C_m\) are contractions on \(A_\omega^2\),
\begin{align}
 \sum_{m\geq N}\|h_m\|_{A_\omega^2}^2
 \leq
 2\sum_{m\geq N}
 \left(
  \|g_m\|_{A_\omega^2}^2
  +
  \|g_{m+1}\|_{A_\omega^2}^2
 \right) 
 \leq
 4\sum_{\ell\geq N}\|g_\ell\|_{A_\omega^2}^2 
 \leq
 4\sum_{\ell\geq N}
 \|\Pi_\ell x\|_{\ell^2(\mathcal Q)}^2
 \leq
 4\|x\|_{\ell^2(\mathcal Q)}^2.
 \label{eq:synthesis-low-pass-square-sum-new}
\end{align}

We first estimate the error term \(h_m\). Fix \(I\in\mathcal Q_{m+1},m\geq N\), and let \(K\in\mathcal Q_m\) be the
dyadic parent of \(I\). By \eqref{eq:tent-volume},
$
 \omega(T(K))\leq C\omega(T(I)).
$
Since every \(J\in\mathcal Q_{m+M}\) with \(J\subseteq I\) is also contained in \(K\), the Carleson condition
\eqref{eq:shifted-Carleson-condition} implies
\begin{align}
 \sum_{\substack{J\in\mathcal Q_{m+M}: J\subseteq I}}
 |b_J|^2
 \leq
 \sum_{j\geq m+M}
 \sum_{\substack{J\in\mathcal Q_j: J\subseteq K}}
 |b_J|^2 \leq
 \|b\|_{\mathrm C;N,M}^2\omega(T(K))
 \leq
 C\|b\|_{\mathrm C;N,M}^2\omega(T(I)).
 \label{eq:boundary-one-level-Carleson-new}
\end{align}
If \(J\in\mathcal Q_{m+M}\) and \(J\subseteq I\), then \(e^{i\theta_J}\in I\) and $|a_J|=r_{m+M}\geq r_{m+1}.$
Hence, by \eqref{eq:synthesis-low-pass-fixed-new} and Lemma~\ref{lem:low-frequency-localization},
\[
 |h_m(a_J)|
 =
 |\mathcal C_{m+1}h_m(a_J)|
 \leq
 C\inf_{z\in T(I)}\mathcal M_\omega h_m(z).
\]
Every \(J\in\mathcal Q_{m+M}\) is contained in a unique interval
\(I\in\mathcal Q_{m+1}\). Therefore, by
\eqref{eq:boundary-one-level-Carleson-new},
\begin{align*}
 \sum_{J\in\mathcal Q_{m+M}}
 |b_J|^2|h_m(a_J)|^2\leq
 C\|b\|_{\mathrm C;N,M}^2
 \sum_{I\in\mathcal Q_{m+1}}
 \omega(T(I))
 \left(
  \inf_{z\in T(I)}\mathcal M_\omega h_m(z)
 \right)^2.
\end{align*}
Then,
\[
 \sum_{I\in\mathcal Q_{m+1}}
 \omega(T(I))
 \left(
  \inf_{z\in T(I)}\mathcal M_\omega h_m(z)
 \right)^2
 \leq
 \|\mathcal M_\omega h_m\|_{L_\omega^2}^2.
\]
By Lemma~\ref{thm:tent-maximal}, we get
\[
 \sum_{J\in\mathcal Q_{m+M}}
 |b_J|^2|h_m(a_J)|^2
 \leq
 C\|b\|_{\mathrm C;N,M}^2
 \|h_m\|_{A_\omega^2}^2.
\]
Summing over \(m\geq N\) and using
\eqref{eq:synthesis-low-pass-square-sum-new}, we obtain
\begin{align}
 &\sum_{j\geq N+M}
 \sum_{J\in\mathcal Q_j}
 |b_J|^2
 \left|
  \mathcal S_{N,j-M}x(a_J)
  -
  \mathcal C_{j-M}f(a_J)
 \right|^2 \leq
 C\|b\|_{\mathrm C;N,M}^2
 \|x\|_{\ell^2(\mathcal Q)}^2.
 \label{eq:synthesis-boundary-Carleson-new}
\end{align}
On the other hand, by Lemma~\ref{lem:car-emb-baby}
\begin{align}
 \sum_{j\geq N+M}
 \sum_{J\in\mathcal Q_j}
 |b_J|^2
 |\mathcal C_{j-M}f(a_J)|^2
 \leq
 C\|b\|_{\mathrm C;N,M}^2
 \|f\|_{A_\omega^2}^2 
 \leq
 C\|b\|_{\mathrm C;N,M}^2
 \|x\|_{\ell^2(\mathcal Q)}^2,
 \label{eq:synthesis-low-pass-Carleson-new}
\end{align}
where we used
\[
 \|f\|_{A_\omega^2}
 =
 \|W^*\Pi_{\geq N}x\|_{A_\omega^2}
 \leq
 \|\Pi_{\geq N}x\|_{\ell^2(\mathcal Q)}
 \leq
 \|x\|_{\ell^2(\mathcal Q)}.
\]
Finally,
\[
 |\mathcal S_{N,j-M}x(a_J)|^2
 \leq
 2|\mathcal C_{j-M}f(a_J)|^2
 +
 2\left|
  \mathcal S_{N,j-M}x(a_J)
  -
  \mathcal C_{j-M}f(a_J)
 \right|^2.
\]
By \eqref{eq:synthesis-boundary-Carleson-new} and \eqref{eq:synthesis-low-pass-Carleson-new}, the inequality \eqref{eq:shifted-Carleson-embedding} holds. This completes the whole proof.
\end{proof}

\subsection{The proof of Proposition~\ref{lem:upper-triangular-Carleson-scalar}}
By \eqref{eq:frame-element-first-order-overlap} and \eqref{eq:tent-volume-ratio}, we have the following lemma.
\begin{lemma}\label{lem:triangular-freezing-error-scalar}
Let $\omega\in\widehat{\mathcal D}$ and $u\in L^\infty(\D)$. There exists a constant $C=C(\omega,\chi)>0$ such that for any $I\in\mathcal Q_k, J\in\mathcal Q_j$ with $k\ge j\ge1$, we have 
\[
 \bigg|\int_\D u(z) \bigl(\phi_J(z)-\phi_J(a_I)\bigr)\overline{\phi_I(z)}\omega(z)\dd A(z)\bigg|\leq  C\norm{u}_{L^\infty(\D)}
 2^{-3(k-j)/2} \bigl(1+d_j(I,J)\bigr)^{-3}
\]
and
\[
\bigg| \int_\D u(z)\phi_I(z)\bigl(\overline{\phi_J(z)}-\overline{\phi_J(a_I)}\bigr)\omega(z)\dd A(z)\bigg|\leq C\norm{u}_{L^\infty(\D)} 2^{-3(k-j)/2}\bigl(1+d_j(I,J)\bigr)^{-3}.
\]
\end{lemma}

For integers $N,M\ge1$, we define a matrix
$R_{N,4}^{M}(u)$ by
\begin{equation*}\label{eq:upper-triangular-remainder-matrix-scalar}
 R_{N,4}^{M}(u)(I,J)
 =
 \begin{cases}
 \displaystyle
 \int_\D u(z)\phi_J(z)
 \bigl(
  \overline{\phi_I(z)}
  -\overline{\phi_I(a_J)}
 \bigr)
 \omega(z)\dd A(z),
 & \begin{array}{l}
    I\in\mathcal Q_k,\ J\in\mathcal Q_j,\\
    j,k\ge N,\ j\ge k+M,
   \end{array}\\[5mm]
 0, & \text{otherwise}.
 \end{cases}
\end{equation*}

\begin{corollary}\label{cor:triangular-remainder-operator-scalar}
Let $\omega\in\widehat{\mathcal D}$ and $u\in L^\infty(\D)$. There exists a constant $C=C(\omega,\chi)>0$, independent of $N$ and
$M$, such that
\begin{equation*}\label{eq:triangular-remainder-operator-scalar}
  \norm{R_{N,4}^{M}(u)} \le C2^{-M}\norm{u}_{L^\infty(\D)}.
\end{equation*}
\end{corollary}

\begin{proof}
For every $d\ge M$, let $R_{N,4}^{(d)}(u)$ be a matrix given by
\[
 R_{N,4}^{(d)}(u)(I,J)
 =
 \begin{cases}
 \displaystyle
 \int_\D u(z)\phi_J(z)
 \bigl(
  \overline{\phi_I(z)}
  -\overline{\phi_I(a_J)}
 \bigr)
 \omega(z)\dd A(z),
 & \begin{array}{l}
    I\in\mathcal Q_k,\ J\in\mathcal Q_{k+d},\\
    k\ge N,
   \end{array}\\[5mm]
 0, & \text{otherwise}.
 \end{cases}
\]
Then
\[
 R_{N,4}^{M}(u) = \sum_{d\ge M}R_{N,4}^{(d)}(u).
\]
Fix $I\in\mathcal Q_k$. For $J\in\mathcal Q_{k+d}$, by Lemma~\ref{lem:triangular-freezing-error-scalar}, 
\[
 \bigl|R_{N,4}^{(d)}(u)(I,J)\bigr| \le  C\norm{u}_{L^\infty(\D)} 2^{-3d/2}\bigl(1+d_k(J,I)\bigr)^{-3}.
\]
By Lemma~\ref{lem:angular-counting},
\begin{align*}
 \sum_{J\in\mathcal Q_{k+d}}
 \bigl|R_{N,4}^{(d)}(u)(I,J)\bigr|
 \le
 C\norm{u}_{L^\infty(\D)}2^{-3d/2}
 \sum_{J\in\mathcal Q_{k+d}}
 \bigl(1+d_k(J,I)\bigr)^{-3}
 \le
 C\norm{u}_{L^\infty(\D)}2^{-d/2}.
\end{align*}
Thus,
\begin{equation}\label{eq:upper-remainder-row-sum-scalar}
 \sup_{I\in\mathcal Q}
 \sum_{J\in\mathcal Q}
 \bigl|R_{N,4}^{(d)}(u)(I,J)\bigr|
 \le
 C\norm{u}_{L^\infty(\D)}2^{-d/2}.
\end{equation}
Next, fix $J\in\mathcal Q_j$. By a similar argument, we get
\begin{align*}
 \sum_{I\in\mathcal Q_{j-d}} \bigl|R_{N,4}^{(d)}(u)(I,J)\bigr|
 \le
 C\norm{u}_{L^\infty(\D)}2^{-3d/2}
 \sum_{I\in\mathcal Q_{j-d}}
 \bigl(1+d_{j-d}(J,I)\bigr)^{-3}\le
 C\norm{u}_{L^\infty(\D)}2^{-3d/2}.
\end{align*}
It follows that
\begin{equation}\label{eq:upper-remainder-column-sum-scalar}
 \sup_{J\in\mathcal Q}
 \sum_{I\in\mathcal Q}
 \bigl|R_{N,4}^{(d)}(u)(I,J)\bigr|
 \le
 C\norm{u}_{L^\infty(\D)}2^{-3d/2}.
\end{equation}
Using Schur's test (Lemma \ref{lem:Schur}) with \eqref{eq:upper-remainder-row-sum-scalar} and
\eqref{eq:upper-remainder-column-sum-scalar}, we obtain
\begin{align*}
 \norm{R_{N,4}^{(d)}(u)}\le
 \left(
  C\norm{u}_{L^\infty(\D)}2^{-d/2}
 \right)^{1/2}
 \left(
  C\norm{u}_{L^\infty(\D)}2^{-3d/2}
 \right)^{1/2}\le
 C\norm{u}_{L^\infty(\D)}2^{-d}.
\end{align*}
Therefore,
\begin{align*}
 \norm{R_{N,4}^{M}(u)}_{\mathcal L(\ell^2(\mathcal Q))}
 \le
 \sum_{d\ge M}
 \norm{R_{N,4}^{(d)}(u)}_{\mathcal L(\ell^2(\mathcal Q))}
 \le
 C\norm{u}_{L^\infty(\D)}
 \sum_{d\ge M}2^{-d}
 \le
 C2^{-M}\norm{u}_{L^\infty(\D)}.
\end{align*}
This completes the proof.
\end{proof}

\begin{proof}[Proof of Proposition~\ref{lem:upper-triangular-Carleson-scalar}]
Suppose that $j\ge k+M$, with $J\in\mathcal Q_j$ and $I\in\mathcal Q_k$.  Recalling that
\[
 c_J=\int_\D u(z)\phi_J(z)\omega(z)\dd A(z),
\]
we observe 
\begin{align}
 A_u(I,J)
 =
 \int_\D u(z)\phi_J(z)\overline{\phi_I(z)}
 \omega(z)\dd A(z)\notag=
 \overline{\phi_I(a_J)}c_J
 +R_{N,4}^{M}(u)(I,J).
 \label{eq:upper-freezing-scalar}
\end{align}
By Corollary~\ref{cor:triangular-remainder-operator-scalar},
\begin{equation}\label{eq:upper-remainder-scalar}
 \norm{R_{N,4}^{M}(u)}_{\mathcal L(\ell^2(\mathcal Q))}
 \le
 C2^{-M}\norm{u}_{L^\infty(\D)}.
\end{equation}
Moreover, for $j\ge k+M, j,k\ge N$,
\[
 \bigl(A_{N,4}^{M}(u)-R_{N,4}^{M}(u)\bigr)(I,J)
 =
 \overline{\phi_I(a_J)}c_J.
\]
Let $x,y\in\ell^2(\mathcal Q)$ be finitely supported. Then
\begin{align*}
 \left\langle
 \bigl(A_{N,4}^{M}(u)-R_{N,4}^{M}(u)\bigr)x,y
 \right\rangle_{\ell^2(\mathcal Q)}
 &=
 \sum_{j\ge N+M}\sum_{J\in\mathcal Q_j}
 \sum_{k=N}^{j-M}\sum_{I\in\mathcal Q_k}
 \overline{\phi_I(a_J)}c_Jx_J\overline{y_I}\\
 &=
 \sum_{j\ge N+M}\sum_{J\in\mathcal Q_j}
 x_Jc_J\,
 \overline{
  \sum_{k=N}^{j-M}\sum_{I\in\mathcal Q_k}
  y_I\phi_I(a_J)
 }\\
 &=
 \sum_{j\ge N+M}\sum_{J\in\mathcal Q_j}
 x_Jc_J\,
 \overline{\mathcal S_{N,j-M}y(a_J)}.
\end{align*}
Thus,
\begin{equation*}\label{eq:upper-frozen-bilinear}
 \left\langle
 \bigl(A_{N,4}^{M}(u)-R_{N,4}^{M}(u)\bigr)x,y
 \right\rangle_{\ell^2(\mathcal Q)}
 =
 \sum_{j\ge N+M}\sum_{J\in\mathcal Q_j}
 x_Jc_J\,
 \overline{\mathcal S_{N,j-M}y(a_J)}.
\end{equation*}
By the Cauchy--Schwarz inequality,
\begin{align*}
 \left|
 \left\langle
 \bigl(A_{N,4}^{M}(u)-R_{N,4}^{M}(u)\bigr)x,y
 \right\rangle_{\ell^2(\mathcal Q)}
 \right| \le
 \norm{x}_{\ell^2(\mathcal Q)}
 \left(
  \sum_{j\ge N+M}\sum_{J\in\mathcal Q_j}
  |c_J|^2
  |\mathcal S_{N,j-M}y(a_J)|^2
 \right)^{1/2}.
\end{align*}
Using Proposition~\ref{lem:shifted-Carleson} with $b_J=c_J$, we get
\[
 \sum_{j\ge N+M}\sum_{J\in\mathcal Q_j}
 |c_J|^2
 |\mathcal S_{N,j-M}y(a_J)|^2
 \le
 C\mathfrak C_N^M(u)
 \norm{y}_{\ell^2(\mathcal Q)}^2.
\]
It follows that
\[
 \left|
 \left\langle
 \bigl(A_{N,4}^{M}(u)-R_{N,4}^{M}(u)\bigr)x,y
 \right\rangle_{\ell^2(\mathcal Q)}
 \right|
 \le
 C\mathfrak C_N^M(u)^{1/2}
 \norm{x}_{\ell^2(\mathcal Q)}
 \norm{y}_{\ell^2(\mathcal Q)}.
\]
Taking the supremum over all finitely supported unit vectors $x$ and $y$, we obtain
\[
 \norm{
  A_{N,4}^{M}(u)-R_{N,4}^{M}(u)
 }_{\mathcal L(\ell^2(\mathcal Q))}
 \le
 C\mathfrak C_N^M(u)^{1/2}.
\]
Combining this estimate with \eqref{eq:upper-remainder-scalar} yields
\[
 \norm{A_{N,4}^{M}(u)}_{\mathcal L(\ell^2(\mathcal Q))}
 \le
 C\mathfrak C_N^M(u)^{1/2}
 +C2^{-M}\norm{u}_{L^\infty(\D)}.
\]
This completes the whole proof.
\end{proof}

\subsection{Moment estimates for $\widehat{\mathcal{D}}$-weights}
To simplify notation, write the moment function as
\[
 \mathfrak m(x)
 =
 \omega_{2x+1}
 =
 \int_0^1r^{2x+1}\omega(r)\dd r,
 \qquad x>0.
\]
There exists a constant $C=C(\omega)\geq1$ such that
\begin{equation}\label{eq:new-constant-weight}
 0\leq-\frac{\mathfrak m'(x)}{\mathfrak m(x)}
 \leq\frac{C}{x},
 \qquad x\geq1.
\end{equation}
Indeed, using 
\[
 -\mathfrak m'(x)
 =
 2\int_0^1r^{2x+1}(-\log r)\omega(r)\dd r.
\]
the elementary estimate
\[
 \sup_{0<r<1}r^x(-\log r)\leq\frac1{ex},
 \qquad x>0,
\]
\eqref{eq:moment-doubling} and the monotonicity of the moments, we obtain
\begin{align*}
 -\mathfrak m'(x)
 &\leq
 \frac{2}{ex}\omega_{x+1}
 \leq
 \frac{C}{x}\omega_{2x+2}
 \leq
 \frac{C}{x}\mathfrak m(x).
\end{align*}
This is the desired inequality \eqref{eq:new-constant-weight}. Then, for $x\geq1$ and $N\geq0$,
\begin{align*}
 \log\frac{\mathfrak m(x)}{\mathfrak m(x+N)}
 =
 \int_x^{x+N}
 -\frac{\mathfrak m'(t)}{\mathfrak m(t)}\dd t\leq
 C\log\left(1+\frac{N}{x}\right),
\end{align*}
and
\begin{equation}\label{eq:moment-ratio-lower-bound}
 \frac{\mathfrak m(x+N)}{\mathfrak m(x)}
 \geq
 \left(1+\frac{N}{x}\right)^{-C},
 \qquad x\geq1,\quad N\geq0.
\end{equation}

Fix $0<\epsilon\leq1/4$, and let $C$ be the constant in \eqref{eq:moment-ratio-lower-bound}. Set
\begin{equation}\label{eq:import-constant-epsilon}
 d_\epsilon
 =
 \max\left\{
 4,\,
 \left\lceil
 1+\log_2\left(\frac{C}{\epsilon}\right)
 \right\rceil
 \right\}.
\end{equation}
Then
\begin{equation}\label{eq:d-epsilon-choice}
 1-\left(1+2^{1-d_\epsilon}\right)^{-C}
 \leq
 C2^{1-d_\epsilon}
 \leq\epsilon.
\end{equation}
For $x>0$ and $k\geq1$, define
\begin{equation*}\label{eq:E-x-k}
 \mathcal E(x,k)
 =
 \frac1{\mathfrak m(x)}
 \int_0^1(1-r^{2k})^2r^{2x+1}\omega(r)\dd r.
\end{equation*}
For every fixed $k\geq1$, the function $\mathcal E(x,k)$ is nonincreasing on $(0,+\infty)$. In fact, a direct computation shows 
\begin{align*}
 \frac{\partial}{\partial x}\mathcal E(x,k)
 =
 \frac1{\mathfrak m(x)^2}
 \int_0^1\int_0^1
 \left((1-r^{2k})^2-(1-s^{2k})^2\right)
 \left(\log r-\log s\right)
 r^{2x+1}s^{2x+1}\omega(r)\omega(s)\dd r\dd s.
\end{align*}
Since $(1-r^{2k})^2$ is decreasing on $(0,1)$ and $\log r$ is increasing,
\[
 \left((1-r^{2k})^2-(1-s^{2k})^2\right)
 \left(\log r-\log s\right)
 \leq0.
\]
Therefore,
$
 \frac{\partial}{\partial x}\mathcal E(x,k)\leq0.
$
For any integer $m\geq d_\epsilon+2$, using \eqref{eq:moment-ratio-lower-bound} with $x=2^{m-1}$, $N=2^{m-d_\epsilon}$ and \eqref{eq:d-epsilon-choice}, we have
\begin{align*}
 \mathcal E(2^{m-1},2^{m-d_\epsilon})
 \leq
 1-
 \frac{
 \mathfrak m(2^{m-1}+2^{m-d_\epsilon})
 }{
 \mathfrak m(2^{m-1})
 }
 \leq
 1-\left(1+2^{1-d_\epsilon}\right)^{-C}
 \leq\epsilon.
\end{align*}
Take an integer $m_0\geq d_\epsilon+2$.  Starting from $m_0$, define recursively, for every $\nu\geq0$,
\begin{equation}\label{eq:recursive-stopping-indices}
 \begin{split}
 x_\nu&=2^{m_\nu-1},\\
 h_\nu
 &=
 \max\left\{
 h\in\Z:
 h\geq m_\nu-d_\epsilon,\ 
 \mathcal E(x_\nu,2^h)\leq\epsilon
 \right\},\\
 N_\nu&=2^{h_\nu},
 \qquad
 m_{\nu+1}=h_\nu+d_\epsilon+1.
 \end{split}
\end{equation}
For every fixed $x>0$, the Dominated Convergence Theorem gives
\[
 \lim_{h\to\infty}\mathcal E(x,2^h)
 =
 \frac1{\mathfrak m(x)}
 \int_0^1r^{2x+1}\omega(r)\dd r
 =
 1.
\]

Since $\epsilon<1$, by its definition, $h_\nu\geq m_\nu-d_\epsilon$ and $m_{\nu+1}= h_\nu+d_\epsilon+1 \geq m_\nu+1$. It follows inductively that $(m_\nu)_{\nu\geq0}$ is strictly increasing and tends to infinity. The family of intervals $\{[m_\nu,m_{\nu+1}-1]\cap\Z\}_{\nu\geq0}$ forms a partition of $[m_0,\infty)\cap\Z$.   

Suppose now that $m_\nu\leq j<m_{\nu+1}$ and $\chi(2^{-j}n)\neq0$. Then,
\[
 n>2^{j-1}\geq2^{m_\nu-1}=x_\nu.
\]
Since $\mathcal E(x,k)$ is nonincreasing in $x$, by the definition of $h_\nu$,
\begin{align}\label{eq:adaptive-radial-smallness}
 \frac1{\omega_{2n+1}}
 \int_0^1(1-r^{2N_\nu})^2r^{2n+1}\omega(r)\dd r
 =
 \mathcal E(n,N_\nu)\leq
 \mathcal E(x_\nu,N_\nu)
 \leq
 \epsilon.
\end{align}
Since $h_\nu+1>h_\nu$ and $h_\nu+1\geq m_\nu-d_\epsilon$, the definition of $h_\nu$ implies
$
 \mathcal E(x_\nu,2^{h_\nu+1})>\epsilon
$
and, moreover,
$
 1-
 \frac{
 \mathfrak m(x_\nu+2^{h_\nu+1})
 }{
 \mathfrak m(x_\nu)
 }>\epsilon.
$
Therefore
\begin{equation}\label{eq:intermediate-moment-drop}
 \mathfrak m(x_\nu+2^{h_\nu+1})
 <
 (1-\epsilon)\mathfrak m(x_\nu).
\end{equation}
Note that
$
 x_\nu+2^{h_\nu+1}\leq x_{\nu+1}
$
and $\mathfrak m$ is nonincreasing, by \eqref{eq:intermediate-moment-drop}, we obtain
\begin{align*}\label{eq:stopping-moment-decay}
 \mathfrak m(x_{\nu+1})
 \leq
 \mathfrak m(x_\nu+2^{h_\nu+1})
 <
 (1-\epsilon)\mathfrak m(x_\nu),
\end{align*}
In particular,
\begin{equation}\label{eq:iterated-stopping-moment-decay}
 \mathfrak m(x_{\nu+1})
 <
 (1-\epsilon)^{\nu+1}\mathfrak m(x_0),
 \qquad \nu\geq0.
\end{equation}
Recall that
$
 x_{\nu+1}
 =
 2^{h_\nu+d_\epsilon}
 =
 2^{d_\epsilon}N_\nu.
$
Using \eqref{eq:moment-doubling} and the monotonicity of the moments, we obtain
\begin{align*}
 \mathfrak m(N_\nu)
 =
 \omega_{2N_\nu+1}\leq
 C_\epsilon
 \omega_{2^{d_\epsilon}(2N_\nu+1)}\leq
 C_\epsilon
 \omega_{2x_{\nu+1}+1}
 =
 C_\epsilon\mathfrak m(x_{\nu+1}).
\end{align*}
Combining this with \eqref{eq:iterated-stopping-moment-decay}, we conclude that
\begin{align*}
 \sum_{\nu=0}^\infty\mathfrak m(N_\nu)
 \leq
 C_\epsilon
 \sum_{\nu=0}^\infty\mathfrak m(x_{\nu+1})\leq
 C_\epsilon\mathfrak m(x_0)
 \sum_{\nu=0}^\infty(1-\epsilon)^{\nu+1}
 \leq
 C_\epsilon'\mathfrak m(x_0).
\end{align*}
Hence, there is a constant $C_\epsilon''=C_\epsilon''(\omega,\epsilon)>0$ such that
\begin{equation}\label{eq:adaptive-modulation-sum}
 \sum_{\nu=0}^\infty\mathfrak m(N_\nu)
 \leq
 C_\epsilon''\mathfrak m(x_0).
\end{equation}

\subsection{The proof of Proposition \ref{cor:scalar-Carleson-tail}}
For integers $M\ge 4, h\ge M+1$, define
\begin{equation}\label{eq:fixed-band-matrix-projection}
 \Pi_{h,M}
 =
 \sum_{\substack{j\ge1: |j-h|\le M+1}}\Pi_j.
\end{equation}
For $N\ge M+1$, let 
\begin{equation*}\label{eq:band-compression-seminorm}
 \norm{A_u}_{N;M}=\sup_{h\ge N}\norm{\Pi_{h,M}A_u\Pi_{h,M}}.
\end{equation*}
For fixed integers $M,L\geq1$, recall
\begin{equation*}
 \tau(N;M,L)=
 \sup\left\{
 |A_u(I,J)|:
 \begin{array}{c}
 I\in\mathcal Q_k,\ J\in\mathcal Q_j,\ j,k\geq N,\\
 |j-k|\leq M,\ d_{j\wedge k}(I,J)\leq L
 \end{array}
 \right\}.
\end{equation*}

\begin{proposition}\label{prop:scalar-fixed-band-compression}
Let $\omega\in\widehat{\mathcal D}$ and $u\in L^\infty(\D)$. If $\widetilde u_\omega(a)\to0$ as $|a|\to1^-$, then, for every integer
$M\ge4$,
\begin{equation}\label{eq:scalar-fixed-band-vanishing}
 \lim_{N\to\infty}\norm{A_u}_{N;M}=0.
\end{equation}
\end{proposition}

\begin{proof}
Fix $M\ge4$, $L\ge1$, and $N\ge2M+4$. Let $h\ge N$. If $I\in\mathcal Q_k, J\in\mathcal Q_j$ index an entry of
$\Pi_{h,M}A_u\Pi_{h,M}$, then $j,k\ge h-M-1$ and $|j-k|\le2M+2.$ According as $d_{j\wedge k}(I,J)\le L$ or $d_{j\wedge k}(I,J)>L$, we decompose
\[
 \Pi_{h,M}A_u\Pi_{h,M}=A_{\mathrm D}+A_{\mathrm L}.
\]
Every nonzero entry of $A_{\mathrm D}$ is bounded by $\tau(h-M-1;2M+2,L)$. Moreover, by Lemma~\ref{lem:local-band-degree-Dhat}, with its bandwidth parameter equal to $2M+2$, each row and each column of $A_{\mathrm D}$ has at most
$
 2^{2M+8}(2M+2)L
$
nonzero entries. By Schur's test (Lemma \ref{lem:Schur}), 
\[
 \norm{A_{\mathrm D}}
 \le
 2^{2M+8}(2M+2)L\,
 \tau(h-M-1;2M+2,L).
\]
Since $A_{\mathrm L}$ is a compression of $A_{h-M-1,2}^{2M+3,L}(u)$, by Lemma~\ref{lem:near-diagonal-scalar-tail-2},
\[
 \norm{A_{\mathrm L}}
 \le
 C2^{(2M+3)/2}(1+L)^{-2}
 \norm{u}_{L^\infty(\D)},
\]
where $C=C(\omega,\chi)>0$ is independent of $N$, $h$, $M$, and $L$.

Since $\tau(\,\cdot\,;2M+2,L)$ is nonincreasing, taking the supremum over $h\ge N$, we have
\begin{align*}\label{eq:scalar-fixed-band-quantitative}
 \norm{A_u}_{N;M}
 \le
 2^{2M+10} \tau(N-M-1;2M+2,L)LM
 +
 C2^{M+2}(1+L)^{-2}
 \norm{u}_{L^\infty(\D)}.
\end{align*}
For fixed $M$ and $L$, by Lemma~\ref{lem:local-band-entries-Dhat},
\[
 \lim_{N\to\infty}\tau(N-M-1;2M+2,L)=0.
\]
Hence,
\[
 \limsup_{N\to\infty}\norm{A_u}_{N;M}
 \le
 C2^{M+2}(1+L)^{-2}\norm{u}_{L^\infty(\D)}.
\]
Letting $L\to\infty$, we get the desired result \eqref{eq:scalar-fixed-band-vanishing}. This completes the whole proof.
\end{proof}

\begin{lemma}\label{lem:radial-Bessel-modulation}
Let \(\omega\in\widehat{\mathcal D}\), let \(\Lambda\) be a set of positive integers, and let \(\{\alpha_j\}_{j\in\Lambda}\) be a family of measurable functions defined on $[0,1)$ and taking values in $[0,1]$. Assume that there exists a constant \(\alpha\ge0\) such that, for every \(j\in\Lambda\) and every \(n\ge1\) satisfying \(\chi(2^{-j}n)\ne0\),
\begin{equation}\label{eq:radial-Bessel-moment-condition}
 \frac1{\omega_{2n+1}}
 \int_0^1\alpha_j(r)^2r^{2n+1}\omega(r)\dd r
 \le\alpha.
\end{equation}
Then, for every $g\in L^2_\omega$,
\begin{equation*}\label{eq:radial-Bessel-bound}
 \sum_{j\in \Lambda}\sum_{J\in\mathcal Q_j}
 \left|
  \int_\D g(z)\alpha_j(|z|)\overline{\phi_J(z)}
  \omega(z)\dd A(z)
 \right|^2
 \le\alpha\norm{g}_{L^2_\omega}^2.
\end{equation*}
\end{lemma}

\begin{proof}
For $g\in L^2_\omega$, using the same arguments as in the proof of Proposition~\ref{prop:Parseval}, we have
\begin{align*}
 \sum_{J\in\mathcal Q_j}
 \left|
  \int_\D g(z)\alpha_j(|z|)\overline{\phi_J(z)}
  \omega(z)\dd A(z)
 \right|^2=
 \sum_{n\ge1}\chi(2^{-j}n)^2
 \left|
  \int_\D g(z)\alpha_j(|z|)\overline{e_n(z)}
  \omega(z)\dd A(z)
 \right|^2.
\end{align*}
For $n\ge0$, we have
\begin{align*}
 \left|
  \int_\D g(z)\alpha_j(|z|)\overline{e_n(z)}
  \omega(z)\dd A(z)
 \right|^2
 \le\frac{2}{\omega_{2n+1}}
 \left(\int_0^1|g_n(r)|^2r\omega(r)\dd r\right)\times
 \left(\int_0^1\alpha_j(r)^2r^{2n+1}\omega(r)\dd r\right).
\end{align*}
Using \eqref{eq:radial-Bessel-moment-condition}, summing over
$j\in \Lambda$, and then using
$
 \sum_{j\in\Z}\chi(2^{-j}n)^2=1,
$
we obtain
\begin{align*}
 \sum_{j\in \Lambda}\sum_{J\in\mathcal Q_j}
 \left|
  \int_\D g(z)\alpha_j(|z|)\overline{\phi_J(z)}
  \omega(z)\dd A(z)
 \right|^2 \le
 2\alpha\sum_{n\ge1}
 \int_0^1|g_n(r)|^2r\omega(r)\dd r
 \le\alpha\norm{g}_{L^2_\omega}^2.
\end{align*}
This completes the proof.
\end{proof}

\begin{lemma}\label{lem:modulated-frame-norm-only-k}
Let $\omega\in\widehat{\mathcal D}$. There exists a constant $C=C(\chi)>0$ such that, for every $J\in\mathcal Q_j,j\geq 1$, and integer $k\ge1$,
\begin{equation*}\label{eq:modulated-frame-norm-general}
 \norm{z^k\phi_J}_{A^2_\omega}^2\le C\frac{\mathfrak m(k)}{\mathfrak m(2^{j+1})}.
\end{equation*}
\end{lemma}

\begin{proof}
By the frame definition and the orthonormality of the monomials, 
\[
 \norm{z^k\phi_J}_{A^2_\omega}^2
 =
 \frac1{2^{j+3}}
 \sum_{n\ge1}\chi(2^{-j}n)^2
 \frac{\mathfrak m(n+k)}{\mathfrak m(n)}.
\]
If $\chi(2^{-j}n)\ne0$, then $n<2^{j+1}$. Since $\mathfrak m$ is
nonincreasing,
\[
 \mathfrak m(n+k)\le \mathfrak m(k),
 \qquad
 \mathfrak m(n)\ge \mathfrak m(2^{j+1}).
\]
Moreover, there are fewer than $2^{j+1}$ nonzero summands. Therefore,
\begin{align*}
 \norm{z^k\phi_J}_{A^2_\omega}^2
 \le
 \frac{2^{j+1}}{2^{j+3}}
 \norm{\chi}_{L^\infty}^2
 \frac{\mathfrak m(k)}{\mathfrak m(2^{j+1})}\le
 C\frac{\mathfrak m(k)}{\mathfrak m(2^{j+1})}.
\end{align*}
The proof is complete now.
\end{proof}

Let $0<\epsilon\leq1/4$, let $d_\epsilon$ be the integer defined in
\eqref{eq:import-constant-epsilon}, and set
\begin{align}\label{eq:setM-N}
 M_\epsilon=d_\epsilon+4,
 \qquad
 N_\epsilon=d_\epsilon+3.
\end{align}

\begin{proposition}\label{prop:upper-Carleson-one-sided}
Let $\epsilon\in(0,1/4]$. Let $\omega\in\widehat{\mathcal D}$ and $M_\epsilon,N_\epsilon$ be given by \eqref{eq:setM-N}. There exist constants
$C_0=C_0(\omega,\chi)>0$ and
$C_\epsilon=C_\epsilon(\omega,\chi)>0$ such that, for every
$u\in L^\infty(\D)$, $M\geq M_\epsilon$, and $N\geq N_\epsilon$,
\begin{equation}\label{eq:upper-Carleson-main-estimate}
 \mathfrak C_N^M(u)
 \leq
 C_0\epsilon\norm{u}_{L^\infty(\D)}^2
 +C_\epsilon\bigl(\norm{A_u}_{N;d_\epsilon+2}\bigr)^2
 +C_\epsilon2^{-2M}\norm{u}_{L^\infty(\D)}^2.
\end{equation}
\end{proposition}

\begin{proof}
Fix $M\geq M_\epsilon$ and $N\geq N_\epsilon$. It is enough to consider $K\in\mathcal Q_k$ with $k\geq N$. Set
\[
 m_0=k+M,
\]
and let $m_\nu,h_\nu,N_\nu$, and $x_\nu$ be given by
\eqref{eq:recursive-stopping-indices}. Since
\[
 h_\nu\geq m_\nu-d_\epsilon
 \geq k+M-d_\epsilon>N,
\]
we have
\begin{equation}\label{eq:stopping-block-band-bound}
 \norm{\Pi_{h_\nu,d_\epsilon+2}A_u
 \Pi_{h_\nu,d_\epsilon+2}}
 \leq
 \norm{A_u}_{N;d_\epsilon+2}.
\end{equation}
For $m_\nu\leq j<m_{\nu+1}$ and $J\in\mathcal Q_j$, decompose
$
 c_J=e_J+d_J,
$
where
\[
 e_J
 =
 \int_\D u(z)(1-|z|^{2N_\nu})\phi_J(z)
 \omega(z)\dd A(z), \quad \text{and}
 \]
 \[
 d_J
 =
 \int_\D u(z)|z|^{2N_\nu}\phi_J(z)
 \omega(z)\dd A(z).
\]
We first estimate the terms $e_J$. Let
\[
 E_K
 =
 \{re^{i\theta}:r_k\leq r<1,\ e^{i\theta}\in5K\}.
\]
Since $k\geq N_\epsilon$, the arc $5K$ has length less than $2\pi$,
and hence,
$
 \omega(E_K)\leq5\omega(T(K)).
$
Write $e_J=e_J^{\mathrm{l}}+e_J^{\mathrm{o}}$ according as the integration is over $E_K$ or $\D\setminus E_K$, respectively.
By \eqref{eq:adaptive-radial-smallness} and Lemma~\ref{lem:radial-Bessel-modulation},
\[
 \sum_{j\geq k+M}\sum_{J\in\mathcal Q_j}|e_J^{\mathrm{l}}|^2 \leq 5\epsilon\norm{u}_{L^\infty(\D)}^2\omega(T(K)).
\]
On the other hand, if $J\subset K$, $j\geq k+M$, and $z=re^{i\theta}\notin E_K$, then
\[
 1+2^j(1-r)+2^jd_\T(\theta,\theta_J)
 \geq c2^{j-k}.
\]
By Lemma~\ref{lem:frame-element-moments} with exponent three, 
\[
 \int_{\D\setminus E_K}|\phi_J(z)|\omega(z)\dd A(z)
 \leq
 C2^{-3(j-k)}\sqrt{\omega(T(J))}.
\]
Moreover, we obtain
\begin{equation}\label{eq:eJ-Carleson-one-sided}
 \frac1{\omega(T(K))}
 \sum_{j\geq k+M}
 \sum_{\substack{J\in\mathcal Q_j:J\subset K}}
 |e_J|^2
 \leq
 C\epsilon\norm{u}_{L^\infty(\D)}^2
 +
 C2^{-6M}\norm{u}_{L^\infty(\D)}^2.
\end{equation}

We next estimate $d_J$. For a fixed $K$. set
\[
 \mathcal J_{\nu,K}
 =
 \bigcup_{j=m_\nu}^{m_{\nu+1}-1}
 \{J\in\mathcal Q_j:J\subset K\}.
\]
Let $(\alpha_J)_{J\in\mathcal J_{\nu,K}}$ be a unit vector in
$\ell^2(\mathcal J_{\nu,K})$, and define
\[
 F
 =
 \sum_{J\in\mathcal J_{\nu,K}}
 \frac{\alpha_J}{\overline{\phi_K(a_J)}}\phi_J.
\]
Since $j\geq k+M\geq k+4$, Lemma~\ref{lem:frame-element-core} and
the contractivity of $W^*$ imply
\begin{equation}\label{eq:F-adaptive-norm}
 \norm{F}_{A^2_\omega}
 \leq
 C\sqrt{\omega(T(K))}.
\end{equation}
For $J\in\mathcal J_{\nu,K}$, set
\[
 \rho_{K,J}^\nu
 =
 \int_\D u(z)|z|^{2N_\nu}\phi_J(z)
 \bigl(
 \overline{\phi_K(z)}-\overline{\phi_K(a_J)}
 \bigr)\omega(z)\dd A(z).
\]
Then
\[
 \left\langle
 A_uW(z^{N_\nu}\phi_J),
 W(z^{N_\nu}\phi_K)
 \right\rangle
 =
 \overline{\phi_K(a_J)}d_J+\rho_{K,J}^\nu.
\]
By Lemma~\ref{lem:frame-element-overlap},
\[
 |\rho_{K,J}^\nu|
 \leq
 C\norm{u}_{L^\infty(\D)}
 2^{-3(j-k)/2}
 \bigl(1+d_k(K,J)\bigr)^{-3}.
\]
Since $K$ has $2^{j-k}$ descendants at level $j$, the
Cauchy--Schwarz inequality gives
\begin{align}
 \left|
 \sum_{J\in\mathcal J_{\nu,K}}\alpha_Jd_J
 -
 \left\langle
 A_uW(z^{N_\nu}F),
 W(z^{N_\nu}\phi_K)
 \right\rangle
 \right|
 \leq{}&
 C\norm{u}_{L^\infty(\D)}
 \sqrt{\omega(T(K))}
 \left(
 \sum_{j=m_\nu}^{m_{\nu+1}-1}
 2^{-2(j-k)}
 \right)^{1/2}.
 \label{eq:adaptive-block-freezing-error}
\end{align}
For  $m_\nu\leq j<m_{\nu+1}=h_\nu+d_\epsilon+1$ and $N_\nu=2^{h_\nu}$, we have
\[
 W(z^{N_\nu}F)
 =
 \Pi_{h_\nu,d_\epsilon+2}W(z^{N_\nu}F).
\]
Moreover, $h_\nu\geq k+4$ implies
\[
 W(z^{N_\nu}\phi_K)
 =
 \Pi_{h_\nu,d_\epsilon+2}W(z^{N_\nu}\phi_K).
\]
Consequently, by \eqref{eq:stopping-block-band-bound},
\[
 \left|
 \left\langle
 A_uW(z^{N_\nu}F),
 W(z^{N_\nu}\phi_K)
 \right\rangle
 \right|
 \leq
 \norm{A_u}_{N;d_\epsilon+2}
 \norm{z^{N_\nu}F}_{A^2_\omega}
 \norm{z^{N_\nu}\phi_K}_{A^2_\omega}.
\]
By Lemma~\ref{lem:modulated-frame-norm-only-k},
\[
 \norm{z^{N_\nu}\phi_K}_{A^2_\omega}
 \leq
 C
 \left(
 \frac{\mathfrak m(N_\nu)}
 {\mathfrak m(2^{k+1})}
 \right)^{1/2}.
\]
Recalling \eqref{eq:F-adaptive-norm}, we have
\[
 \left|
 \left\langle
 A_uW(z^{N_\nu}F),
 W(z^{N_\nu}\phi_K)
 \right\rangle
 \right|
 \leq
 C\norm{A_u}_{N;d_\epsilon+2}
 \sqrt{\omega(T(K))}
 \left(
 \frac{\mathfrak m(N_\nu)}
 {\mathfrak m(2^{k+1})}
 \right)^{1/2}.
\]
Combining this estimate with
\eqref{eq:adaptive-block-freezing-error} and taking the supremum over
all unit vectors $(\alpha_J)$, we obtain
\begin{align}
 \frac1{\omega(T(K))}
 \sum_{J\in\mathcal J_{\nu,K}}|d_J|^2
 \leq
 C\bigl(\norm{A_u}_{N;d_\epsilon+2}\bigr)^2
 \frac{\mathfrak m(N_\nu)}
 {\mathfrak m(2^{k+1})}
 +
 C\norm{u}_{L^\infty(\D)}^2
 \sum_{j=m_\nu}^{m_{\nu+1}-1}2^{-2(j-k)}.
 \label{eq:adaptive-block-d-bound}
\end{align}
Since $x_0=2^{k+M-1}\geq2^{k+1}$, by the monotonicity of $\mathfrak m$ and \eqref{eq:adaptive-modulation-sum},
\[
 \sum_{\nu=0}^\infty
 \frac{\mathfrak m(N_\nu)}
 {\mathfrak m(2^{k+1})}
 \leq C_\epsilon.
\]
The stopping blocks partition the integers $j\geq k+M$. Summing
\eqref{eq:adaptive-block-d-bound} over $\nu$ therefore gives
\begin{equation}\label{eq:dJ-Carleson-one-sided}
 \frac1{\omega(T(K))}
 \sum_{j\geq k+M}
 \sum_{\substack{J\in\mathcal Q_j\\J\subset K}}
 |d_J|^2
 \leq
 C_\epsilon
 \bigl(\norm{A_u}_{N;d_\epsilon+2}\bigr)^2
 +
 C2^{-2M}\norm{u}_{L^\infty(\D)}^2.
\end{equation}

Hence, by \eqref{eq:eJ-Carleson-one-sided} and \eqref{eq:dJ-Carleson-one-sided}, and
$2^{-6M}\leq2^{-2M}$, we have
\[
 \frac1{\omega(T(K))}
 \sum_{j\geq k+M}
 \sum_{\substack{J\in\mathcal Q_j:J\subset K}}
 |c_J|^2
 \leq
 C_0\epsilon\norm{u}_{L^\infty(\D)}^2
 +
 C_\epsilon\bigl(\norm{A_u}_{N;d_\epsilon+2}\bigr)^2
 +
 C_\epsilon2^{-2M}\norm{u}_{L^\infty(\D)}^2.
\]
Taking the supremum over $K\in\mathcal Q_k,k\geq N$, we get \eqref{eq:upper-Carleson-main-estimate} and complete the proof.
\end{proof}

\begin{proof}[Proof of Proposition~\ref{cor:scalar-Carleson-tail}]
Fix $0<\epsilon\leq1/4$, and let $d_\epsilon$, $M_\epsilon$, and $N_\epsilon$ be the parameters in
Proposition~\ref{prop:upper-Carleson-one-sided}. By Proposition~\ref{prop:scalar-fixed-band-compression},
\[
 \lim_{N\to\infty}
 \norm{A_u}_{N;d_\epsilon+2}
 =
 0.
\]
Consequently, for every $M\geq M_\epsilon$,
\[
 \limsup_{N\to\infty}\mathfrak C_N^M(u)
 \leq
 C_0\epsilon\norm{u}_{L^\infty(\D)}^2
 +
 C_\epsilon2^{-2M}\norm{u}_{L^\infty(\D)}^2.
\]
Letting first $M\to\infty$ and then $\epsilon\downarrow0$ proves
\eqref{eq:Carleson-tail-double-limit}.
\end{proof}

\section{The proof of Proposition \ref{prop:conuterexample}}\label{sec:counter}
Recall the weight considered in Proposition \ref{prop:conuterexample} is
\begin{equation}\label{eq:weight-new}
 \omega(r)=
 \begin{cases}
 \displaystyle \frac{e}{4}, &0\le r\le 1-e^{-1},\\[6pt]
 \displaystyle \frac{1}{(1-r)\left(\log\frac{e}{1-r}\right)^2}, &1-e^{-1}<r<1.
 \end{cases}
\end{equation}
Throughout this section $\omega$ is the one defined by \eqref{eq:weight-new}. 
\begin{lemma}\label{lem:log-subharmonic}
The weight $\omega$ is a log-subharmonic weight on $\mathbb{D}$, and hence it is a $\widehat{\mathcal{D}}$ weight.
\end{lemma}
\begin{proof}
Let $\omega(z)=\omega(|z|)$ be the weight induced by \eqref{eq:weight-new}.
Set
\[
 h(r)=\log \omega(r)=
 \begin{cases}
  \displaystyle 1-\log4,
       & 0\le r\le 1-e^{-1},\\[6pt]
  \displaystyle \log\frac{e}{1-r}-1-2\log \log\frac{e}{1-r},
       & 1-e^{-1}<r<1.
 \end{cases}
\] 
Then \(h\in C^1([0,1))\) and is piecewise \(C^2\) on $[0,1)$. Moreover, 
write
\[t=t(r)=\log\frac{e}{1-r}.\]
Then, for \(r> 1-e^{-1}\),
\[
 h'(r) = \frac{t-2}{t(1-r)}\geq0
\]
and
\[
 h''(r) = \frac{1-\frac2t+\frac2{t^2}}{(1-r)^2}>0.
\]
By the formula of radial Laplace \cite{Shi02}, for \(r=|z|>1-e^{-1}\), 
\begin{align*}
 \Delta \log\omega(z) =h''(r)+\frac1r h'(r)\geq 0
\end{align*}
and for \(0\leq r<1-e^{-1}\), \(\Delta\log\omega(z)=0\). It follows that $\log \omega$ is subharmonic on the unit $\mathbb{D}$. 
By \cite[Theorem 1]{PRW19}, $\omega$ is a $\widehat{\mathcal{D}}$ weight. This completes the proof. 
\end{proof}

The following result follows from direct computations, we include a proof here since it plays a central role in our construction.
\begin{lemma}\label{lem:moments} There exist constants $C_1,C_2>0$ such that
\[
 C_1\le (1+\log x)\omega_x\le C_2,\qquad x\ge1.
\]
\end{lemma}
\begin{proof}
Recall that $$\omega_x=\int_0^1r^x\omega(r)\dd r.$$
By \eqref{eq:weight-new} and the change of variables $t=\log\frac{e}{1-r}$,
\[
 \omega_x=\frac e4\int_0^{1-e^{-1}}r^x\dd r+\int_2^\infty\left(1-e^{1-t}\right)^x\frac{\dd t}{t^2}.
\]
Assume that $x\ge e^3$. If $t\ge1+\log x$, then $e^{1-t}\le x^{-1}$, so
\[
 \omega_x\ge \left(1-\frac1x\right)^x\int_{1+\log x}^\infty\frac{\dd t}{t^2}\ge \frac{c}{1+\log x}.
\]
For the reverse estimate, $1-s\le e^{-s}$ gives
\[
 \omega_x\le C(1-e^{-1})^x+\int_2^{1+\log x}e^{-e^{1+\log x-t}}\frac{\dd t}{t^2}+\frac1{1+\log x}.
\]
Splitting the integral at $(1+\log x)/2$, we have
\[
 \int_2^{1+\log x}e^{-e^{1+\log x-t}}\frac{\dd t}{t^2}
 \le Ce^{-e^{(1+\log x)/2}}+\frac{C}{1+\log x}
 \le \frac{C}{1+\log x}.
\]
Since $(1-e^{-1})^x\le C/(1+\log x)$, the upper bound follows. Adjusting the constants on $1\le x\le e^3$ completes the proof.
\end{proof}
\subsection{The construction}
Set $\lambda_1=16\log 2$, and for $j\ge 1$ define recursively
\begin{align}\label{eq:lambdaj}
\lambda_{j+1}=\frac{\lambda_j^2}{\log 2}.
\end{align}
We further define
\begin{align}\label{eq:ejfj}
 E_j=\left\{z:\,2\lambda_j\le\log\frac{e}{1-|z|}\le3\lambda_j\right\},\qquad
 F_j=\left\{z:\,7\lambda_j\le\log\frac{e}{1-|z|}\le8\lambda_j\right\}.
\end{align}
Let
\begin{equation}\label{eq:symbol-u}
u(z)=\sum_{j\ge1}\mathbf{1}_{E_j}(z).
\end{equation}
For each $j\geq 1$, set $l_j=e^{6\lambda_j}$ and
\begin{equation}\label{eq:symbol-vj}
v_j(re^{i\theta})=\mathbf{1}_{F_j}(r)e^{il_j\theta},\qquad z=re^{i\theta}\in\mathbb{D}.
\end{equation}
Then consider a function
\begin{equation}\label{eq:symbol-v}
 v(z)=
 v(re^{i\theta})=\sum_{j\ge1}v_j(re^{i\theta}):=\sum_{j\ge1}\mathbf{1}_{F_j}(r)e^{il_j\theta}.
\end{equation}
Since the sets $E_j$ and $F_j$ are pairwise disjoint, we have $0\le u\le1$ and $|v|\le1$. 

Recall the standard orthonormal basis of radial weighted Bergman space $A_\omega^2$ is
\[
 e_n(z)=\frac{z^n}{\sqrt{2\omega_{2n+1}}},\qquad n\ge0.
\]
Under this basis, for $n\ge0$ and $j\ge1$, we have
\begin{align}\label{eq:shft-vj}
T_{\omega,v_j}e_n=\alpha_{j,n}e_{n+l_j}
\end{align}
and
\begin{align}\label{eq:shft-u}
 T_{\omega,u}e_n=\beta_ne_n,
\end{align}
where
\begin{align}\label{eq:tvalphaj}
 \alpha_{j,n}=\frac1{\sqrt{\omega_{2n+1}\omega_{2(n+l_j)+1}}}
 \int_{F_j}r^{2n+l_j+1}\omega(r)\,dr
\end{align}
and
\begin{align}\label{eq:tubeta}
 \beta_n=\frac1{\omega_{2n+1}}\int_0^1 u(r)r^{2n+1}\omega(r)\,dr
\end{align}
Recall $S=T_{\omega,v}T_{\omega,u}$, for each $j\geq 1$, let $n_j=e^{\lambda_j}$, then
\[
 Se_{n_j}=\beta_{n_j}\sum_{k\ge1}\alpha_{k,n_j}e_{n_j+l_k}.
\]
Hence,
\begin{equation}\label{eq:norm-snj}
\| Se_{n_j}\|_{A_\omega^2}^2= \beta_{n_j}^2  \sum_{k\geq1}\abs{\alpha_{k,n_j}}^2.
\end{equation}
\textbf{Claim A.} There is a constant $c_0>0$ such that
\begin{equation}\label{eq:lowerbetaj}
  \beta_{n_j}^2\geq c_0,\qquad \forall j\geq 1.
\end{equation}
For $r\in E_j$, $1-r\le e^{1-2\lambda_j}$. Since $n_j=e^{\lambda_j}$, we have
\[
 r^{2n_j+1}\ge1-(2n_j+1)e^{1-2\lambda_j}
 =1-2e^{1-\lambda_j}-e^{1-2\lambda_j}>\frac12,
\]
By \eqref{eq:symbol-u} and \eqref{eq:tubeta},
\[
 \beta_{n_j}\ge\frac1{\omega_{2n_j+1}}\int_{E_j}r^{2n_j+1}\omega(r)\,dr\geq \frac{1}{2\omega_{2n_j+1}}\int_{E_j}\omega(r)dr.
\]
On the other hand,
\[
 \int_{E_j}\omega(r)\dd r = \int_{1-e^{1-2\lambda_j}}^{1-e^{1-3\lambda_j}} \frac{\dd r} {(1-r)\left(\log\frac{e}{1-r}\right)^2}.
\]
By using the change of variables
\[
 t=\log\frac{e}{1-r},
\] 
we have
\[
\int_{E_j}\omega(r)\dd r=\int_{2\lambda_j}^{3\lambda_j}\frac{\dd t}{t^2}=\frac{1}{6\lambda_j}.
\]
Therefore,
\[
\beta_{n_j}\ge \frac{1}{12\lambda_j\omega_{2n_j+1}}.
\]
Since $n_j=e^{\lambda_j}$, by Lemma \ref{lem:moments} and the monotonicity of $\omega_x$,  
\[
\log n_j \omega_{2n_j+1}\leq (1+\log n_j)\omega_{n_j}\leq C_2.
\]
Let $c_0=\frac{1}{144C_2^2}>0$, we verify the Claim A.

\noindent \textbf{Claim B.} There is a constant $c_1>0$ such that
\begin{align}\label{eq:lowalphaj}
  \sum_{k\geq1}\abs{\alpha_{k,n_j}}^2\geq c_1\qquad \forall j\geq 1.
\end{align}
First of all,
\[
\sum_{k\geq1}\abs{\alpha_{k,n_j}}^2\geq \abs{\alpha_{j,n_j}}^2.
\]
In addition, if $r\in F_j$, then $1-r\le e^{1-7\lambda_j}$, and hence
\[
 r^{2n_j+l_j+1} \ge1-(2n_j+l_j+1)e^{1-7\lambda_j} =1-2e^{1-6\lambda_j}-e^{1-\lambda_j}-e^{1-7\lambda_j}>\frac{1}{2}.
\]
By \eqref{eq:symbol-vj} and \eqref{eq:tvalphaj}, we have
\[
\alpha_{j,n_j}\geq \frac1{2\sqrt{\omega_{2n_j+1}\omega_{2(n_j+l_j)+1}}}\int_{F_j}\omega(r)\,dr.
\]
Observe that
\begin{align}\label{eq:omegafj}
 \int_{F_j}\omega(r)\dd r=\int_{7\lambda_j}^{8\lambda_j}\frac{\dd t}{t^2}=\frac1{56\lambda_j},
\end{align}
we have
\[
\alpha_{j,n_j}\geq \frac{1}{120\lambda_j\sqrt{\omega_{2n_j+1}\omega_{2(n_j+l_j)+1}}}.
\]
With the same arguments as the proof of Claim A, there is a constant $c_1$ such that
\[
\alpha_{j,n_j}^2>c_1,
\]
which completes the proof of Claim B.

By \eqref{eq:norm-snj}, Claim A and Claim B,
\[
\| Se_{n_j}\|_{A_\omega^2}^2\geq c_0c_1>0.
\]
Since $e_{n_j}$ weakly converges to $0$ in $A_\omega^2$, we have
\begin{lemma}\label{lem:not-compact}
The operator $S=T_{\omega,v}T_{\omega,u}$ is not compact on $A_\omega^2$.
\end{lemma}
\subsection{The Berezin transform}
For any $j\geq 1$, let
\[
 S_j=T_{\omega,v_j}T_{\omega,u}.
\]
Then $\sum_{j=1}^nS_j$ strongly converges to the operator $S$. Hence, for every $z\in\D$,
\begin{equation}\label{eq:berezs}
 \widetilde S(z)=\sum_{j=1}^\infty\widetilde S_j(z).
\end{equation}
For each $j\geq 1$, direct computations show that
\begin{equation}\label{eq:Sj-action}
 S_je_n=\beta_n\alpha_{j,n}e_{n+l_j},\qquad n\ge0.
\end{equation}
Moreover, the Berezin transform of $S_j$ is
\begin{equation}\label{eq:Berezin-Sj}
 \widetilde S_j(z)  =\sum_{n=0}^\infty\beta_n\alpha_{j,n}
 \frac{z^{l_j}|z|^{2n}}{2B_z^\omega(z)\sqrt{\omega_{2n+1}\omega_{2(n+l_j)+1}}}.
\end{equation}
\textbf{Claim C}. For each $j\geq 1$
\[
\lim_{|z|\to 1^-}\widetilde S_j(z)=0.
\]
First of all, for any $z=re^{i\theta}\in\mathbb{D}$, we have
\begin{equation}\label{eq:abBerezin-Sj}
 |\widetilde S_j(z)| =\sum_{n=0}^\infty\beta_n\alpha_{j,n}\frac{r^{2n+l_j}}{2B_z^\omega(z)\sqrt{\omega_{2n+1}\omega_{2(n+l_j)+1}}}.
\end{equation}
\textbf{Subclaim} For any $j\geq 1$,
\begin{align}\label{eq:abn}
\lim_{n\to\infty}\beta_n\alpha_{j,n}=0.
\end{align}
By \eqref{eq:tubeta}, $0<\beta_n\leq 1$. It remains to show
\[
\lim_{n\to\infty}\alpha_{j,n}=0.
\]
Indeed, when $j\geq 1$ is fixed, by the definition of \eqref{eq:ejfj}, for any $r\in F_j$
\[
r\leq 1-e^{1-8\lambda_j}:=\rho_j<1.
\]
Hence,
\[
 \int_{F_j}r^{2n+l_j+1}\omega(r)\dd r\leq \rho_j^{2n+l_j+1} \int_{F_j}\omega(r)dr.
\]
By \eqref{eq:omegafj},
\[
\int_{F_j}r^{2n+l_j+1}\omega(r)\dd r\leq \frac{\rho_j^{2n+l_j+1}}{56\lambda_j}.
\] 
Using Lemma~\ref{lem:moments} and $l_j=e^{6\lambda_j}$,
\[
 \frac1{\sqrt{\omega_{2n+1}\omega_{2(n+l_j)+1}}} \le C_1\sqrt{\lambda_j}\,(1+\log(n+2)).
\]
Therefore,
\[
\alpha_{j,n}=\frac1{\sqrt{\omega_{2n+1}\omega_{2(n+l_j)+1}}}\int_{F_j}r^{2n+l_j+1}\omega(r)\dd r\leq C_1
\frac{\rho_j^{2n+l_j+1}}{56\lambda_j}\sqrt{\lambda_j}\,(1+\log(n+2)),
\]
and \eqref{eq:abn} holds. This proves the subclaim.

Hence, for any given $\epsilon>0$, we may pick a constant $N$ such that 
\[
|\beta_n\alpha_{j,n}|\leq \epsilon \qquad \forall n\geq N.
\]
Therefore,
\[
|\widetilde S_j(z)|\leq \sum_{n=0}^N \beta_n\alpha_{j,n}\frac{r^{2n+l_j}}{2B_z^\omega(z)\sqrt{\omega_{2n+1}\omega_{2(n+l_j)+1}}}+
\epsilon \sum_{n=N+1}^\infty \frac{r^{2n+l_j}}{2B_z^\omega(z)\sqrt{\omega_{2n+1}\omega_{2(n+l_j)+1}}}.
\]
By Cauchy--Schwarz,
\begin{equation}\label{eq:Berezin-Sj-CS}
 \sum_{n=0}^\infty
 \frac{r^{2n+l_j}}{2B_z^\omega(z)\sqrt{\omega_{2n+1}\omega_{2(n+l_j)+1}}} \le1.
\end{equation}
Then
\[
|\widetilde S_j(z)|\leq \sum_{n=0}^N \beta_n\alpha_{j,n}\frac{r^{2n+l_j}}{2B_z^\omega(z)\sqrt{\omega_{2n+1}\omega_{2(n+l_j)+1}}}+\epsilon.
\]
By Lemma \ref{lem:norm-estimates-for-reproducing-kernel},
\[
\lim_{|z|\to 1^-}B_z^\omega(z)=+\infty,
\]
and hence
\[
\lim_{|z|\to 1^-}|\widetilde S_j(z)|\leq \epsilon.
\]
This proves Claim C.

\noindent\textbf{Claim D.} We have the following estimate
\begin{equation}\label{eq:summable-bound}
 \sum_{j=1}^\infty\sup_{z\in\D}|\widetilde S_j(z)|<\infty.
\end{equation}

Recall \eqref{eq:abBerezin-Sj}, for any $j\geq 1$ and $z=re^{i\theta}$, 
\begin{equation}
 |\widetilde S_j(z)| =\sum_{n=0}^\infty\beta_n\alpha_{j,n}\frac{r^{2n+l_j}}{2B_z^\omega(z)\sqrt{\omega_{2n+1}\omega_{2(n+l_j)+1}}}.
\end{equation}
We write
\[
|\widetilde S_j(z)|=\bigg(\sum_{0\leq n\leq e^{4\lambda_j}}+\sum_{e^{4\lambda_j}<n<e^{9\lambda_j}}+\sum_{n\geq e^{9\lambda_j}}\bigg)
\beta_n\alpha_{j,n}\frac{r^{2n+l_j}}{2B_z^\omega(z)\sqrt{\omega_{2n+1}\omega_{2(n+l_j)+1}}}.
\]
We first observe the trivial lower bounds of $B_z^\omega(z)$
 \[
 B_z^\omega(z)=\sum_{n=0}^\infty\frac{|z|^{2n}}{2\omega_{2n+1}} \ge \frac{1}{10}\frac{1}{1-|z|}.
 \]
To estimate the left summation, since $n\leq e^{4\lambda_j}$ and $l_j=e^{6\lambda_j}$, by Lemma \ref{lem:moments},
\[
 \frac{1}{\sqrt{\omega_{2n+1}\omega_{2(n+l_j)+1}}} \leq C_1\sqrt{\bigl(1+\log(2n+1)\bigr)\bigl(1+\log(2(n+l_j)+1)\bigr)}\leq C_1\lambda_j.
\] 
Hence,
\begin{align*}
\sum_{0\leq n\leq e^{4\lambda_j}}\beta_n\alpha_{j,n}\frac{r^{2n+l_j}}{2B_z^\omega(z)\sqrt{\omega_{2n+1}\omega_{2(n+l_j)+1}}}
&\leq 5C_1\lambda_j(1-|z|)\sum_{0\leq n\leq e^{4\lambda_j}}\beta_n\alpha_{j,n}|z|^{2n+l_j}\\
&\leq C_1' \lambda_j e^{4\lambda_j}(1-|z|)|z|^{l_j}.
\end{align*}
In the last inequality, we also use the fact $\alpha_{j,n}\leq 1$ and $\beta_n\leq 1$. For any $z\in\mathbb{D}$,
\[
 (1-|z|)|z|^{l_j}\leq\frac1{l_j+1}\leq e^{-6\lambda_j},
\]
then
\begin{align}\label{eq:left-term}
\sum_{0\leq n\leq e^{4\lambda_j}}\beta_n\alpha_{j,n}\frac{r^{2n+l_j}}{2B_z^\omega(z)\sqrt{\omega_{2n+1}\omega_{2(n+l_j)+1}}}
\leq C_1'\lambda_j e^{-2\lambda_j}.
\end{align}

Next, we consider the middle range $e^{4\lambda_j}<n<e^{9\lambda_j}$. Observe that there is a constant $C_2>0$ which is independent of $j$ such that
\[
 \sum_{k\leq j} \int_{E_k}r^{2n+1}\omega(r)\,\dd r
 \leq
 C_2 e^{-2e\,e^{\lambda_j}}.
\]
By \eqref{eq:lambdaj}, for all $j\geq 1$
\[
\frac{1}{\lambda_{j+1}}=\frac{\log 2}{\lambda_j^2}.
\]
Using \(\int_{E_k}\omega(r)\,\dd r=\frac{1}{6\lambda_k}\), we obtain
 \begin{align*}
 \sum_{k\geq j+1}
 \int_{E_k}r^{2n+1}\omega(r)\,\dd \leq
 \frac16\sum_{k\geq j+1}\frac1{\lambda_k}  
 \leq
 \frac{C_2'}{\lambda_{j+1}}
 \end{align*}
Since \(n<e^{9\lambda_j}\), by Lemma~\ref{lem:moments} and recall \eqref{eq:tubeta}, there is a constant $C_2''$ which is also independent of $j$ such that
 \begin{align*}
 \beta_n
 \leq
 C_2''\left(
 \lambda_j e^{-2e\,e^{\lambda_j}}
 +\frac1{\lambda_j}
 \right).
 \end{align*}
By \eqref{eq:Berezin-Sj-CS} and $\alpha_{j,n}\leq 1$, there is a constant $C_2'''$ that is independent of $j$ such that
\begin{equation}\label{eq:middle-estimate}
\sum_{e^{4\lambda_j}<n<e^{9\lambda_j}}\beta_n\alpha_{j,n}\frac{r^{2n+l_j}}{2B_z^\omega(z)\sqrt{\omega_{2n+1}\omega_{2(n+l_j)+1}}}
 \leq
 C_2'''\left(
 \frac1{\lambda_j}
 +\lambda_j e^{-2e\,e^{\lambda_j}}
 \right).
\end{equation}

Finally, we estimate the term \(n\geq e^{9\lambda_j}\). To do so, we shall show that there is a constant $C_3>0$ which is independent of $j$ such that
\begin{equation}\label{eq:high-alpha-estimate}
 \sup_{n\geq e^{9\lambda_j}}\alpha_{j,n}\leq C_3 e^{-e^{\lambda_j}}.
\end{equation}
For $r\in F_j$, $\log\frac{e}{1-r}\leq8\lambda_j$, then
\[
 r^{2n+l_j+1} \leq \exp\bigl(-(2n+l_j+1)(1-r)\bigr) \leq \exp\bigl(-2e\,n e^{-8\lambda_j}\bigr).
\]
By \eqref{eq:omegafj},
\[
 \int_{F_j}s^{2n+l_j+1}\omega(r)\,\dd r
 \leq
 \frac1{56\lambda_j}
 \exp\bigl(-2e\,n e^{-8\lambda_j}\bigr).
\]
Recall \(l_j=e^{6\lambda_j}\leq n\), by Lemma~\ref{lem:moments}, we have
\[
 \frac1{\sqrt{\omega_{2n+1}\omega_{2(n+l_j)+1}}}  \leq C_3'\bigl(1+\log(n+2)\bigr)
\]
for some $C_3'>0$. Therefore, there is a constant $C_3''>0$ such that
\[
 \alpha_{j,n} \leq C_3''\frac{1+\log(n+2)}{\lambda_j}\exp\bigl(-2e\,n e^{-8\lambda_j}\bigr).
\]
Using another elementary fact that: 
\[
 \frac{1+\log(n+2)}{\lambda_j}\leq Cne^{-8\lambda_j}.
\]
Write $x=ne^{-8\lambda_j}$, we have
\[
\alpha_{j,n}\leq C_3'' xe^{-ex}e^{-ex}.
\]
Recall the following elementary fact that for every $a>0$ and $b>0$,
\begin{align}\label{eq:elest}
 \sup_{x>0}x^b e^{-ax}<\infty.
\end{align}
Since \(x\geq e^{\lambda_j}\), we have
\[
 \alpha_{j,n}
 \leq
 C\exp\bigl(-e\,e^{\lambda_j}\bigr),
 \qquad n\geq e^{9\lambda_j}.
\]
Therefore, we get the desired inequality \eqref{eq:high-alpha-estimate}.

By \eqref{eq:high-alpha-estimate}, \eqref{eq:Berezin-Sj-CS} and $\beta_{n}\leq 1$,
\begin{equation}\label{eq:right-estimate}
\sum_{n\geq e^{9\lambda_j}}\beta_n\alpha_{j,n}\frac{r^{2n+l_j}}{2B_z^\omega(z)\sqrt{\omega_{2n+1}\omega_{2(n+l_j)+1}}}
 \leq C_3 e^{-e^{\lambda_j}}.
\end{equation}
Then by \eqref{eq:left-term}, \eqref{eq:middle-estimate} and \eqref{eq:right-estimate}, there are positive constant $C_4>0$ such that 
\[
|\widetilde S_j(z)|\leq  C_4\left(
 \lambda_j e^{-2\lambda_j}
 +\frac1{\lambda_j}
 +\lambda_j e^{-2e\,e^{\lambda_j}}
 +e^{-e^{\lambda_j}}
 \right):=C_4\delta_j.
\]
Therefore,
\[
 \sum_{j=1}^\infty \sup_{z\in\D}|\widetilde S_j(z)| \leq C_4\sum_{j=1}^\infty\delta_j.
\]
We now verify that
\[
\sum_{j=1}^\infty\delta_j<\infty.
\]
By the construction \eqref{eq:lambdaj} of $\lambda_j$, we have
\[
 \frac{\lambda_{j+1}}{\lambda_j}
 =
 \frac{\lambda_j}{\log 2},
\]
and $\lambda_1/\log 2=16$. Moreover, for all $j\geq 1$,
\begin{equation}\label{eq:lambda-geometric-growth}
 \lambda_j\geq16^{j-1}\lambda_1.
\end{equation}

On the other hand, observe that
\[
\lambda_j e^{-2\lambda_j} = \frac{\lambda_j^2e^{-2\lambda_j}}{\lambda_j}, 
\]
\[
 \lambda_j e^{-2e\,e^{\lambda_j}} \leq \lambda_j e^{-2e\lambda_j}
\]
and
\[
 e^{-e^{\lambda_j}} \leq e^{-\lambda_j} = \frac{\lambda_j e^{-\lambda_j}}{\lambda_j},
\]
there is a constant $C_4'>0$ such that for all $j\geq 1$
\[
\delta_j \leq\frac{C_4'}{\lambda_j}.
\]
By \eqref{eq:lambda-geometric-growth}, we have
\[
\sum_{j\geq 1}\delta_j<\infty,
\]
which completes the proof of Claim D.

\begin{lemma}\label{lem:berezin-zero}
The Berezin transform $\widetilde{S}$ of the operator $S=T_{\omega,v}T_{\omega,u}$ satisfies
\[
\lim_{|z|\to 1^-}\widetilde{S}(z)=0.
\]
\end{lemma}

\begin{proof}
Recall \eqref{eq:berezs}, for every $z\in\D$,
\begin{equation}\label{eq:berezs-new}
 \widetilde S(z)=\sum_{j=1}^\infty\widetilde S_j(z).
\end{equation}
It follows from Claim~D that
\[
 \sum_{j=1}^{\infty}
 \sup_{z\in\D}|\widetilde S_j(z)|<\infty.
\]
Therefore, the left series in \eqref{eq:berezs-new} converges uniformly on the unit disk \(\D\). For any given $\epsilon>0$, there is a constant $N>0$ such that
\[
 \sum_{j>N}\sup_{z\in\D}|\widetilde S_j(z)|<\frac{\epsilon}{2}.
\]
On the other hand, by Claim C, for any $1\leq j\leq N$
\[
\lim_{|z|\to 1^-}\widetilde S_j(z)=0.
\]
Then there is a constant $\rho_\epsilon \in(0,1)$ such that
\[
\sum_{j=1}^N |\widetilde S_j(z)| <\frac{\epsilon}{2}, \qquad \rho_\epsilon<|z|<1.
\]
Hence, for any $z\in\mathbb{D}$ with $|z|>\rho_\epsilon$, we have
 \begin{align*}
 |\widetilde S(z)| \leq \sum_{j=1}^{N}|\widetilde S_j(z)|+\sum_{j>N}|\widetilde S_j(z)| \leq
 \sum_{j=1}^{N}|\widetilde S_j(z)| +\sum_{j>N}\sup_{\zeta\in\D}|\widetilde S_j(\zeta)|<\epsilon.
 \end{align*}
This means
\[
\lim_{|z|\to 1^-}\widetilde S(z)=0,
\]
which completes the whole proof.
\end{proof}

\begin{proof}[Proof of Proposition \ref{prop:conuterexample}]
The proof then follows from Lemma \ref{lem:log-subharmonic}, Lemma \ref{lem:not-compact} and Lemma \ref{lem:berezin-zero}.
\end{proof}
\section{Appendix}\label{S:app-carleson}
In this appendix, we prove the Carleson embedding estimate used in Proposition~\ref{lem:shifted-Carleson}. 
Consider the weighted maximal operator
\begin{equation}\label{eq:tent-maximal}
 \mathcal M_\omega f(z)
 =
 \sup_{\substack{I\subset\mathbb T\\z\in T(I)}}
 \frac{1}{\omega(T(I))}
 \int_{T(I)}|f(\zeta)|\omega(\zeta)\,dA(\zeta).
\end{equation}
By \cite[Theorem~4]{PR-embed}, we have the following $L^2$ bounded estimate.

\begin{lemma}\label{thm:tent-maximal}
Let $\omega\in\widehat{\mathcal D}$. There exists a constant $C=C(\omega)>0$ such that
\begin{equation}\label{eq:tent-maximal-L2}
 \norm{\mathcal M_\omega f}_{L_\omega^2}
 \leq C\norm{f}_{L_\omega^2},
 \qquad f\in L_\omega^2.
\end{equation}
\end{lemma}

Recall also that the operator $\mathcal C_s$ is defined by \eqref{eq:smooth-low-pass}:
\begin{equation*}
 \mathcal C_sf
 =
 \ip f{e_0}e_0
 +
 \sum_{n=1}^{\infty}
 \left(
  \sum_{k=0}^{\infty}\chi(2^{k-s}n)^2
 \right)
 \ip f{e_n}e_n 
\end{equation*}
and its equivalent form form \eqref{eq:smooth-low-pass-coefficients}
\begin{equation*}
 \mathcal C_mf(z)
 =
 c_0+
 \sum_{n=1}^{\infty}
 \left(
  \sum_{k=0}^{\infty}\chi(2^{k-m}n)^2
 \right)c_nz^n.
\end{equation*}
where $f(z)=\sum_{n=0}^\infty c_nz^n$.

\begin{lemma}\label{lem:low-frequency-localization}
Let $\omega\in\widehat{\mathcal D}$. There exists a constant $C=C(\omega,\chi)>0$ such that for $I\in\mathcal Q_s,s\geq 1$, $a\in \mathbb D$ with $a/|a|\in I$ and $|a|\geq r_s=1-2^{-s}$, we have 
\begin{equation}\label{eq:low-frequency-localization}
 |\mathcal C_sf(a)|
 \leq
 C\inf_{z\in T(I)}\mathcal M_\omega f(z)\qquad \forall f\in A_\omega^2.
\end{equation}
\end{lemma}
\begin{proof}
Write
$
 f(z)=\sum_{n=0}^\infty c_nz^n
$
and, for $n\geq1$, set
\[
 \eta_s(n) = \sum_{k=0}^\infty\chi(2^{k-s}n)^2.
\]
For $a\in \mathbb{D}$, write $a=\rho_ae^{i\theta_a}$ with $\rho_a=|a|,\theta_a\in[-\pi,\pi)$. By the equivalent form \eqref{eq:smooth-low-pass-coefficients} of $\mathcal C_s$, 
\[
 \mathcal C_sf(a)
 =
 c_0+\sum_{n=1}^\infty
 \eta_s(n)c_na^n
 =
 c_0+\sum_{n=1}^\infty
 \eta_s(n)c_n\rho_a^ne^{in\theta_a}.
\]
Fix $\rho\in[r_s,1)$ and define
\[
 m_{s,\rho}(x)
 =
 \begin{cases}
 \displaystyle
 \left(
  \sum_{k=0}^\infty\chi(2^{k-s}|x|)^2
 \right)
 \left(\dfrac{\rho_a}{\rho}\right)^x,
 &x\neq0,\\[8pt]
 1,&x=0.
 \end{cases}
\]
Then, we have 
\begin{equation}\label{eq:two-sided-low-pass-identity}
 \mathcal C_sf(a)
 =
 \frac{1}{2\pi}
 \int_{-\pi}^{\pi}
 \left[
  \sum_{n\in\mathbb Z}
  m_{s,\rho}(n)e^{in(\theta_a-t)}
 \right]
 f(\rho e^{it})\,dt.
\end{equation}
Since the support of $\chi$ is contained in the interval $(1/2,2)$ and \eqref{eq:chi-partition}, we have, for $|x|\geq2^{s+1}$, 
$
 \sum_{k=0}^\infty\chi(2^{k-s}|x|)^2=0.
$
Moreover, $\sum_{k=0}^\infty\chi(2^{k-s}|x|)^2=1$ for $0<|x|\leq2^{s-1}$. Therefore, $m_{s,\rho}\in C_c^\infty(\mathbb R)$. In particular,
\[
 \supp m_{s,\rho}\subseteq[-2^{s+1},2^{s+1}].
\]
Let $C_\omega>1$ be the constant given by \eqref{eq:Dhat-intro}.
By direct computations, there is a constant $C=C(\omega,\chi)$ such that
\[
 \max_{0\leq q\leq 3+\lceil\log_2C_\omega\rceil}
 2^{sq}
 \left\|
  \frac{d^q}{dx^q}m_{s,\rho}
 \right\|_\infty
 \leq C.
\]
Using Lemma~\ref{lem:periodic-multiplier} with the order $l=3+\lceil\log_2C_\omega\rceil$, we obtain there is a constant $C$ such that
\begin{equation}\label{eq:low-pass-kernel-decay}
 \left|
  \sum_{n\in\mathbb Z}m_{s,\rho}(n)e^{int}
 \right|
 \leq
 C2^s\bigl(1+2^s|t|\bigr)^{-3-\lceil\log_2C_\omega\rceil},
 \qquad |t|\leq\pi.
\end{equation}
For $\ell\geq0$, let $\Gamma_\ell$ be the arc centered at $e^{i\theta_a}$ with angular radius $\min\{\pi,2^{\ell-s}\}$.
Set $\Gamma_{-1}=\emptyset$, then
\[
 \mathbb T=\cup_{l\geq 0}(\Gamma_\ell\setminus\Gamma_{\ell-1})
\]
and
\[
 1+2^sd_{\mathbb T}(t,\theta_a)\geq 2^{\ell-1},  \qquad e^{it}\in \Gamma_\ell\setminus\Gamma_{\ell-1}.
\]
It follows from \eqref{eq:low-pass-kernel-decay} and \eqref{eq:two-sided-low-pass-identity} that
\begin{align}
 |\mathcal C_sf(a)|
 \leq
 C\sum_{\ell=0}^\infty
 2^{-\ell(2+\lceil\log_2C_\omega\rceil)}
 \frac{1}{|\Gamma_\ell|}
 \int_{\Gamma_\ell}|f(\rho e^{it})|\,dt.
 \label{eq:low-pass-angular-averages}
\end{align}
Then, we have
\begin{align}\label{eq:low-pass-radial-average}
 |\mathcal C_sf(a)|
 \int_{r_s}^1\rho\omega(\rho)\,d\rho 
 \leq
 C\sum_{\ell=0}^\infty
 2^{-\ell(2+\lceil\log_2C_\omega\rceil)}
 \frac{1}{|\Gamma_\ell|}
 \int_{r_s}^1\int_{\Gamma_\ell}
 |f(\rho e^{it})|\rho\omega(\rho)\,dt\,d\rho.
\end{align}
For $\ell\geq0$, set
$
 E_\ell
 =
 \left\{
 \rho e^{it}:r_s\leq\rho<1,\ e^{it}\in\Gamma_\ell
 \right\}.
$
Observe 
\[
 \int_{E_\ell}|f(\zeta)|\omega(\zeta)\,dA(\zeta)
 =
 \frac{1}{\pi}
 \int_{r_s}^1\int_{\Gamma_\ell}
 |f(\rho e^{it})|\rho\omega(\rho)\,dt\,d\rho,
\]
we have
\[
 |\mathcal C_sf(a)|
 \leq
 C\sum_{\ell=0}^\infty
 2^{-\ell(2+\lceil\log_2C_\omega\rceil)}
 \frac{1}{\omega(E_\ell)}
 \int_{E_\ell}|f(\zeta)|\omega(\zeta)\,dA(\zeta).
\]
Since $e^{i\theta_a}\in I$, for every $e^{it}\in I$ we have
$
 d_{\mathbb T}(t,\theta_a)  < 2^{-s}
$
and
$
 I\subseteq\Gamma_0\subseteq\Gamma_\ell
$ for $\ell\geq 0$.
Set
\[
 \rho_\ell
 =
 \max\left\{
 0,1-\frac4\pi|\Gamma_\ell|
 \right\}.
\]
Since
$
 |\Gamma_\ell|
 \geq
 |\Gamma_0|
 =
 2^{1-s},
$
we have
$
 \rho_\ell
 \leq
 1-2^{-s}
 =
 r_s
$, and
$$
 T(I)\cup E_\ell
 \subseteq
 T(\Gamma_\ell).
$$
This geometric configuration is illustrated in Figure~\ref{fig:low-frequency-tent-geometry}.
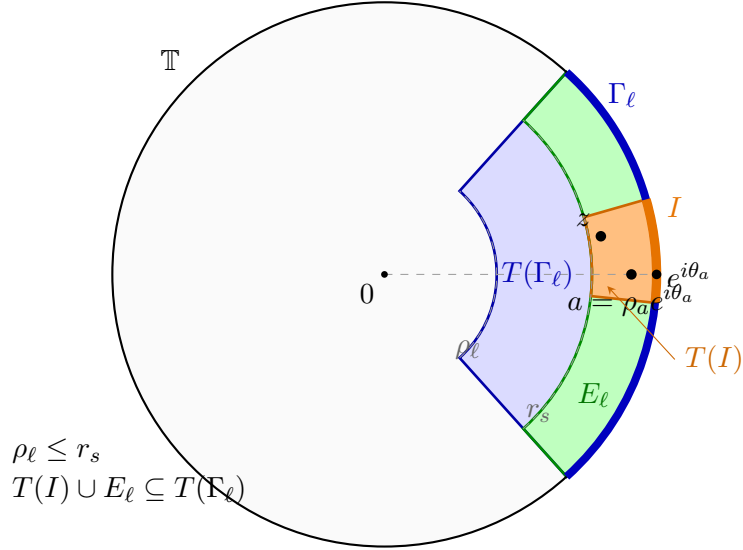
\begin{figure}[htbp]
\centering
\begin{tikzpicture}[scale=0.9,>=stealth]

\def\R{4}
\def\Rs{3.05}
\def\Rl{1.65}
\def\angleG{48}
\def\angleILeft{-6}
\def\angleIRight{16}

\fill[gray!4] (0,0) circle (\R);

\path[fill=blue!14]
 ({-\angleG}:\R)
 arc[start angle={-\angleG},
     end angle=\angleG,
     radius=\R]
 -- (\angleG:\Rl)
 arc[start angle=\angleG,
     end angle={-\angleG},
     radius=\Rl]
 -- cycle;

\path[fill=green!25]
 ({-\angleG}:\R)
 arc[start angle={-\angleG},
     end angle=\angleG,
     radius=\R]
 -- (\angleG:\Rs)
 arc[start angle=\angleG,
     end angle={-\angleG},
     radius=\Rs]
 -- cycle;

\path[fill=orange!50]
 (\angleILeft:\R)
 arc[start angle=\angleILeft,
     end angle=\angleIRight,
     radius=\R]
 -- (\angleIRight:\Rs)
 arc[start angle=\angleIRight,
     end angle=\angleILeft,
     radius=\Rs]
 -- cycle;

\draw[blue!65!black,line width=1pt]
 ({-\angleG}:\R)
 arc[start angle={-\angleG},
     end angle=\angleG,
     radius=\R]
 -- (\angleG:\Rl)
 arc[start angle=\angleG,
     end angle={-\angleG},
     radius=\Rl]
 -- cycle;

\draw[green!55!black,line width=1pt]
 ({-\angleG}:\R)
 arc[start angle={-\angleG},
     end angle=\angleG,
     radius=\R]
 -- (\angleG:\Rs)
 arc[start angle=\angleG,
     end angle={-\angleG},
     radius=\Rs]
 -- cycle;

\draw[orange!80!black,line width=1pt]
 (\angleILeft:\R)
 arc[start angle=\angleILeft,
     end angle=\angleIRight,
     radius=\R]
 -- (\angleIRight:\Rs)
 arc[start angle=\angleIRight,
     end angle=\angleILeft,
     radius=\Rs]
 -- cycle;

\draw[thick] (0,0) circle (\R);
\node at (-3.15,3.15) {$\mathbb T$};

\draw[dashed,gray!75] (0,0)--(\R,0);
\fill (0,0) circle (1.4pt);
\node[below left] at (0,0) {$0$};

\draw[densely dashed,gray!80]
 ({-\angleG}:\Rl)
 arc[start angle={-\angleG},
     end angle=\angleG,
     radius=\Rl];

\draw[densely dashed,gray!80]
 ({-\angleG}:\Rs)
 arc[start angle={-\angleG},
     end angle=\angleG,
     radius=\Rs];

\node[gray!80!black] at (-42:\Rl) {$\rho_\ell$};
\node[gray!80!black] at (-42:\Rs) {$r_s$};

\draw[blue!75!black,line width=3pt]
 ({-\angleG}:\R)
 arc[start angle={-\angleG},
     end angle=\angleG,
     radius=\R];

\draw[orange!90!black,line width=3.5pt]
 (\angleILeft:\R)
 arc[start angle=\angleILeft,
     end angle=\angleIRight,
     radius=\R];

\node[blue!75!black] at (37:4.38) {$\Gamma_\ell$};
\node[orange!90!black] at (13:4.38) {$I$};

\fill (\R,0) circle (2pt);
\node[right] at (\R,0) {$e^{i\theta_a}$};

\coordinate (A) at (0:3.63);
\coordinate (Z) at (10:3.23);

\fill (A) circle (2.2pt);
\node[below] at (A) {$a=\rho_ae^{i\theta_a}$};

\fill (Z) circle (2.2pt);
\node[above left] at (Z) {$z$};

\node[blue!70!black] at (0:2.25)
 {$T(\Gamma_\ell)$};

\node[green!45!black] at (-30:3.55)
 {$E_\ell$};

\node[orange!80!black] (TIlabel) at (4.85,-1.25)
 {$T(I)$};

\draw[->,orange!80!black]
 (TIlabel.west)--(-3:3.28);

\node[align=left] at (-3.75,-2.9)
 {$\rho_\ell\leq r_s$\\[2pt]
  $T(I)\cup E_\ell\subseteq T(\Gamma_\ell)$};

\end{tikzpicture}
\caption{The arc $\Gamma_\ell$, centered at $e^{i\theta_a}$, contains $I$.
The point $a$ satisfies $a/|a|=e^{i\theta_a}\in I$ and $|a|\geq r_s$, while
$z\in T(I)$.}
\label{fig:low-frequency-tent-geometry}
\end{figure}
Moreover,
\[
 1-2^{-\ell-2}(1-\rho_\ell)
 \geq
 1-2^{-s}
 =
 r_s.
\]
Therefore,
\begin{align}
 \widehat\omega(\rho_\ell)
 \leq
 C_\omega^{\ell+2}
 \widehat\omega
 \left(
 1-2^{-\ell-2}(1-\rho_\ell)
 \right) \leq
 C_\omega^{\ell+2}\widehat\omega(r_s).
 \label{eq:Gamma-tail-comparison}
\end{align}
Combining direct computations with \eqref{eq:Gamma-tail-comparison}, we obtain
\begin{equation}\label{eq:Gamma-tent-ratio}
 \omega(T(\Gamma_\ell))
 \leq
 2C_\omega^{\ell+2}\omega(E_\ell).
\end{equation}
Fix $z\in T(I)$. Since
$
 T(I)\cup E_\ell\subseteq T(\Gamma_\ell),
$
it follows from \eqref{eq:Gamma-tent-ratio} that
\begin{align*}
 \frac{1}{\omega(E_\ell)}
 \int_{E_\ell}|f(\zeta)|\omega(\zeta)\,dA(\zeta)
 &\leq
 \frac{\omega(T(\Gamma_\ell))}
      {\omega(E_\ell)}
 \frac{1}{\omega(T(\Gamma_\ell))}
 \int_{T(\Gamma_\ell)}
 |f(\zeta)|\omega(\zeta)\,dA(\zeta)\\
 &\leq
 2C_\omega^{\ell+2}M_\omega f(z).
\end{align*}
Then, we get for any $z\in T(I)$,
\[
 |\mathcal C_sf(a)|
 \leq
 C\sum_{\ell=0}^\infty
 \left(
 C_\omega
 2^{-2-\lceil\log_2C_\omega\rceil}
 \right)^\ell
 \mathcal{M}_\omega f(z).
\]
Since $C_\omega\leq 2^{\lceil\log_2C_\omega\rceil}$, we have $C_\omega2^{-2-\lceil\log_2C_\omega\rceil} \leq\frac{1}{4}$.
Hence
\[
 |\mathcal C_sf(a)|
 \leq
 2C \mathcal{M}_\omega f(z),
 \qquad z\in T(I).
\]
Taking the infimum over $z\in T(I)$, we complete the proof.
\end{proof}

For a nonnegative measurable function $g$ on $\mathbb D$, write
\[
 \langle g\rangle_{T(I),\omega}
 =
 \frac{1}{\omega(T(I))}
 \int_{T(I)}g(z)\omega(z)\,dA(z).
\]
We shall use the following version of the dyadic Carleson embedding theorem (see \cite[Theorem~3.1]{NTV03},
\cite[Theorem~5.8]{Tol14}, and \cite[Lemma~10]{FW15}).

\begin{lemma}\label{lem:classical-car-emb}
Let $(\gamma_I)_{I\in\mathcal Q}$ be a nonnegative sequence such that
\[
 \sum_{\substack{I\in\mathcal Q:I\subseteq K}}\gamma_I
 \leq C_1\omega(T(K)),
 \qquad K\in\mathcal Q.
\]
Then, for every nonnegative $g\in L_\omega^2$, 
\[
 \sum_{I\in\mathcal Q}
 \gamma_I|\langle g\rangle_{T(I),\omega}|^2
 \leq
 C C_1\|g\|_{L_\omega^2}^2.
\]
Here, $C>0$ is an absolute constant.
\end{lemma}

Using the Carleosn condition \eqref{eq:shifted-Carleson-condition}, we have
\begin{equation}\label{eq:shifted-Carleson-condition-new}
 \sup_{\substack{k\geq1:I\in\mathcal Q_k}}\frac{1}{\omega(T(I))} \sum_{j\geq\max\{k,N\}+M} \sum_{\substack{J\in\mathcal Q_j:J\subseteq I}}|b_J|^2<\infty.
\end{equation}

\begin{proof}[Proof of Lemma \ref{lem:car-emb-baby}]
For $I\in\mathcal Q_k$ with $k\geq N$, set
\[
 \beta_I= \sum_{\substack{J\in\mathcal Q_{k+M}:J\subseteq I}}|b_J|^2,
\]
and set $\beta_I=0$ when $I\in\mathcal Q_k$ with $k<N$. Let $K\in\mathcal Q_q, q\geq1$. Since $I\subseteq K$, by condition \eqref{eq:shifted-Carleson-condition-new}, we have
\begin{align}
 \sum_{\substack{I\in\mathcal Q:I\subseteq K}}\beta_I
=
 \sum_{j\geq\max\{q,N\}+M}
 \sum_{\substack{J\in\mathcal Q_j:J\subseteq K}}
 |b_J|^2 \leq
 \norm{b}_{\mathrm C;N,M}^2\omega(T(K)).
 \label{eq:carleson-822}
\end{align}
For $j\geq N+M, J\in\mathcal Q_j$, let $I_J\in\mathcal Q_{j-M}$ be the unique dyadic arc containing $J$. Then $e^{i\theta_J}\in I_J$ and $|a_J|=r_j\geq r_{j-M}$. By Lemma~\ref{lem:low-frequency-localization} with $s=j-M$,
\[
 |\mathcal C_{j-M}f(a_J)|
 \leq
 C\inf_{z\in T(I_J)}\mathcal M_\omega f(z)
 \leq
 C\langle\mathcal M_\omega f\rangle_{T(I_J),\omega}.
\]
Grouping the arcs $J$ according to their $M$-th dyadic ancestors and applying Lemma~\ref{lem:classical-car-emb} to
$(\beta_I)_{I\in\mathcal Q}$, we obtain
\[
\begin{aligned}
 \sum_{j\geq N+M}\sum_{J\in\mathcal Q_j}
 |b_J|^2|\mathcal C_{j-M}f(a_J)|^2&\leq
 C\sum_{k\geq N}\sum_{I\in\mathcal Q_k}
 \beta_I
 |\langle\mathcal M_\omega f\rangle_{T(I),\omega}|^2\\
 &\leq
 C\norm{b}_{\mathrm C;N,M}^2
 \|\mathcal M_\omega f\|_{L_\omega^2}^2\\
 &\leq
 C\norm{b}_{\mathrm C;N,M}^2
 \|f\|_{A_\omega^2}^2,
\end{aligned}
\]
where the last inequality follows from Lemma~\ref{thm:tent-maximal}. The constant $C=C(\omega,\chi)>0$ is independent of $N$ and $M$. This completes the whole proof.
\end{proof}

\end{document}